\documentclass{article}
\usepackage[nottoc]{tocbibind}
\usepackage{graphicx}
\usepackage{mathrsfs}
\usepackage{geometry}
\usepackage{titletoc}
\usepackage{amsmath}
\allowdisplaybreaks[4]
\usepackage{amssymb}
\usepackage{bbm}
\usepackage{bm}
\usepackage{mathtools}
\usepackage{enumerate}
\usepackage{amsthm}
\usepackage[font=small]{caption}
\usepackage{subcaption}
\usepackage{cases}
\usepackage{mleftright}
\mleftright
\usepackage
{hyperref} 
\hypersetup{
hidelinks
}

\newlength{\bibitemsep}
\newlength{\bibparskip}
\let\oldthebibliography\thebibliography
\renewcommand\thebibliography[1]{%
  \oldthebibliography{#1}%
  \setlength{\parskip}{\bibitemsep}%
  \setlength{\itemsep}{\bibparskip}%
}
\usepackage[table]{xcolor}
\usepackage[all]{xy}
\usepackage{tikz}

\newcommand\cC{{\mathcal C}}

\newcommand\cF{{\mathcal F}}

\newcommand\cO{{\mathcal O}}

\newcommand\cS{{\mathcal S}}

\newcommand\cW{{\mathcal W}}

\newcommand{\Vol}{\mathrm{Vol}\,}
\newcommand{\VolL}{\mathrm{Vol}\,L}

\newcommand{\ord}{\mathrm{ord}}

\newcommand{\Val}{\mathrm{Val}}
\newcommand{\lct}{\mathrm{lct}}

\newcommand{\CC}{\mathbb {C}}

\newcommand{\NN}{{\mathbb N}}
\newcommand{\PP}{{\mathbb P}}
\newcommand{\QQ}{{\mathbb Q}}
\newcommand{\RR}{{\mathbb R}}
\newcommand{\ZZ}{{\mathbb Z}}

\DeclareMathOperator{\co}{co}

\theoremstyle{plain}
\newtheorem{theorem}{Theorem}[section]
\newtheorem{proposition}[theorem]{Proposition}
\newtheorem{prop}[theorem]{Proposition}

\newtheorem{lemma}[theorem]{Lemma}

\newtheorem{claim}[theorem]{Claim}
\newtheorem{corollary}[theorem]{Corollary}

\newtheorem{conjecture}[theorem]{Conjecture}

\newtheorem{problem}[theorem]{Problem}

\theoremstyle{definition}
\newtheorem{definition}[theorem]{Definition}
\newtheorem{remark}[theorem]{Remark}
\newtheorem{example}[theorem]{Example}

\def\K{K\"ahler }
\def\KE{K\"ahler--Einstein }
\newcommand{\beq}{\begin{equation}}
\newcommand{\eeq}{\end{equation}}
\newcommand{\bpf}{\begin{proof}}
\newcommand{\epf}{\end{proof}}
\newcommand{\baligned}{\begin{aligned}}
\newcommand{\ealigned}{\end{aligned}}
\newcommand{\bdefn}{\begin{definition}}
\newcommand{\edefn}{\end{definition}}
\newcommand{\bremark}{\begin{remark}}
\newcommand{\eremark}{\end{remark}}
\newcommand{\bconj}{\begin{conjecture}}
\newcommand{\econj}{\end{conjecture}}
\newcommand{\bcor}{\begin{corollary}}
\newcommand{\ecor}{\end{corollary}}
\newcommand{\blem}{\begin{lemma}}
\newcommand{\elem}{\end{lemma}}
\newcommand{\bclaim}{\begin{claim}}
\newcommand{\eclaim}{\end{claim}}
\newcommand{\bprob}{\begin{problem}}
\newcommand{\eprob}{\end{problem}}
\newcommand{\bprop}{\begin{proposition}}
\newcommand{\eprop}{\end{proposition}}
\newcommand{\bthm}{\begin{theorem}}
\newcommand{\ethm}{\end{theorem}}
\def\lb#1{\label{#1}}
\def\ra{\rightarrow}

\def\q{\quad}

\def\disp{\displaystyle}

\def\eps{\epsilon}

\font\smlsev=cmr7

\def\fin{\operatorname{fin}}
\def\Div{\operatorname{div}}
\def\Val{\operatorname{Val}}

\newcommand{\ValX}{\Val_X}
\newcommand{\ValXfin}{{\Val}_X^{\fin}}
\newcommand{\ValXdiv}{{\Val}_X^{\Div}}

\title{
Asymptotics of discrete Okounkov bodies and thresholds}
\author{Chenzi Jin, Yanir A. Rubinstein, Gang Tian}
\date{January 25, 2025}

\begin{document}

\maketitle

\begin{abstract}
This article initiates the study of discrete Okounkov bodies
and higher-dimensional Weierstrass gap phenomena,
with applications to asymptotic analysis
of stability and global log canonical thresholds. 

\end{abstract}


\section{Introduction}

This article has three main goals. First, to initiate the study of {\it discrete Okounkov bodies}
and their asymptotic behavior. Second, to 
study asymptotics of algebraic geometric invariants that are fundamental in the study of K-stability and 
canonical \K metrics, generally referred to as {\it stability  thresholds} and {\it global
log canonical thresholds}.
The second goal is achieved partly as an application of the first.
Third, to point out that, surprisingly, 
estimates on thresholds follow from a new
higher-dimensional generalization of classical Weierstrass gap theory.

Towards these goals we introduce several new ideas to the study of asymptotic
invariants in algebraic geometry: volume quantiles from statistics to study the big cone, joint large $(k,m(k))$-limits 
using Grassmannians in linear systems to unify the study of 
global
log canonical thresholds and stability thresholds, 
a notion of collapsing Okounkov bodies inspired by Riemannian convergence theory,
and discrete versions of tangent cones and blow-up estimates for collapsing discrete Okounkov bodies.
This also leads to new valuative criteria for stability on the big cone.

\subsection{Asymptotics of stability thresholds}
\lb{DeltaSubSec}
Stability thresholds provide sufficient and/or necessary conditions  
for K-stability and Chow stability. They have been studied extensively in the past
decade. These invariants can also be characterized analytically as {\it coercivity thresholds} for the corresponding energy functionals whose critical points are K\"ahler--Einstein or balanced metrics (or their
logarithmic generalizations).
Despite the absolute centrality of these invariants to moduli problems in algebraic geometry
and to PDE arising in complex geometry, it seems very little is known about their asymptotic behavior in $k$. Here, $k$ is the power of the line bundle, with $1/k$ playing the r\^ole of
Planck's constant $\hbar$ in geometric quantization.

Let $L$ be a big line bundle on a projective variety $X$ with
klt singularities.
Stability thresholds, known as $\delta_k$-invariants 
(or still as 
basis log canonical thresholds), were introduced by Fujita--Odaka in 2016 to measure the complex singularity exponent of the ``most singular" basis-type divisor in $|kL|$ \cite{FO18}. By a theorem of Rubinstein--Tian--Zhang these invariants coincide with the analytic coercivity thresholds of the quantized
Ding functional and characterize existence of $k$-balanced metrics
\cite{RTZ21}.

According to a theorem of Blum--Jonsson the limit in $k$ of these invariants exists: it is the celebrated Fujita--Odaka $\delta$-invariant with 
the Fujita--Li valuative criterion 
$\delta>1$ characterizing uniform
K/Ding-stability \cite{Li17,F19,BX19},\cite[Theorem 0.3]{FO18},\cite[Theorems B, 4.4]{BJ20},
\cite{BermanBJ} (see
also \cite{Berman}). 
Beyond the foundational results of Fujita--Odaka and Blum--Jonsson, it seems almost nothing is known about 
the asymptotic behavior of these invariants in~$k$. 
In addition, the work of Blum--Liu has shown that
asymptotics of $\delta_k$-invariants is crucial for results on lower-semi-continuity
of the $\delta$-invariant, i.e., openness of K-(semi)stability \cite{BL18,BL22}.

\bprob
\lb{AsympProb}
Determine the asymptotics of $\delta_k$ in $k$.
\eprob

Recently, two of us solved Problem \ref{AsympProb}
in the toric and ample case \cite{JR2}. In this article, using
 completely different methods, we give the first general result on 
Problem \ref{AsympProb}, i.e.,
for an arbitrary big line bundle on a general variety.
It seems new already for surfaces.

\bthm
\lb{MainDeltaThm}
Let $X$ be a normal
projective variety
with  klt singularities, and $L$ a big line bundle on $X$.
If $\delta$ is computed by a divisorial valuation (Definition \ref{computeddivisorialDef}),
$$
    \delta_k=\delta+O(1/k).
$$
\ethm

By results in the toric case \cite[Corollary 2.12]{JR2}, 
the rate $O(1/k)$ in Theorem \ref{MainDeltaThm} is optimal. 
The condition ``$\delta$ is computed by a divisorial valuation" 
(Definition \ref{computeddivisorialDef})
is a natural
one in the literature (e.g., the Fujita--Li valuative criterion for 
$\delta$ is phrased only in terms
of divisorial valuations \cite{Li17,F19,BX19},\cite[Theorem 2.9]{LXZ22}), and there
are no known examples where it does not hold.
By a deep theorem of Liu--Xu--Zhuang it holds in rather great generality,
when $L=-K_X-D$ is ample, $(X,D)$ is klt, and $\delta<\frac{n+1}{n}$ \cite[Theorem 1.2]{LXZ22}.
In fact, the proof of Theorem \ref{MainDeltaThm} only requires this condition
for the upper bound: see Theorem \ref{DeltatauThm} for the stronger statement.

Theorem \ref{MainDeltaThm} is related to a seminal result of Blum--Jonsson who
showed that $\delta$ is always computed by a valuation when $L$ is ample \cite[Theorem E]{BJ20}.
A key step in their proof is to modify the $v$-graded linear series on $L$ 
(Definition
\ref{GradedLinearSerierDef}) by certain multiplier ideals in such a way that it can be `lifted' to a
birational model over $X$. For this modified series they obtain $O(1/k)$
estimates for a twisted expected vanishing order \cite[Theorem 5.3]{BJ20}, and, ultimatley, through 
a tour de force involving deep facts on multiplier ideal sheaves they are
able to obtain the necessary compactness in the space of valuations and guarantee a
minimizing valuation exists. Roughly, Theorem \ref{MainDeltaThm} is a sort of converse
(in the more general big case but under the more restrictive divisorial assumption),
starting by assuming that a minimizing valuation exists and showing the optimal
$O(1/k)$ asymptotics for the expected vanishing order 
on~the~{\it actual}~$v$-graded~linear~series. 

\subsection{Asymptotics of global log canonical thresholds}
\lb{AlphaSubSec}
Our approach to Problem \ref{AsympProb} 
allows us to tackle
another central open problem, concerning another classical family of
invariants, the $k$-th {\it global log canonical 
thresholds} 
capturing the ``most singular" divisor
in $|kL|$,
studied by Shokurov and Cheltsov in the algebraic geometry literature starting in the 1990's.
By a theorem of Demailly, these invariants coincide with 
Tian's analytic $\alpha_k$-invariants introduced in 1988 in the context
of the \KE problem \cite{Tian87,Tian90,Tian90-2,CS08}.

\bprob
\lb{AsympAlphaProb} 
Determine the asymptotics of $\alpha_k$ in $k$.
\eprob

Problem \ref{AsympAlphaProb} 
can be considered as a refinement of 
Tian's Stabilization Problem posed in 1988 asking whether the $\alpha_k$
become eventually constant in $k$ when $L=-K_X$ \cite{Tian90}.
It is also very loosely related to the
ACC conjecture in algebraic geometry 
that stipulates (roughly) that the set of all $\alpha_k$ on varieties of
a fixed dimension contains no infinite strictly increasing sequence
\cite[Remark 3.5]{DK01},
\cite{DeFernexMustata,DeFernexEinMustata}.
Finally, Demailly--Koll\'ar studied a {\it local}
version of $\alpha_k$ on $\CC^n$ for which they obtained $O(1/k)$
asymptotics,
$$
\big|\lct_x (f)-\lct_x (f_k)\big|\le \frac n{k+1},
$$
with $k$ being the Taylor series truncation instead of the power of the line bundle
\cite[(0.2.6), Theorem 2.9]{DK01}. This estimate
plays a key r\^ole in establishing their celebrated semi-continuity of complex singularity exponents,
that also implies 
$\inf_k\alpha_k=\alpha$
for the actual $\alpha_k$-invariants \cite[Theorem A.3]{CS08}.

Surprisingly, there are still
very few results on Problem \ref{AsympAlphaProb}.
According to a theorem of Demailly and Shi
the limit in $k$ exists \cite[Theorem A.3]{CS08}, \cite[Theorem 2.2]{Shi10}, and equals Tian's $\alpha$-invariant
(with $\alpha>\frac n{n+1}$ being a classical sufficient condition
for K-stability \cite{Tian87}; we generalize this in three different ways below). 
In the non-exceptional log
Fano setting (i.e., $\alpha\le 1$) a celebrated theorem of Birkar states that there is some $\ell\in\NN$ so that $\alpha_{k\ell}=\alpha$ for all $k\in\NN$ \cite[Theorem 1.7]{Birkar2022}.
The toric Fano case is the only setting where the stabilization problem is completely solved \cite{JR1}
(in addition, the case of smooth del Pezzo surfaces
can be deduced from \cite{Cheltsov08}). 
A beautiful and difficult theorem of Blum--Jonsson is the only general result on Problem \ref{AsympAlphaProb}, proved using sophisticated arguments involving multiplier ideal sheaves \cite[Corollary 5.2]{BJ20}:

\bthm
\lb{BJAlphaThm}
{\rm (Blum--Jonsson)}
Let $X$ be a normal
projective variety
with klt singularities, and $L$ an ample line bundle on $X$.
Then
$
    \alpha_k=\alpha+O(1/k).
$
\ethm

Our second main result addresses Problem 
\ref{AsympAlphaProb} in the general big case.

\bthm
\lb{MainAlphaThm}
Let $X$ be a normal
projective variety of dimension $n$ with  klt singularities, and $L$ a big line bundle on $X$.
If $\alpha$ is computed by a divisorial valuation (Definition \ref{computeddivisorialDef}),
$$
\alpha_k=\alpha+O(1/k^{1/n}).
$$
If for some $\ell\in\NN$, $\alpha=\alpha_\ell$, then $\alpha_k=\alpha+O(1/k)$.
\ethm

The convergence rate of Theorem \ref{MainAlphaThm} is sometimes
slower than the one in the ample setting. We do not know if this
can be improved. The proof involves delicate a priori estimates (Theorem \ref{maxp1Thm})
inspired by blow-up analysis of tangent cones in Riemannian convergence
theory, and adapted to collapsing {\it discrete} Okounkov bodies, introduced in 
\S\ref{DiscreteOkBSubsec}. In a precise sense, these estimates generalize 
Weierstrass gap theory from the setting of curves, and this is explained in 
\S\ref{DiscreteOkBSubsec}--\ref{WeierSubSec}.
The condition ``$\alpha$ is computed by a divisorial valuation"
is once again natural. It is implied by the existence of an $\ell\in\NN$
such that $\alpha=\alpha_\ell$ (Lemma \ref{birlem}, Corollary \ref{BirkarCor}), and this latter
condition holds by Birkar's theorem in even greater generality than for $\delta$:
whenever $L=-K_X-D$ is big and nef, $(X,D)$ is klt, and $\alpha\le1$ (i.e., $(X,D)$ is non-exceptional in the sense of Shokurov;
according to H\"oring--\'Smiech any smooth (or even canonical
singularities) Fano $X$ with $D=0$ should fall into this class (and this would
follow from the effective nonvanishing conjecture),
and in dimension up to 5 it is also the case \cite{HS23}). 
It is interesting
to note that $\alpha\le1$ is a weaker assumption than $\delta<\frac{n+1}n$
but that is not necessarily relevant to whether $\delta$
or $\alpha$ are computed by divisors. E.g., for all smooth del Pezzo
surfaces $\alpha\le1$, while $\delta\ge 3/2$ for cubic surfaces
with an Eckardt point, generic 7-point blow-ups of $\PP^2$,
and 8-point blow-ups of $\PP^2$ \cite[Table 2.1]{Cheltsov-book}.
Regardless, both $\alpha$ and $\delta$ are always computed
by a divisor in this dimension \cite{Cheltsov08,Cheltsov-book}.

\subsection{A unified approach to thresholds:
volume quantile, Grassmannians, and joint asymptotics}
\lb{TauSubSec}

The {\it volume} function on the big cone is a very central notion in algebraic geometry.
One of the innovations of this article is the idea of introducing {\it volume
quantiles}. This is done in two completely different ways: using Grassmannians
in spaces of sections, and using valuations. Theorem \ref{DeltatauThm} shows
that these coincide in an appropriate {\it joint limit}, i.e., the former
is a quantization of the latter. A subtle point compared to many quantization
results in that here the limit is not simply a so-called large $k$ limit,
but rather a joint one in $(k,m)$ with $m=m(k)$ the Grassmannian parameter.
It follows from our analysis that the notion of a compatible $m$-basis is precisely
the quantization of the volume quantile.

As an important application, this allows us 
to unify the study of $\alpha_k$ and $\delta_k$. 
To that end, a family of invariants $\delta_{k,m}$
that include both as extreme cases is introduced. Theorem 
\ref{MainDeltaThm}
is then obtained as a special case of Theorem \ref{DeltatauThm} below.
Theorem \ref{MainAlphaThm} can be considered as a limiting degenerate
case of 
Theorem \ref{DeltatauThm}. As such it requires careful additional treatment and does not follow from Theorem \ref{DeltatauThm}.

\begin{definition}
\lb{dkNLdeltakmDef}
Let $d_k:=\dim H^0\left(X,kL\right)$, and
$\NN(L):=\{k\in\NN\,:\, d_k>0\}$.
    For $k\in\NN(L)$ and $m\in\{1,\ldots,d_k\}$, let (see Definition
\ref{lctDef}) 
    \begin{equation}\label{delta def}
        \delta_{k,m}:=\inf_{\substack{\left\{s_\ell\right\}_{\ell=1}^m\subset H^0\left(X,kL\right)
\\\mbox{\smlsev linearly independent}}}km\,\lct\left(\sum_{\ell=1}^m\left(s_\ell\right)\right).
    \end{equation}
\end{definition}

We call the $\delta_{k,m}$ {\it Grassmannian stability thresholds}.
They are inspired by the invariants $\alpha_{k,m}$ 
\cite{Tian90-2,JR1} but are quite different. By definition, $\delta_{k,1}=\alpha_k$
while $\delta_{k,d_k}=\delta_k$ (Definition \ref{old notion}). 
We study their 
{\it joint asymptotics} in $(k,m(k)=m_k)$. This allows to study
the asymptotics of $\alpha_k$ and $\delta_k$ simultaneously.
It also yields a 1-parameter family of new invariants 
whose endpoints are $\alpha$ and $\delta$, depending on the 
volume of the linear subsystem, with parameter being the {\it volume quantile} (Definition \ref{ccdfDef})
\beq\lb{tauEq}
\tau:=\lim_{k\in\NN(L)} \frac{m(k)}{d_k}\in[0,1].
\eeq
Importantly, this idea allows us to realize
$\alpha_k$ as invariants associated to a {\it collapsing}
family of Okounkov bodies. Finally, it motivates 
the study of discrete Okounkov bodies, as~described~in~\S\ref{DiscreteOkBSubsec}.

The following generalizes Theorem \ref{MainDeltaThm}.

\bthm
\lb{DeltatauThm}
Let $X$ be a normal
projective variety
with klt singularities, and  $L$ a big line bundle on $X$.
For any sequence $\{m_k\}_{k\in\NN(L)}$,
with     $
         \lim_{k\to\infty}{m_k}/{d_k}=\tau\in[0,1],
     $
the limit
$$
    \bm{\delta}_\tau:=\lim_{k\to\infty}\delta_{k,m_k}
$$
exists and is independent of the choice of $\{m_k\}_{k\in\NN(L)}$. In particular, $\bm{\delta}_0=\alpha$, $\bm{\delta}_1=\delta$.

\noindent
Moreover, there exists $C>0$ independent of $k$ such that
$$
    \delta_{k,m_k}\geq\left(1-\frac{C}{k}\right)\min\left\{\frac{m_k}{\tau d_k},1\right\}\bm{\delta}_\tau.
$$
If $\tau>0$ and $\bm{\delta}_\tau$ is computed by a divisorial valuation (Definition \ref{computeddivisorialDef}), then
$$
    \delta_{k,m_k}\leq\left(1+\frac{C}{k}\right)\max\left\{\frac{m_k}{\tau d_k},1\right\}\bm{\delta}_\tau.
$$
\ethm

One idea in the proof of Theorem \ref{DeltatauThm}
is to construct a {\it rooftop Okounkov body} as well
as its discrete analogues above the ones
associated to $(X,L)$ and a flag associated to the divisorial
valuation that computes $\bm{\delta}_\tau$. This generalizes
a well-known construction of Donaldson \cite{Don} from the toric setting,
that was generalized by Witt Nystr\"om \cite{WN12}.
For toric $X$, we previously discretized Donaldson's construction
to obtain asymptotics of $\delta_k$ using Ehrhart theory \cite{JR2}. Here,
Ehrhart theory does not apply so instead we prove certain {\it uniform}
estimates on lattice point counts. Another idea is to use notions from probability theory
to control the convergence rates of the $\delta_{k,m_k}$ by interpreting
the associated invariants $S_{k,m_k}$ as barycenters of tail distributions
arising from the discrete Okounkov body $\Delta_k$ via discrete and continuous ccdfs
and volume
quantiles. This makes contact with the beautiful construction of Bouckson--Chen 
of 1-parameter filtered families of Okounokv bodies \cite{BC11} (in some sense, it is
its quantization). A third idea is to
introduce the notion of {\it idealized expected vanishing order} and {\it
idealized jumping function} to get useful estimates~on~the~$S_{k,m_k}$.

It is natural to ask
whether the
new invariants $\bm{\delta}_\tau$ are related to K-stability.
A model result, going back to Tian in 1987, and generalized over the years
by Odaka--Sano, Fujita, and Blum--Jonsson 
\cite[Theorem 2.1]{Tian87},\cite[Theorem 1.4]{OS},\cite[Proposition 2.1]{Fuj19},~\cite[Theorem A]{BJ20}
is the following.

\bthm
\lb{TianOdakaSanoThm} {\rm (Tian, Odaka--Sano)}
Let $X$ be a normal
projective variety
with klt singularities and $-K_X$ ample.
Then
$\alpha \stackrel{(\ge)}{>} \frac n{n+1}
$
implies K-(semi)stability.
\ethm

Despite multiple generalizations, it has been an open problem whether
Theorem \ref{TianOdakaSanoThm} extends to big classes since all proofs
use ampleness crucially.
In the spirit of \S\ref{DeltaSubSec}--\S\ref{AlphaSubSec} it is natural to ask:

\bprob
\lb{TianOdakaSanoProb}
Does Tian--Odaka--Sano's criterion generalize to the big cone?
\eprob

A major motivation for Problem \ref{TianOdakaSanoProb} is that 
 Theorem \ref{TianOdakaSanoThm}
is one of the main tools for constructing smooth and singular \KE metrics.
Thus, a positive
solution of Problem \ref{TianOdakaSanoProb} would yield 
many new
K-(semi)stable varieties with $-K_X$ big. 
Via the work of Darvas--Zhang \cite{DarvasZhang} this can
be used to construct many new singular twisted \KE metrics on big classes. 

In a sequel \cite{JRT2} we show how our work here, together with new estimates
on {\it sub-barycenters} in convex geometry together with {\it Nakayama's minimal
vanishing order invariant} can be used to 
simultaneously resolve Problem \ref{TianOdakaSanoProb}
and, in the spirit of \S\ref{TauSubSec}, show it extends to all volume
quantiles $\tau\in[0,1]$ and to all big $L$ (not just $L=-K_X$). 
We state the result next in order to showcase this application
of Theorem \ref{DeltatauThm} and the circle of ideas involved in its proof.
A key ingredient, in addition to Theorem \ref{DeltatauThm}, is the realization, inspired by convex geometry, that  
while $\bm{\delta}_\tau$ are defined on the big cone, they
must be normalized in a delicate manner outside the nef cone in order to encode
K-stability. 

Note that in the next definition $\bm{\tilde\delta}_1=\delta$ always, while on the nef cone, 
$\bm{\tilde\delta}_\tau=\bm{\delta}_\tau$ and $\tilde\alpha=\alpha$.

\bdefn
\lb{tildedeltatauDef}
Let  $\sigma:\ValXdiv\ra\RR_+$ be 
the minimal vanishing order (Definition \ref{S0Def}).~Set~(see~\S\ref{prelim})
$$\bm{\tilde\delta}_\tau:=
\begin{cases}
            \displaystyle\inf_{v\in \ValXdiv} \frac A{\cS_0+\sigma/n}=:\tilde\alpha,&\tau=0,\\
            \\
            \displaystyle\inf_{v\in\ValXdiv} \frac A{\cS_\tau+\Big[\frac1{\tau}\left(1-\left(1-\tau\right)^\frac{n+1}{n}\right)-1\Big]\sigma},&\tau\in\left(0,1\right].
\end{cases}
$$
\edefn

These new invariants yield a 
generalization~of~Theorem~\ref{TianOdakaSanoThm}~resolving~Problem~\ref{TianOdakaSanoProb}~affirmatively \cite{JRT2}.

\begin{theorem}\lb{GenTianThm}
Let $X$ be a normal
projective variety
with klt singularities, and  $L$ a big line bundle on $X$.
Then
    $$
        \bm{\tilde\delta}_\tau\leq\begin{cases}
            \frac{n}{n+1}\delta,&\tau=0;\\ 
            \tau\left(1-\left(1-\tau\right)^\frac{n+1}{n}\right)^{-1}\delta,&0<\tau\leq1.
        \end{cases}
    $$
    In particular, $\delta\stackrel{(\ge)}{>}1$  if
$\disp\bm{\tilde\delta}_\tau \stackrel{(\ge)}{>} \frac n{n+1-\root n \of \tau}$
for some $\tau\in[0,1].$
\end{theorem}



\subsection{Discrete Okounkov bodies and their gap asymptotics}
\lb{DiscreteOkBSubsec}

To prove our main results on asymptotics of thresholds (Theorems \ref{MainAlphaThm} and \ref{DeltatauThm})
requires in an essential way to introduce the concept of discrete Okounkov bodies
and study their asymptotics.
As we explain next, these provide the natural generalization of Ehrhart
theory and monomial sections from the toric setting. In \S\ref{WeierSubSec} we explain how
they generalize Weierstrass gap theory to any dimension.

Indeed, results in the toric case on the asymptotics of $\delta_k$ \cite{JR2} pointed
out a fascinating link
between Ehrhart theory of asymptotics of lattice points
in lattice polytopes and K-stability theory. 
Such asymptotics have been lacking for the vastly
more general Okounkov bodies.
This is one of the inspirations for this article, aiming 
to develop quantitative lattice point
estimates in Okounkov bodies. Aside from their crucial
application to asymptotics of thresholds, we believe some
of the techniques developed in this article will be 
important for further study of quantitative asymptotics in 
the theory of Okounkov bodies.

First, recall the definition of the Okounkov body \cite[\S1]{LM09},
\cite{KK}
associated
to a big line bundle $L\ra X$ over an $n$-dimensional irreducible variety and a choice of admissible flag
    $$
        Y_\bullet:\ X=Y_0\supset Y_1\supset\cdots\supset Y_n=\{p\},
    $$
    on $X$ (i.e.,
each $Y_i$ is an irreducible subvariety of codimension $i$ regular at $p$).
Choose a holomorphic coordinate chart $f:U\to V\subset\CC^n$ around $p$ such that
    $
        Y_i\cap U=f^{-1}\left(\left\{x\in V\,:\,x_j=0\text{ for }1\leq j\leq i\right\}\right).
    $
    By shrinking $U$ we may assume $L|_U$ is trivial. Fix a trivialization on $U$. Let $\nu$ denote the lexicographic valuation, i.e., for a non-zero power series $s=\sum_Ia_Ix^I$,
    $
        \nu\left(s\right)=\min\left\{I\,:\,a_I\neq0\right\},
    $
    where we take the lexicographic order on the space of multi-indices.
    Let $V_\bullet$ be a graded linear series belonging to $L$. The Okounkov body is the compact convex set
    $$
        \Delta:=\overline{\co\left(\bigcup_{k\in\NN}{\textstyle\frac1k}\,\nu\left(V_k\setminus\left\{0\right\}\right)\right)}\subset\RR^n,
    $$
    where $\co$ denotes the convex hull.

In this article we define and initiate the study of the discrete analogues of the Okounkov body:
\bdefn\lb{DiscreteOkBodyDef}
    For $k\in\NN$, the {\it $k$-th discrete Okounkov body}\ \ is
    \begin{equation}\label{DeltakEq}
        \Delta_k:={\textstyle\frac1k}\,\nu\left(V_k\setminus\left\{0\right\}\right)\subset\ZZ^n/k.
    \end{equation}
\edefn

A hallmark of Ehrhart theory of a convex lattice polytope $P$ is that 
the {\it Ehrhart function} $E(k):=\#(P\cap \ZZ^n/k)$ is a polynomial
of degree $n$, with its coefficients intensely studied (though still
quite mysterious) \cite{Gruber}. For a general convex body $K$ the Ehrhart function
is rather mysterious beyond the leading term $|K|k^n$. 
Rather recently it was determined for irrational planar
triangles \cite{CGLS}. An old, but quite different problem, known as the generalized
Gauss Circle Problem is to determine the sub-leading asymptotics for a {\it real}
dilation parameter $r>0$. Namely, to determine the asymptotics of $\big|\#(rK\cap \ZZ^n)-|K|r^n\big|$.
This is still wide open even for $K$ the unit disc in $\RR^2$ (see, e.g., \cite{Huxley}).

These classical problems are hallmarks of Discrete Geometry and Geometry of Numbers.
In the context of Okounkov bodies we pose a
somewhat different, but related, set of problems, inspired by Problems \ref{AsympProb} and \ref{AsympAlphaProb}. Namely, we are interested to quantify to what
extent do the $\Delta_k$ approximate $\Delta$. In this setting the asymptotics of
$\#(\Delta_k)=d_k$ are completely known (at least when $L$ is ample): this is the Hilbert polynomial of $L$
(recall Definition \ref{dkNLdeltakmDef}). Thus, instead
of comparing $\Delta$ to the $\Delta_k$ the natural problem is to consider
{\it gaps} in $\Delta\cap \ZZ^n/k$. To make this precise, recall the 
notion of a valuative point of $\Delta\cap\QQ^n$, i.e.,
a point in the image of $\nu_k$ for {\it some} $k\in \NN(L)$
\cite[Definition 1.2.44]{KL}. We extend this to discrete~Okounkov~bodies:

\bdefn
\lb{kgapsDefn}
A point in $\Delta\cap \ZZ^n/k$ is {\it $k$-valuative} if it is in $\Delta_k$,
and is a {\it $k$-gap point} otherwise. The {\it $k$-gap set of $\Delta$} is
$(\Delta\cap \ZZ^n/k)\setminus\Delta_k$.
\edefn

\bprob
\lb{DkdkProb} {\rm (Generalized Weierstrass Problem)}
Determine the asymptotics~of~$\#(\Delta\cap \ZZ^n/k)\!-\!\#(\Delta_k)$.
\eprob

In the smooth toric setting the Ehrhart polynomial coincides with the Hilbert polynomial. 
Equivalently, every $k$-lattice point corresponds to a monomial section and there are no gaps. 
Thus, Problem \ref{DkdkProb} concerns how 
far can an Okounkov body be from the ``completely integrable" toric case,
in an asymptotic sense. 
Equivalently, how far is the discrete Okounkov body $\Delta_k$
from the {\it idealized discrete Okounkov body} $\Delta\cap \ZZ^n/k$.

In fact, we are particularly interested in two refinements of Problem \ref{DkdkProb}.
The first is a version for ``quantiles" and the second is a microlocal version, both 
involving a second real parameter $t$.

To introduce these recall the construction of a 1-parameter family of Okounkov bodies $\Delta_v^t$
(Definition \ref{G def})
using graded linear subseries of $L$, $V_{v,\bullet}^t$, associated
to a valuation on $X$ (Definition \ref{GradedLinearSerierDef}).
The guiding example to keep in mind is the case $v$ is the order of vanishing
along $Y_1$, the divisor in the flag, but everything makes sense in general.
In this case, $\Delta_v^t$ is simply the super-level set of $p_1:\RR^n\ra\RR$,
the projection to the first coordinate (Lemma \ref{sliceDivOkBodyLem}), see
Figure \ref{FigDkdk}.

\begin{figure}[ht]
    \centering
    \begin{tikzpicture}
        \draw[->](-.5,0)--(5,0);
        \draw[->](0,-.5)--(0,3);
        \draw(0,0)node[below left]{$0$};
        \draw(0,2)--(1.5,3)node[above left]{$\Delta^0\!=\!\Delta$}--(3,2)--(4.5,1)node[midway, above right]{$\Delta^t$}--(4,.5)--(1,0.5)--(0,1)--cycle;
        \filldraw(0,1)circle(2pt);
        \filldraw[fill=white](0,2)circle(2pt);
        \filldraw(1.5,1)circle(2pt);
        \filldraw[fill=white](1.5,2)circle(2pt);
        \filldraw[fill=white](1.5,3)circle(2pt);
        \filldraw(3,1)circle(2pt);
        \filldraw(3,2)circle(2pt);
        \filldraw[fill=white](4.5,1)circle(2pt);
        \draw[dashed](2.2,0)node[below]{\scriptsize$t$}--(2.2,2.55);
        \draw[dashed](3,0)node[below]{\scriptsize$\max_{\Delta_k}\! p_1$}--(3,2.1);
        \draw[dashed](4.5,0)node[below]{\scriptsize$\max_\Delta\! p_1$}--(4.5,1);
    \end{tikzpicture}
    \caption{A 1-parameter family of Okounkov bodies and a discrete Okounkov body associated to a
    divisorial valuation with $\#(\Delta_k)=d_k=4$ and $\#(\Delta\cap \ZZ^n/k)=D_k=8$.
    }
\lb{FigDkdk}
\end{figure}
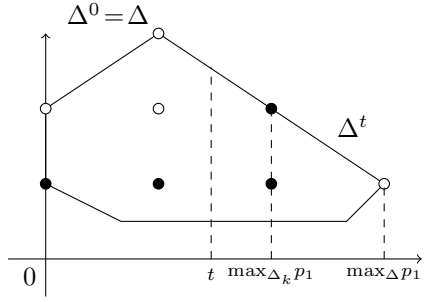

One can pose Problem \ref{DkdkProb} as a function of $t$. However,
it turns out to be preferable to pose the problem in terms of the {\it volume quantile},
a generalized inverse to the function $t\mapsto |\Delta_v^t|/|\Delta|$,
denoted $Q_v$, defined on $\tau\in[0,1]$, introduced in Definition \ref{ccdfDef}.

\bprob
\lb{DkdkQuantileProb}
Determine the asymptotics of $\#\big(\Delta
^{Q_v(\tau)}\cap \ZZ^n/k\big)-\#\big((\Delta^{Q_v(\tau)})_k\big)$, for
$\tau\in(0,1]$
\eprob

In fact, in \S\ref{lb section}--\S\ref{ub section} we solve a more difficult problem by determining the asymptotics
not just of the lattice point difference for each $\tau$ but also that of the first
moments of the corresponding empirical measures. The associated barycenters
are then the invariants $\cS_\tau$ (Proposition \ref{S_tau def})
 who are used to show the well-definedness
of the invariants $\bm{\delta}_\tau$ (Corollary \ref{deltadoubleasymCor}).
The barycenters of these ``partial Okounkov bodies" $\Delta
^{Q_v(\tau)}$ can be thought of as
associated to tail distributions (Definition \ref{ccdfDef}) of the Lebesgue measure with respect to the
volume quantile. On the other hand, the barycenters of the empirical measures
associated to $(\Delta^{Q_v(\tau)})_k$ can be understood simultaneously 
via quantum tail distribution (Definition \ref{iota defn})
and geometrically via 
the counting function on an associated {\it rooftop Okounkov body} (proof 
of Theorem \ref{SdoubleasymCor}).
This idea is inspired by a similar construction we employed in the toric
case \cite{JR2}.

Finally, we pose a microlocal version of Problem \ref{DkdkQuantileProb}, related to the 
collapsing limit $\tau\ra 0$. 
Set
\beq\lb{p1plusbodyEq}
\#\big(\Delta^{\max_{\Delta_k} p_1\pm}\cap \ZZ^n/k\big):=
\#\big(\Delta^{\max_{\Delta_k} p_1\pm\eps}\cap \ZZ^n/k\big),
\q \hbox{\rm  for any $\eps\in(0,1/k)$}.
\eeq

\bprob
\lb{DkdkMicrolocalProb}
Determine the asymptotics of 
$\#\big(\Delta^{\max_{\Delta_k} p_1\pm}\cap \ZZ^n/k\big)$.
\eprob

Note that as $k\ra\infty$, the bodies $\Delta^{\max_{\Delta_k}p_1-}$ collapse
to a lower-dimensional convex set. This corresponds to a collapse of the 
model given by the corresponding Kodaira maps.

The following key estimate, shows that 
Problem \ref{DkdkMicrolocalProb} is closely related to determining the asymptotics of 
$\max_{\Delta} p_1 - \max_{\Delta_k} p_1$, which is itself equivalent to 
Problem \ref{AsympAlphaProb}. 

\bthm
\lb{maxp1Thm}
For $v=\ord_{Y_1}$,
$$
\disp
0
\le
\max_{\Delta} p_1 - \max_{\Delta_k} p_1
\le 
C\Big(\#\big(\Delta^{\max_{\Delta_k} p_1+}\cap \ZZ^n/k\big)\Big)^{1/n}/k.
$$
\ethm

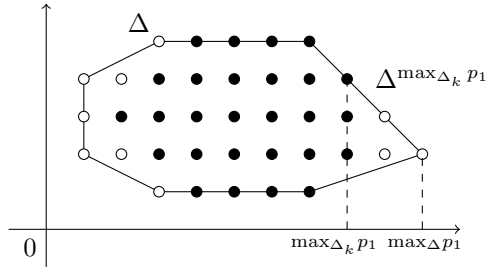
\begin{figure}[ht]
    \centering
    \begin{tikzpicture}
        \draw[->](-1,0)--(5,0);
        \draw[->](-.5,-.5)--(-.5,3);
        \draw(-.5,0)node[below left]{$0$};
        \draw(0,2)--(1,2.5)node[above left]{$\Delta$}--(3,2.5)--(4.5,1)node[midway,above right]{$\Delta^{\max_{\Delta_k} p_1}$}--(3,.5)--(1,.5)--(0,1)--cycle;
        \filldraw[fill=white](0,1)circle(2pt);
        \filldraw[fill=white](0,1.5)circle(2pt);
        \filldraw[fill=white](0,2)circle(2pt);
        \filldraw[fill=white](.5,1)circle(2pt);
        \filldraw(.5,1.5)circle(2pt);
        \filldraw[fill=white](.5,2)circle(2pt);
        \filldraw[fill=white](1,.5)circle(2pt);
        \filldraw(1,1)circle(2pt);
        \filldraw(1,1.5)circle(2pt);
        \filldraw(1,2)circle(2pt);
        \filldraw[fill=white](1,2.5)circle(2pt);
        \filldraw(1.5,.5)circle(2pt);
        \filldraw(1.5,1)circle(2pt);
        \filldraw(1.5,1.5)circle(2pt);
        \filldraw(1.5,2)circle(2pt);
        \filldraw(1.5,2.5)circle(2pt);
        \filldraw(2,.5)circle(2pt);
        \filldraw(2,1)circle(2pt);
        \filldraw(2,1.5)circle(2pt);
        \filldraw(2,2)circle(2pt);
        \filldraw(2,2.5)circle(2pt);
        \filldraw(2.5,.5)circle(2pt);
        \filldraw(2.5,1)circle(2pt);
        \filldraw(2.5,1.5)circle(2pt);
        \filldraw(2.5,2)circle(2pt);
        \filldraw(2.5,2.5)circle(2pt);
        \filldraw(3,.5)circle(2pt);
        \filldraw(3,1)circle(2pt);
        \filldraw(3,1.5)circle(2pt);
        \filldraw(3,2)circle(2pt);
        \filldraw(3,2.5)circle(2pt);
        \filldraw(3.5,1)circle(2pt);
        \filldraw(3.5,1.5)circle(2pt);
        \filldraw(3.5,2)circle(2pt);
        \filldraw[fill=white](4,1)circle(2pt);
        \filldraw[fill=white](4,1.5)circle(2pt);
        \filldraw[fill=white](4.5,1)circle(2pt);
        \draw[dashed](3.5,0)node[below]{\scriptsize$\!\!\!\!\!\!\!\max_{\Delta_k}\!p_1$}--(3.5,2);
        \draw[dashed](4.5,0)node[below]{\scriptsize$\;\max_{\Delta}\!p_1$}--(4.5,1);
    \end{tikzpicture}
    \caption{A possible Okounkov body $\Delta$ of some big line bundle over a surface with respect to some flag $Y_\bullet:Y_0\supset Y_1\supset Y_2$. The dots are the $k$-th lattice points $\Delta\cap\ZZ^2/k$, of which the solid dots are the ones corresponding to sections in $R_k$. Let $v=\ord_{Y_1}$. Then there are no solid dots whose first coordinate exceeds $\max_{\Delta_k}\!p_1$ (or $S_{k,1}(v)$
    in the notation of Definition \ref{SkmDef}; $\max_{\Delta_k}\!p_1=\cS_0(v)$ in the
    notation of Definition \ref{S0Def}).}
\end{figure}

Theorem \ref{maxp1Thm} is a crucial ingredient in the proof of Theorem \ref{MainAlphaThm}.
To see why,
note that by the proof of Theorem \ref{delta main} (in the case $m_k=1$,
see \S\ref{ProofofTwoThmSubSec})
there exists a constant $C>0$ independent of $k$
such that 
$$
\alpha=\frac C{\max_{\Delta} p_1}\le\alpha_k\le \frac C{\max_{\Delta_k} p_1}.
$$
Thus, Theorem \ref{maxp1Thm} and Lemma \ref{DkdkLem} imply  
$
\alpha_k-\alpha=O(k^{-\frac1{n}}),
$
which is the conclusion of Theorem \ref{MainAlphaThm}.

Theorem \ref{maxp1Thm} is also related to Problem \ref{DkdkProb} since
\beq\lb{TrivialLatticeCountEq}
\#\big(\Delta^{\max_{\Delta_k} p_1+}\cap \ZZ^n/k\big)\le 
\#(\Delta\cap \ZZ^n/k)-\#(\Delta_k).
\eeq
E.g., in the toric case the right hand side is zero, 
since $\alpha=\alpha_k$ for all $k\in\NN$ \cite{JR1}.
In general, the right hand side is $O(k^{n-1})$. 
A interesting problem is to 
obtain better estimates that would take into account the geometry
of $(X,L)$ and the flag, and we come back to that in \S\ref{WeierSubSec}.

One such improvement is Proposition \ref{secondconeprop} that implies the estimate
\beq\lb{DeltaimprovedasympEq}
\#\big(\Delta^{\max_{\Delta_k} p_1+}\cap \ZZ^n/k\big)
\ge C (\max_{\Delta} p_1 - \max_{\Delta_k} p_1)^{n-\iota}k^n,
\eeq
where $\iota:=\dim \Delta\cap p_1^{-1}(\max_\Delta p_1)$ is the dimension
of the ``collapsed Okounkov body." This implies
the improved blow-up estimate
$$
0
\le \max_{\Delta} p_1 - \max_{\Delta_k} p_1
\le  C/{k^{\frac1{n-\iota}}},
$$
and consequently  
\beq\lb{alphaimprovedasympEq}
\alpha_k-\alpha=O(k^{-\frac1{n-\iota}}).
\eeq

\subsection{Extending Weierstrass gaps to higher dimensions}
\lb{WeierSubSec}

As just explained in \S\ref{DiscreteOkBSubsec},
Theorem \ref{maxp1Thm} implies that gaps of discrete Okounkov bodies
control asymptotics of thresholds. 
This leads us, surprisingly, to offer a higher-dimensional
generalization of Weierstrass gap theory for curves \cite{delCentina,BN74,Hau96,Ols72,Nee84,Rau55,EH,HM}. To the best of 
our knowledge such a generalization is new.

To see this, let us first phrase the precise manner in which a Weierstrass gap sequence can
be understood in terms of a particular kind of discrete Okounkov bodies and
their gaps. In fact, this simultaneously recovers all {\it higher Weierstrass gap
sequences} in the sense of H\"urwitz \cite[\S3]{Hurwitz}:

\bprop
\lb{Weierstrass2IntroProp}
Let $C$ be a smooth curve of genus $g\geq2$ and $p$ 
a point in $C$. Let $d_1=g$ and $d_k=k(2g-2)+1-g$ for $k\ge2$.
 The $k$-Weierstrass gap sequence of $p$ is $1=N^{(k)}_1<\cdots<N^{(k)}_{d_k}\le k(2g-2)+1$,~with
$$
\{N^{(k)}_1,\ldots,
N^{(k)}_{d_k}\}=k\Delta_k+1,
$$
where $\Delta_k$ is $k$-th discrete Okounkov body of 
$K_C$ associated 
to the flag $C\supset\{p\}$.
For $k\ge2$, 
$$\max\Delta-\max\Delta_k
\le \frac{\#(\Delta\cap \ZZ^n/k)-\#(\Delta_k)}k=\frac{g}k,
$$
with equality if and only if $p$ is not a $k$-Weierstrass point.
\eprop

Proposition \ref{Weierstrass2IntroProp} is proved in \S\ref{classicalWeierSubSec}.
Figure \ref{2WeierstrassFig} depicts
the case $g=3$ and $k\in\{1,2\}$.

\begin{figure}[ht]
    \centering
    \begin{subfigure}{.3\textwidth}
        \centering
        \begin{tikzpicture}
            \draw[->](-.5,0)--(3,0);
            \filldraw(0,0)node[below]{$0$}circle(2pt);
            \filldraw(.625,0)circle(2pt);
            \filldraw(1.25,0)circle(2pt);
            \filldraw[fill=white](1.875,0)circle(2pt);
            \filldraw[fill=white](2.5,0)node[below]{$4$}circle(2pt);
        \end{tikzpicture}
    \end{subfigure}
    \begin{subfigure}{.3\textwidth}
        \centering
        \begin{tikzpicture}
            \draw[->](-.5,0)--(3,0);
            \filldraw(0,0)node[below]{$0$}circle(2pt);
            \filldraw(.625,0)circle(2pt);
            \filldraw[fill=white](1.25,0)circle(2pt);
            \filldraw(1.875,0)circle(2pt);
            \filldraw[fill=white](2.5,0)node[below]{$4$}circle(2pt);
        \end{tikzpicture}
    \end{subfigure}
    \begin{subfigure}{.3\textwidth}
        \centering
        \begin{tikzpicture}
            \draw[->](-.5,0)--(3,0);
            \filldraw(0,0)node[below]{$0$}circle(2pt);
            \filldraw(.625,0)circle(2pt);
            \filldraw[fill=white](1.25,0)circle(2pt);
            \filldraw[fill=white](1.875,0)circle(2pt);
            \filldraw(2.5,0)node[below]{$4$}circle(2pt);
        \end{tikzpicture}
    \end{subfigure}
    \begin{subfigure}{.3\textwidth}
        \centering
        \begin{tikzpicture}
            \draw[->](-.5,0)--(3,0);
            \filldraw(0,0)node[below]{$0$}circle(2pt);
            \filldraw(.3125,0)circle(2pt);
            \filldraw(.625,0)circle(2pt);
            \filldraw(.9375,0)circle(2pt);
            \filldraw(1.25,0)circle(2pt);
            \filldraw(1.5625,0)circle(2pt);
            \filldraw[fill=white](1.875,0)circle(2pt);
            \filldraw[fill=white](2.1875,0)circle(2pt);
            \filldraw[fill=white](2.5,0)node[below]{$4$}circle(2pt);
        \end{tikzpicture}
    \end{subfigure}
    \begin{subfigure}{.3\textwidth}
        \centering
        \begin{tikzpicture}
            \draw[->](-.5,0)--(3,0);
            \filldraw(0,0)node[below]{$0$}circle(2pt);
            \filldraw(.3125,0)circle(2pt);
            \filldraw(.625,0)circle(2pt);
            \filldraw(.9375,0)circle(2pt);
            \filldraw(1.25,0)circle(2pt);
            \filldraw[fill=white](1.5625,0)circle(2pt);
            \filldraw(1.875,0)circle(2pt);
            \filldraw[fill=white](2.1875,0)circle(2pt);
            \filldraw[fill=white](2.5,0)node[below]{$4$}circle(2pt);
        \end{tikzpicture}
    \end{subfigure}
    \begin{subfigure}{.3\textwidth}
        \centering
        \begin{tikzpicture}
            \draw[->](-.5,0)--(3,0);
            \filldraw(0,0)node[below]{$0$}circle(2pt);
            \filldraw(.3125,0)circle(2pt);
            \filldraw(.625,0)circle(2pt);
            \filldraw[fill=white](.9375,0)circle(2pt);
            \filldraw(1.25,0)circle(2pt);
            \filldraw(1.5625,0)circle(2pt);
            \filldraw[fill=white](1.875,0)circle(2pt);
            \filldraw[fill=white](2.1875,0)circle(2pt);
            \filldraw(2.5,0)node[below]{$4$}circle(2pt);
        \end{tikzpicture}
    \end{subfigure}\vspace{1em}
    \begin{subfigure}{.3\textwidth}
        \centering
        \begin{tikzpicture}
            \draw[->](-.5,0)--(3,0);
            \filldraw(0,0)node[below]{$0$}circle(2pt);
            \filldraw(.625,0)circle(2pt);
            \filldraw(1.25,0)circle(2pt);
            \filldraw[fill=white](1.875,0)circle(2pt);
            \filldraw[fill=white](2.5,0)node[below]{$4$}circle(2pt);
        \end{tikzpicture}
    \end{subfigure}
    \begin{subfigure}{.3\textwidth}
        \centering
        \begin{tikzpicture}
            \draw[->](-.5,0)--(3,0);
            \filldraw(0,0)node[below]{$0$}circle(2pt);
            \filldraw(.625,0)circle(2pt);
            \filldraw(1.25,0)circle(2pt);
            \filldraw[fill=white](1.875,0)circle(2pt);
            \filldraw[fill=white](2.5,0)node[below]{$4$}circle(2pt);
        \end{tikzpicture}
    \end{subfigure}
    \begin{subfigure}{.3\textwidth}
        \centering
        \begin{tikzpicture}
            \draw[->](-.5,0)--(3,0);
            \filldraw(0,0)node[below]{$0$}circle(2pt);
            \filldraw(.625,0)circle(2pt);
            \filldraw(1.25,0)circle(2pt);
            \filldraw[fill=white](1.875,0)circle(2pt);
            \filldraw[fill=white](2.5,0)node[below]{$4$}circle(2pt);
        \end{tikzpicture}
    \end{subfigure}
    \begin{subfigure}{.3\textwidth}
        \centering
        \begin{tikzpicture}
            \draw[->](-.5,0)--(3,0);
            \filldraw(0,0)node[below]{$0$}circle(2pt);
            \filldraw(.3125,0)circle(2pt);
            \filldraw(.625,0)circle(2pt);
            \filldraw(.9375,0)circle(2pt);
            \filldraw(1.25,0)circle(2pt);
            \filldraw[fill=white](1.5625,0)circle(2pt);
            \filldraw(1.875,0)circle(2pt);
            \filldraw[fill=white](2.1875,0)circle(2pt);
            \filldraw[fill=white](2.5,0)node[below]{$4$}circle(2pt);
        \end{tikzpicture}
    \end{subfigure}
    \begin{subfigure}{.3\textwidth}
        \centering
        \begin{tikzpicture}
            \draw[->](-.5,0)--(3,0);
            \filldraw(0,0)node[below]{$0$}circle(2pt);
            \filldraw(.3125,0)circle(2pt);
            \filldraw(.625,0)circle(2pt);
            \filldraw(.9375,0)circle(2pt);
            \filldraw(1.25,0)circle(2pt);
            \filldraw[fill=white](1.5625,0)circle(2pt);
            \filldraw[fill=white](1.875,0)circle(2pt);
            \filldraw(2.1875,0)circle(2pt);
            \filldraw[fill=white](2.5,0)node[below]{$4$}circle(2pt);
        \end{tikzpicture}
    \end{subfigure}
    \begin{subfigure}{.3\textwidth}
        \centering
        \begin{tikzpicture}
            \draw[->](-.5,0)--(3,0);
            \filldraw(0,0)node[below]{$0$}circle(2pt);
            \filldraw(.3125,0)circle(2pt);
            \filldraw(.625,0)circle(2pt);
            \filldraw(.9375,0)circle(2pt);
            \filldraw(1.25,0)circle(2pt);
            \filldraw[fill=white](1.5625,0)circle(2pt);
            \filldraw[fill=white](1.875,0)circle(2pt);
            \filldraw[fill=white](2.1875,0)circle(2pt);
            \filldraw(2.5,0)node[below]{$4$}circle(2pt);
        \end{tikzpicture}
    \end{subfigure}
    \caption{The discrete Okounkov bodies $\Delta_k\subset\Delta\cap\ZZ/k\subset\Delta=[0,4]$ of $\cO_C(K_C)$ over a smooth quartic plane curve $C$, with the flag chosen to be $C\supset\{p\}$, for $k\in\{1,2\}$. There are six cases: $p$ is not a 2-Weierstrass point, $p$ a flex point, $p$ a hyperflex point, and $p$ a $s$-sextactic point for $s\in\{1,2,3\}$, respectively \cite[p. 10]{Ver}, \cite[Lemma 1]{AS},\cite{Cayley}.}\lb{2WeierstrassFig}
\end{figure}
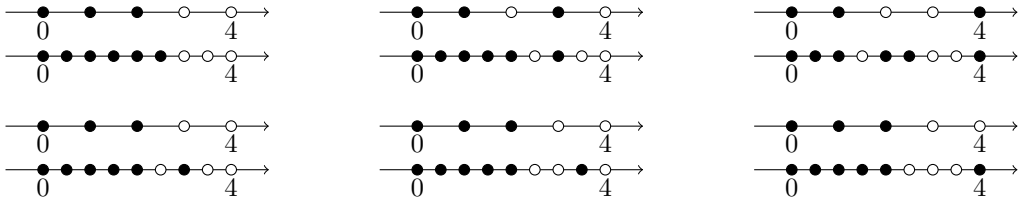

In view of Proposition \ref{Weierstrass2IntroProp}, Problems \ref{DkdkProb}--\ref{DkdkMicrolocalProb}
are precisely higher-dimensional generalizations of Weierstrass gap theory.
In our generalization, points are replaced by {\it flags}, gap sequences
by the $k$-gap {\it sets} $(\Delta\cap \ZZ^n/k)\setminus\Delta_k$ (Definition \ref{kgapsDefn}),
and the genus $g$ by the gap set asymptotics.

Naturally, there is also a dual description of Weierstrass gaps in terms
of a different kind of discrete Okounkov bodies:

\bprop
\lb{WeierstrassIntroProp}
Let $C$ be a smooth curve of genus $g$ and $p$ a 
point in $C$. The Weierstrass gap sequence of $p$ is $1=N_1<\cdots<N_g\le 2g-1$, with
$$
N_i=\min\{k\in\NN\,:\,\#(\Delta\cap \ZZ/k)-\#(\Delta_k)=i\},
$$
where $\Delta_k$ is $k$-th discrete Okounkov body of 
$\cO_C(p)$ associated 
to the flag $C\supset\{p\}$.

\smallskip\noindent
For any $\tau\in(0,1]$ and $k$ sufficiently large,
$\#\big(\Delta^{Q_v(\tau)}\cap \ZZ^n/k\big)-\#\big((\Delta^{Q_v(\tau)})_k\big)=g$. 
\eprop

Proposition \ref{WeierstrassIntroProp} is 
proved in \S\ref{classicalWeierSubSec}. 
The case $g=3$ and $k\le5$ 
is depicted in Figure \ref{VerFig}.

\bigskip
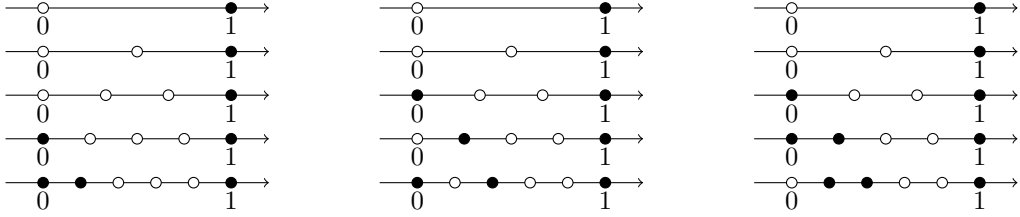
\begin{figure}[ht]
    \centering
        \begin{subfigure}{.3\textwidth}
        \centering
        \begin{tikzpicture}
            \draw[->](-.5,0)--(3,0);
            \filldraw[fill=white](0,0)node[below]{$0$}circle(2pt);
            \filldraw(2.5,0)node[below]{$1$}circle(2pt);
        \end{tikzpicture}
    \end{subfigure}
        \begin{subfigure}{.3\textwidth}
        \centering
        \begin{tikzpicture}
            \draw[->](-.5,0)--(3,0);
            \filldraw[fill=white](0,0)node[below]{$0$}circle(2pt);
            \filldraw(2.5,0)node[below]{$1$}circle(2pt);
        \end{tikzpicture}
    \end{subfigure}
        \begin{subfigure}{.3\textwidth}
        \centering
        \begin{tikzpicture}
            \draw[->](-.5,0)--(3,0);
            \filldraw[fill=white](0,0)node[below]{$0$}circle(2pt);
            \filldraw(2.5,0)node[below]{$1$}circle(2pt);
        \end{tikzpicture}
    \end{subfigure}
        \begin{subfigure}{.3\textwidth}
        \centering
        \begin{tikzpicture}
            \draw[->](-.5,0)--(3,0);
            \filldraw[fill=white](0,0)node[below]{$0$}circle(2pt);
            \filldraw[fill=white](1.25,0)circle(2pt);
            \filldraw(2.5,0)node[below]{$1$}circle(2pt);
        \end{tikzpicture}
    \end{subfigure}
        \begin{subfigure}{.3\textwidth}
        \centering
        \begin{tikzpicture}
            \draw[->](-.5,0)--(3,0);
            \filldraw[fill=white](0,0)node[below]{$0$}circle(2pt);
            \filldraw[fill=white](1.25,0)circle(2pt);
            \filldraw(2.5,0)node[below]{$1$}circle(2pt);
        \end{tikzpicture}
    \end{subfigure}
        \begin{subfigure}{.3\textwidth}
        \centering
        \begin{tikzpicture}
            \draw[->](-.5,0)--(3,0);
            \filldraw[fill=white](0,0)node[below]{$0$}circle(2pt);
            \filldraw[fill=white](1.25,0)circle(2pt);
            \filldraw(2.5,0)node[below]{$1$}circle(2pt);
        \end{tikzpicture}
    \end{subfigure}
    \begin{subfigure}{.3\textwidth}
        \centering
        \begin{tikzpicture}
            \draw[->](-.5,0)--(3,0);
            \filldraw[fill=white](0,0)node[below]{$0$}circle(2pt);
            \filldraw[fill=white](.833,0)circle(2pt);
            \filldraw[fill=white](1.666,0)circle(2pt);
            \filldraw(2.5,0)node[below]{$1$}circle(2pt);
        \end{tikzpicture}
    \end{subfigure}
    \begin{subfigure}{.3\textwidth}
        \centering
        \begin{tikzpicture}
            \draw[->](-.5,0)--(3,0);
            \filldraw(0,0)node[below]{$0$}circle(2pt);
            \filldraw[fill=white](.833,0)circle(2pt);
            \filldraw[fill=white](1.666,0)circle(2pt);
            \filldraw(2.5,0)node[below]{$1$}circle(2pt);
        \end{tikzpicture}
    \end{subfigure}
    \begin{subfigure}{.3\textwidth}
        \centering
        \begin{tikzpicture}
            \draw[->](-.5,0)--(3,0);
            \filldraw(0,0)node[below]{$0$}circle(2pt);
            \filldraw[fill=white](.833,0)circle(2pt);
            \filldraw[fill=white](1.666,0)circle(2pt);
            \filldraw(2.5,0)node[below]{$1$}circle(2pt);
        \end{tikzpicture}
    \end{subfigure}
    \begin{subfigure}{.3\textwidth}
        \centering
        \begin{tikzpicture}
            \draw[->](-.5,0)--(3,0);
            \filldraw(0,0)node[below]{$0$}circle(2pt);
            \filldraw[fill=white](.625,0)circle(2pt);
            \filldraw[fill=white](1.25,0)circle(2pt);
            \filldraw[fill=white](1.875,0)circle(2pt);
            \filldraw(2.5,0)node[below]{$1$}circle(2pt);
        \end{tikzpicture}
    \end{subfigure}
    \begin{subfigure}{.3\textwidth}
        \centering
        \begin{tikzpicture}
            \draw[->](-.5,0)--(3,0);
            \filldraw[fill=white](0,0)node[below]{$0$}circle(2pt);
            \filldraw(.625,0)circle(2pt);
            \filldraw[fill=white](1.25,0)circle(2pt);
            \filldraw[fill=white](1.875,0)circle(2pt);
            \filldraw(2.5,0)node[below]{$1$}circle(2pt);
        \end{tikzpicture}
    \end{subfigure}
    \begin{subfigure}{.3\textwidth}
        \centering
        \begin{tikzpicture}
            \draw[->](-.5,0)--(3,0);
            \filldraw(0,0)node[below]{$0$}circle(2pt);
            \filldraw(.625,0)circle(2pt);
            \filldraw[fill=white](1.25,0)circle(2pt);
            \filldraw[fill=white](1.875,0)circle(2pt);
            \filldraw(2.5,0)node[below]{$1$}circle(2pt);
        \end{tikzpicture}
    \end{subfigure}
    \begin{subfigure}{.3\textwidth}
        \centering
        \begin{tikzpicture}
            \draw[->](-.5,0)--(3,0);
            \filldraw(0,0)node[below]{$0$}circle(2pt);
            \filldraw(.5,0)circle(2pt);
            \filldraw[fill=white](1,0)circle(2pt);
            \filldraw[fill=white](1.5,0)circle(2pt);
            \filldraw[fill=white](2,0)circle(2pt);
            \filldraw(2.5,0)node[below]{$1$}circle(2pt);
        \end{tikzpicture}
    \end{subfigure}
    \begin{subfigure}{.3\textwidth}
        \centering
        \begin{tikzpicture}
            \draw[->](-.5,0)--(3,0);
            \filldraw(0,0)node[below]{$0$}circle(2pt);
            \filldraw[fill=white](.5,0)circle(2pt);
            \filldraw(1,0)circle(2pt);
            \filldraw[fill=white](1.5,0)circle(2pt);
            \filldraw[fill=white](2,0)circle(2pt);
            \filldraw(2.5,0)node[below]{$1$}circle(2pt);
        \end{tikzpicture}
    \end{subfigure}
    \begin{subfigure}{.3\textwidth}
        \centering
        \begin{tikzpicture}
            \draw[->](-.5,0)--(3,0);
            \filldraw[fill=white](0,0)node[below]{$0$}circle(2pt);
            \filldraw(.5,0)circle(2pt);
            \filldraw(1,0)circle(2pt);
            \filldraw[fill=white](1.5,0)circle(2pt);
            \filldraw[fill=white](2,0)circle(2pt);
            \filldraw(2.5,0)node[below]{$1$}circle(2pt);
        \end{tikzpicture}
    \end{subfigure}\caption{The discrete Okounkov bodies $\Delta_k\subset\Delta\cap\ZZ/k\subset\Delta=[0,1]$ of $\cO_C(p)$ over a smooth quartic plane curve $C$, with the flag chosen to be $C\supset\{p\}$, for $k\in\{1,\ldots,5\}$. There are three cases: 
    $p$ is a non-Weierstrass point, i.e., its gap sequence is $1,2,3$; $p$ is a flex point, for which the gap sequence is $1,2,4$; $p$ is a hyperflex point, for which the gap sequence is $1,2,5$ \cite[p. 10]{Ver}.}
    \lb{VerFig}
\end{figure}

\noindent

In \S\ref{DiscreteOkBSubsec} the distribution of gaps near the $\tau=0$ boundary
is emphasized (e.g., Theorem \ref{maxp1Thm}). Propositions \ref{Weierstrass2IntroProp}--\ref{WeierstrassIntroProp}
serve to emphasize that the same phenomenon occurs in classical Weierstrass theory.

With Propositions \ref{Weierstrass2IntroProp}--\ref{WeierstrassIntroProp}
in mind,
it is natural to consider the higher-dimensional
analogue:

\bprob
\lb{WeierstrassProb}
Understand Weierstrass gap patterns of discrete Okounkov bodies.
\eprob

See Figure \ref{HighDimWeierFig} for an example in $n=2$. 

\begin{figure}[ht]
    \centering
    \begin{subfigure}{.3\textwidth}
        \centering
        \begin{tikzpicture}
            \draw[->](-.25,0)--(.75,0);
            \draw(.5,0)node[below]{$\frac{1}{2}$};
            \draw[->](0,-.25)--(0,3.25);
            \draw(0,0)node[below left]{$0$};
            \draw(.5,0)--(.5,1)--(0,3);
            \filldraw(0,0)circle(2pt);
            \filldraw(0,1)circle(2pt)node[left]{$1$};
            \filldraw(0,2)circle(2pt)node[left]{$2$};
            \filldraw(0,3)circle(2pt)node[left]{$3$};
        \end{tikzpicture}
    \end{subfigure}
    \begin{subfigure}{.3\textwidth}
        \centering
        \begin{tikzpicture}
            \draw[->](-.25,0)--(.75,0);
            \draw[->](0,-.25)--(0,3.25);
            \draw(0,0)node[below left]{$0$};
            \draw(.5,0)--(.5,1)--(0,3);
            \filldraw(0,0)circle(2pt);
            \filldraw(0,.5)circle(2pt)node[left]{$\frac{1}{2}$};
            \filldraw(0,1)circle(2pt)node[left]{$1$};
            \filldraw(0,1.5)circle(2pt)node[left]{$\frac{3}{2}$};
            \filldraw(0,2)circle(2pt)node[left]{$2$};
            \filldraw(0,2.5)circle(2pt)node[left]{$\frac{5}{2}$};
            \filldraw(0,3)circle(2pt)node[left]{$3$};
            \filldraw(.5,0)circle(2pt)node[below]{$\frac{1}{2}$};
            \filldraw(.5,.5)circle(2pt);
            \filldraw[fill=white](.5,1)circle(2pt);
        \end{tikzpicture}
    \end{subfigure}
    \begin{subfigure}{.3\textwidth}
        \centering
        \begin{tikzpicture}
            \draw[->](-.25,0)--(.75,0);
            \draw[->](0,-.25)--(0,3.25);
            \draw(0,0)node[below left]{$0$};
            \draw(.5,0)--(.5,1)--(0,3);
            \filldraw(0,0)circle(2pt);
            \filldraw(0,.5)circle(2pt)node[left]{$\frac{1}{2}$};
            \filldraw(0,1)circle(2pt)node[left]{$1$};
            \filldraw(0,1.5)circle(2pt)node[left]{$\frac{3}{2}$};
            \filldraw(0,2)circle(2pt)node[left]{$2$};
            \filldraw(0,2.5)circle(2pt)node[left]{$\frac{5}{2}$};
            \filldraw(0,3)circle(2pt)node[left]{$3$};
            \filldraw(.5,0)circle(2pt)node[below]{$\frac{1}{2}$};
            \filldraw[fill=white](.5,.5)circle(2pt);
            \filldraw(.5,1)circle(2pt);
        \end{tikzpicture}
    \end{subfigure}
    \caption{The Okounkov body of $\cO_X(1,1)$ over $X=\PP^1\times\PP^1$, with the flag chosen to be $X\supset C\supset \{p\}$, where $C$ is a smooth $(2,1)$-curve. The discrete Okounkov body $\Delta_k\subset\Delta\cap\ZZ^2/k$ is drawn in three cases: $k=1$, $k=2$ with $p$ not a ramification point, and $k=2$ with $p$ a ramification point, respectively.}
\lb{HighDimWeierFig}
\end{figure}
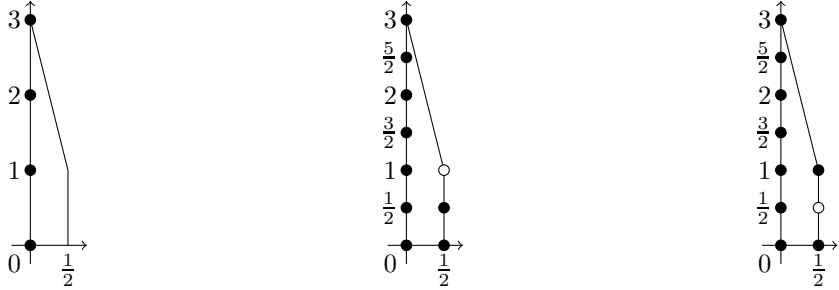

It seems remarkable that determining the asymptotics of $\alpha_k$ and $\delta_k$
is intimately related to higher-dimensional Weierstrass gap phenomena.
As shown here (recall \eqref{DeltaimprovedasympEq}--\eqref{alphaimprovedasympEq}), understanding such patterns 
would lead to improved asymptotic estimates on such thresholds.
In fact, Problem
\ref{WeierstrassProb} seems a rich and fascinating problem in its own right
given the vastness of the theory of Weierstrass gaps on curves. 
For instance, which Weierstrass gap patterns arise in surfaces?

\bigskip
\noindent
{\bf Acknowledgments.} 
Thanks to
H. Blum, I. Cheltsov, and
O. Regev
for helpful discussions.
Research supported in part 
by grants NSF DMS-1906370,2204347, BSF 2020329, 
NSFC 11890660, 12341105, and NKRDPC 2020YFA0712800.

\section{Preliminaries}\label{prelim}

\subsection{Singularities, valuations, log canonical thresholds}
Let $X$ be a normal projective variety and $L$ a line bundle over $X$. 
For $k\in\NN$ denote
\begin{equation}\label{R_k}
    R_k:=H^0\left(X,kL\right),
\end{equation}
(here $R_0=\CC$) and consider the graded ring
$
    R:=\bigoplus_{k=0}^\infty R_k.
$
Let
\begin{equation}\label{d_k}
    d_k:=\dim R_k.
\end{equation}
Recall the equality \cite[Lemma 1.4]{LM09},
\begin{equation}\label{numDiscOkBFull}
    \#\left(\Delta_k\right)=\dim R_k.
\end{equation}
We always assume $L$ is big, i.e., its volume
$$
    \Vol
    L
    :=n!\lim_{k\to\infty}\frac{d_k}{k^n}
$$
is positive, and $\NN(L)$ (Definition \ref{dkNLdeltakmDef}) contains all sufficiently large integers.
\begin{definition}\label{valuation}
    Let $K$ be a field, $K^*=K\setminus\{0\}$ its multiplictive group, and $\Gamma$ an abelian totally ordered group. A map $v:K\to\Gamma\cup\{\infty\}$ is called a {\it valuation} if $v(0)=\infty$ and $v|_{K^*}:K^*\to\Gamma$ is a group homomorphism satisfying
    $
        v\left(a+b\right)\geq\min\left\{v\left(a\right),v\left(b\right)\right\},
    $
    with equality if $v(a)\neq v(b)$. We say $v$ is trivial if $v|_{K^*}\equiv0$.
\end{definition}

Let $\CC(X)$ denote the function field of $X$, i.e., the field of rational functions on $X$. Denote
\beq\lb{ValX}
\ValX:=\{v\,:\, \hbox{$v$ is a non-trivial valuation on $\CC(X)$ trivial on
$\CC\subset\CC(X)$}\}.
\eeq

\begin{example}
\lb{VanOrdEx}
    Any prime divisor $F$ on $X$ determines a discrete valuation (i.e., integer valued) $\ord_F$ on $\CC(X)$, called the {\it vanishing order} along $F$. See \cite[\S II.6]{Har}.
\end{example}

Note that for any section $s\in R_k$, we can make sense of $v(s)$ ($v(D)$, respectively) by choosing a trivialization and expressing $s$ (the defining function of $D$, respectively) locally. More generally, by homomorphism property, $v(kD)=kv(D)$, allowing us to define $v$ on $\QQ$-divisors.

\begin{definition}
\lb{DivValDef}
    For any proper birational morphism $\pi:Y\to X$ where $Y$ is normal, a prime divisor $F$ on $Y$ is called a {\it divisor over $X$}. 
    Let $\ord_F$ denote the vanishing order along $F$.
    A valuation $v$ is called {\it divisorial} if it is in the form of $Cv_{F,\pi}:=C\,\ord_F\circ\pi^*$ for some $C>0$. Specifically, for a section $s\in R_k$,
    $$
        v\left(s\right):=C\,\ord_F\left(\pi^*s\right).
    $$
    Let
    \beq\lb{ValXdiv}
        \ValXdiv\subset\ValX
    \eeq
    denote the set of all divisorial valuations.
\end{definition}

\begin{example}
    Consider $X=\PP^2$. 
    Let $(z_1,z_2)$ be the coordinate on some chart 
    $\CC^2\hookrightarrow\PP^2$. For a holomorphic function
    $
        f\left(z\right)=\sum_{I=(i_1,i_2)}a_Iz_1^{i_1}z_2^{i_2}
    $
    around $0$, define $v(0):=\infty$ and
    $
        v\left(f\right):=\min_{a_I\neq0}\left(i_1+\sqrt2\,i_2\right).
    $
    This defines a valuation. Indeed, define an order on the multi-indices by $I\leq I'$ if
    $
        i_1+\sqrt2\,i_2\leq i_1'+\sqrt2\,i_2'.
    $
    Let
    $
        \deg:f\mapsto\min_{a_I\neq0}\left(i_1,i_2\right)
    $
    denote the multi-degree with respect to this order. Then
    $
        \deg fg=\deg f+\deg g,
    $
    and
    $
        \deg\left(f+g\right)\geq\min\left\{\deg f,\deg g\right\},
    $
    with equality if $\deg f\neq\deg g$.    
    However,
    $v(z_1)=1$ and $v(z_2)=\sqrt2$, while for any divisorial valuation $w_{F,\pi}$ and $f_1,f_2\in\CC(\PP^2)^*$, $\frac{w_{F,\pi}(f_1)}{w_{F,\pi}(f_2)}=\frac{\ord_F(f_1\circ\pi)}{\ord_F(f_2\circ\pi)}\in\QQ$
    by Example \ref{VanOrdEx}. Thus, $v\not\in\hbox{\rm Val}_{\PP^2}^{\Div}$.
\end{example}

\begin{definition}
    Let $X$ be $\QQ$-{\it Gorenstein}, i.e., $X$ is normal and $K_X$ is $\QQ$-Cartier. For any projective birational morphism $\pi:Y\to X$ where $Y$ is normal and a prime divisor $F$ on $Y$, its {\it log discrepancy}
    $$
        A\left(F\right):=1+\ord_F\left(K_Y-\pi^*K_X\right),
    $$
    Say $X$ has (at worst) {\it klt singularities} when $A>0$ on $\ValXdiv$ \eqref{ValXdiv}.
The log discrepancy can be naturally extended from $\ValXdiv$ to
$\ValX$ as a homogeneous function of degree 1 \cite[Theorem 3.1]{BFFU11}.
\end{definition}

\begin{definition}
\lb{lctDef}
    For $s\in R_k$ \eqref{R_k},  its {\it log canonical threshold} is
    $$
        \lct\left(s\right):=\sup\left\{c>0\,:\,\left|s\right|^{-2c}\text{ is locally integrable on }X\right\}.
    $$
    Note that since $\lct(s)=\lct(\lambda s)$ for $\lambda\neq0$, 
    lct descends to a function on divisors.
\end{definition}

The following gives an algebraic definition for log canonical threshold. See, e.g., \cite[\S 1.8]{BJ20}.
\begin{theorem}\label{lct def}
    Let $X$ be klt. For $s\in R_k$ \eqref{R_k},
    $
    \lct\left(s\right)=\inf_{v\in\ValXdiv}\frac{A\left(v\right)}{v\left(s\right)}=\inf_{v\in\ValX}\frac{A\left(v\right)}{v\left(s\right)}.
    $
    Moreover, infimum is obtained in $\ValXdiv$ \eqref{ValXdiv}.
\end{theorem}

\subsection{The jumping function on the discrete Okounkov bodies}

References for this subsection are \cite{HM,BC11,BJ20}.

\begin{definition}
\lb{JumpingDef}
    For $v\in\ValX$ \eqref{ValX}, define
    $$
        \cF_v^\lambda R_k:=\left\{s\in R_k\,:\,v\left(s\right)\geq \lambda\right\}.
    $$
    The {\it jumping function of $R$}, $j=j_{\bullet,\bullet}:
    \ValX\ra\cup_{k\in\NN(L)}\ZZ_+^{d_k}$,
    with $j(v)=\{\{j_{k,m}(v)\}_{m\in\{1,\ldots,d_k\}}\}
    _{k\in\NN(L)}$, 
    and
    $$
        j_{k,\ell}(v):=\inf\left\{\lambda\geq0\,:\,\dim\cF_v^\lambda R_k\leq\ell-1\right\}, \q
\hbox{for $\ell\in\{1,\ldots,d_k\}$.}
    $$
\end{definition}
Note that
\begin{align}
        j_{k,\ell}(v)&=\inf\left\{\lambda\geq0\,:\,\dim\cF_v^\lambda R_k\leq\ell-1\right\}
        =\sup\left\{\lambda\geq0\,:\,\dim\cF_v^\lambda R_k\geq\ell\right\}\nonumber\\
        &=\max\left\{\lambda\geq0\,:\,\dim\cF_v^\lambda R_k\geq\ell\right\},
        \label{jumping number}
\end{align}
with the supremum achieved because of left-continuity of $\dim\cF_v^\lambda R_k$. Left-continuity follows from the fact that $\cF_v^\lambda R_k$ is decreasing with respect to $\lambda$ and
$
    \cF_v^\lambda R_k=\bigcap_{\lambda'<\lambda}\cF_v^{\lambda'}R_k.
$
Moreover, by \eqref{jumping number},
\beq\lb{jumpingEq}
    j_{k,1}\geq\cdots\geq j_{k,d_k}.
\eeq
For instance, if $v$ is a divisorial valuation
associated to a divisor $D\subset X$, then
$j_{k,1}$ is the largest vanishing order of a section in $R_k$ along $D$.

\begin{remark}
    Sometimes jumping numbers are defined as an increasing sequence \cite[p. 256]{HM},\cite{BJ20}. Here they are defined in decreasing order in anticipation of Definition \ref{upper S}.
\end{remark}

\subsection{The \texorpdfstring{$S_{k,m}$}{Skm} invariants}

The next definition generalizes the ``$k$-th pseudoeffective threshold" $S_{k,1}$ ($m=1$)
and the $k$-th expected vanishing order $S_{k,d_k}$ ($m=d_k$).
It plays a central r\^ole in this article.

\begin{definition}\lb{SkmDef}
    For $k\in\NN(L)$ and $m\in\{1,\ldots,d_k\}$, 
    denote by $S_{k,m}:\ValX\ra[0,\infty]$,
    $$
        S_{k,m}
        :=\frac{1}{km}\sum_{\ell=1}^mj_{k,\ell}.
    $$
\end{definition}

The notation $S_{k,m/d_k}$ would have perhaps been more suggestive, in view of 
Proposition \ref{S_tau def}. For brevity, we stick with $S_{k,m}$.

\begin{definition}\lb{S0Def}
The {\it maximal vanishing order} is the function $\cS_0:\ValX\ra[0,\infty]$,
    \begin{equation}\label{def T}
        \cS_0:=\sup_{k\in\NN(L)}S_{k,1}=\sup_k\frac{j_{k,1}}{k}.
    \end{equation}
Let $\sigma:\ValX\ra[0,\infty)$,
$$\sigma(v):=\inf_k\frac{j_{k,d_k}}{k}.$$
be the {\it minimal vanishing order of $v$} \cite[Definition 1.1, p. 79]{Nak}.
    Let
    \beq\lb{ValXfin}
        \ValXfin:=\left\{v\in\ValX\,:\,\cS_0\left(v\right)<\infty\right\}
    \eeq
    be the set of non-trivial {\it valuations with finite vanishing order} (also
    called valuations of linear growth).    
    The {\it expected vanishing order} is the function $\cS_1:\ValXfin\ra[0,\infty)$,
    $$
        \cS_1
        :=\lim_{k\to\infty}S_{k,d_k}
    $$
    (the limit exists \cite[Corollary 3.6]{BJ20}).
\end{definition}
Note that $\ValXdiv\subset\ValXfin$ (recall \eqref{ValXdiv}).
Restricted to $\ValXdiv$, $\cS_0$ is the pseudoeffective threshold, and 
$\cS_1$ is the expected~vanishing~order.

\begin{remark}\label{Fkt rmk}
The following is well-known; we record it since it will
be used several times below in other instances.
    Recall Fekete's Lemma: let $\{c_k\}_{k=1}^\infty$ be sub-additive (super-additive, respectively), i.e.,
    $
        c_{i+j}\leq c_i+c_j \q (c_{i+j}\geq c_i+c_j, \hbox{\rm respectively}).
    $
    Then
    $  \lim_{k\to\infty}\frac{c_k}{k}
        $
        equals 
        $\inf_k\frac{c_k}{k}
    $ ($\sup_k\frac{c_k}{k}$, respectively).
    Let $v\in\ValX$. For $k,k'\in\NN$, $s\in R_k$, $s\in R_{k'}$, $s\otimes s'\in R_{k+k'}$ and (recall the homomorphism property in Definition \ref{valuation})
    $
        v\left(s\otimes s'\right)=v\left(s\right)+v\left(s'\right).
    $
    Consequently,
    \begin{equation}\label{fkt}
        kS_{k,1}\left(v\right)+k'S_{k',1}\left(v\right)\leq\left(k+k'\right)S_{k+k',1}\left(v\right),
    \end{equation}
    so \eqref{def T} can be improved to
    $
        \lim_{k\to\infty}S_{k,1}\left(v\right)=\sup_kS_{k,1}\left(v\right)=\cS_0\left(v\right).
    $
\end{remark}

Notice that $kS_{k,m}$ is the average of the largest $m$ jumping numbers, so by \eqref{jumpingEq}, 
    \begin{equation}\label{m monotonicity}
    m\leq m'
    \q \Rightarrow \q
        S_{k,m}\geq S_{k,m'}.
    \end{equation}
Going back to our example of $v$ associated to $D\subset X$,
suppose one is tasked with finding $m$ linearly independent 
sections $s_1,\ldots,s_m$ 
of $H^0(X,kL)$ such that their average vanishing order along $D$ is maximized.
That maximum is achieved and equal $S_{k,m}$. This is the content of the next lemma for general 
valuations (this generalizes \cite[Lemma 2.2]{FO18},\cite[Lemma 3.5]{BJ20} 
see~also~\cite[Lemma 2.7]{CRZ}):

\begin{lemma}\label{compatible basis lem}
Let $k\in\NN(L)$ and $m\in\{1,\ldots,d_k\}$.
    For $v\in\ValX$,
    \begin{equation}\label{compatible basis}
        kmS_{k,m}\left(v\right)=\max_{\substack{\left\{s_\ell\right\}_{\ell=1}^m\subset R_k\\\mbox{\smlsev linearly independent}}}v\left(\sum_{\ell=1}^m\left(s_\ell\right)\right).
    \end{equation}
    For any maximizing linearly independent family, $v(s_1),\ldots,v(s_m)$ are the largest $m$ jumping numbers.
\end{lemma}
\begin{proof}
    We may assume $v(s_1)\geq\cdots\geq v(s_m)$. Then for any $\ell\geq\ell'$,
    $$
        s_{\ell'}\in\cF_v^{v\left(s_\ell\right)}R_k.
    $$
    Since $\{s_\ell\}_{\ell=1}^m$ is linearly independent,
    $
        \dim\cF_v^{v\left(s_\ell\right)}R_k\geq\ell.
    $
    By \eqref{jumping number},
    $
        j_{k,\ell}\left(v\right)\geq v\left(s_\ell\right).
    $
    Therefore,
    $$
        v\left(\sum_{\ell=1}^m\left(s_\ell\right)\right)=\sum_{\ell=1}^mv\left(s_\ell\right)\leq\sum_{\ell=1}^m j_{k,\ell}\left(v\right)=kmS_{k,m}\left(v\right).
    $$
    To see the converse, construct a sequence $\{s_\ell\}_{\ell=1}^m$ inductively as follows. By \eqref{jumping number},
    $
        \dim\cF_v^{j_{k,1}\left(v\right)}\geq 1,
    $
    so we can find a non-zero section $s_1\in\cF_v^{j_{k,1}\left(v\right)}$.
    
    For each $\ell\geq2$, suppose $s_1,\ldots,s_{\ell-1}$ have been chosen. By \eqref{jumping number},
    $
        \dim\cF_v^{j_{k,\ell}\left(v\right)}\geq\ell,
    $
    so we can find $s_\ell\in\cF_v^{j_{k,\ell}\left(v\right)}$ that is not spanned by $s_1,\ldots,s_{\ell-1}$.
    
    Now, by construction, we have chosen a family $\{s_\ell\}_{\ell=1}^m$ such that each
    \begin{equation}\label{ind assumption}
        s_\ell\in\cF_v^{j_{k,\ell}\left(v\right)}
    \end{equation}
    is not spanned by $s_1,\ldots,s_{\ell-1}$, i.e., it is a linearly independent family $\{s_\ell\}_{\ell=1}^m$ satisfying (by rewriting \eqref{ind assumption} using \eqref{jumping number})
    $
        v\left(s_\ell\right)\geq j_{k,\ell}\left(v\right).
    $
    Hence
    $$
        v\left(\sum_{\ell=1}^m\left(s_\ell\right)\right)=\sum_{\ell=1}^mv\left(s_\ell\right)\geq\sum_{\ell=1}^mj_{k,\ell}\left(v\right)=kmS_{k,m}\left(v\right).
    $$
\end{proof}

\begin{definition}\label{compat def}
    Let $k\in\NN(L)$, $m\in\{1,\ldots,d_k\}$, and
     $v\in\ValX$ \eqref{ValX}. A linearly independent family $\{s_\ell\}_{\ell=1}^m$ is said to be 
    {\it compatible with} $\cF_v$ if
    $
        v\left(\sum_{\ell=1}^m\left(s_\ell\right)\right)=kmS_{k,m}\left(v\right).
    $
\end{definition}

\begin{remark}\label{adapted extension}
    Let $m'\geq m$. 
    The inductive construction of Lemma \ref{compatible basis lem} shows that a
    compatible family $\{s_\ell\}_{\ell=1}^m$ can be extended to a larger
    compatible family $\{s_\ell\}_{\ell=1}^{m'}\supset \{s_\ell\}_{\ell=1}^m$.
\end{remark}

\subsection{Measures associated to a 
valuation}

\begin{definition}
\lb{GradedLinearSerierDef}
For each $k\in\ZZ_+, 
t\in\RR_+$, and $v\in\ValX$, denote by (recall Definition 
\ref{JumpingDef})
    $$
        V_{v,k}^t:=\cF_v^{kt}R_k=
        \left\{s\in R_k\,:\,v\left(s\right)\geq kt\right\},
    $$
the $\CC$-vector subspace of $R_k$, and by
$$
V_{v,\bullet}^t:=\{V_{v,k}^t\}_{k\in\ZZ_+},
$$
the {\it $v$-graded linear series of $L$} \cite[\S2.4]{BJ20}. The {\it volume} of this graded linear series is
    \beq\lb{VolGradedEq}
        \Vol V_{v,\bullet}^t:=\lim_{k\to\infty}\frac{n!}{k^n}\dim V_{v,k}^t.
    \eeq
\end{definition}

The Brunn--Minkowski inequality implies \cite[Lemma 2.22]{BKMS} (see also \cite[Proposition 2.3]{BJ20}):
\begin{proposition}\label{V.t}
Recall Definition \ref{S0Def}.
     $t\mapsto(\Vol V_{v,\bullet}^t)^\frac{1}{n}$ is concave on $[0,\cS_0(v))$.
\end{proposition}
Let
$
    \delta_x
$
denote the Dirac function supported on $x\in\RR$ and consider the sequence of
empirical measures on $\RR$ (recall (\ref{jumping number}))
\beq\lb{muvkEq}
    \mu_{v,k}=\frac{1}{d_k}\sum_{\ell=1}^{d_k}\delta_{\frac{1}{k}j_{k,\ell}(v)}.
\eeq
Recall the following \cite[Theorem 1.11]{BC11} (see also \cite[Corollary 2.4, Theorem 2.8]{BJ20}).
\begin{theorem}\label{mu}
    For $v\in\ValXfin$ \eqref{ValXfin}, 
    $\mu_{v,k}$ converges weakly to the probability measure on $\RR$,
    \beq\lb{muvEq}
        \mu_v=-\frac{d}{dt}\frac{\Vol V_{v,\bullet}^t}{\Vol L}.
    \eeq
    By concavity (and hence locally Lipschitz regularity) of $(\Vol V_{v,\bullet}^t)^\frac{1}{n}$ (recall Proposition \ref{V.t}), $\mu_v$ is absolutely continuous on $[0,\cS_0(v))$ with respect to Lebesgue measure, and
    \begin{equation}\label{mu charge}
        \mu_v\left(\left\{\cS_0\left(v\right)\right\}\right)=\lim_{t\to\cS_0\left(V\right)^-}\frac{\Vol V_{v,\bullet}^t}{\Vol L}.
    \end{equation}
\end{theorem}

We record two consequences of the concavity of the function $t\mapsto(\Vol V_{v,\bullet}^t)^\frac{1}{n}$ on $[0,\cS_0(v))$ (Proposition \ref{V.t}).
The first was first used by Fujita--Odaka \cite[Lemma 1.2]{FO18},
while the second is straightforward.

\begin{claim}\label{conv lem 1}
    For $t\in[0,\cS_0(v))$,
    $\disp
        \Vol V_{v,\bullet}^t\geq\left(1-\frac{t}{\cS_0\left(v\right)}\right)^n\VolL.
    $    
\end{claim}
\begin{proof}
    Pick $T_0\in[t,\cS_0(v))$. By Proposition \ref{V.t},
    $
        \left(\Vol V_{v,\bullet}^t\right)^\frac{1}{n}\geq\left(1-\frac{t}{T_0}\right)\left(\Vol V_{v,\bullet}^0\right)^\frac{1}{n}+\frac{t}{T_0}\left(\Vol V_{v,\bullet}^{T_0}\right)^\frac{1}{n}\geq\left(1-\frac{t}{T_0}\right)\left(\VolL\right)^\frac{1}{n}.
    $
    Since this holds for any $T_0\in[t,\cS_0(v))$,
    $
        \left(\Vol V_{v,\bullet}^t\right)^\frac{1}{n}\!\geq\!\sup_{T_0\in\left[t,\cS_0\left(v\right)\right)}\left(1-\frac{t}{T_0}\right)\left(\VolL\right)^\frac{1}{n}=\left(1-\frac{t}{\cS_0\left(v\right)}\right)\left(\VolL\right)^\frac{1}{n}.
    $
\end{proof}
\begin{claim}\label{conv lem 2}
    Let $s\in[0,\cS_0(v))$. If
    $
        \Vol V_{v,\bullet}^s<\VolL,
    $
    then for $t\in(s,\cS_0(v))$,
    $
        \Vol V_{v,\bullet}^t<\Vol V_{v,\bullet}^s.
    $
\end{claim}
\begin{proof}
    Assume there is $t\in(s,\cS_0(v))$ such that
    $
        \Vol V_{v,\bullet}^t=\Vol V_{v,\bullet}^s.
    $
    By Proposition \ref{V.t},
    $$
        \left(\Vol V_{v,\bullet}^s\right)^\frac{1}{n}\geq\left(1-\frac{s}{t}\right)\left(\Vol V_{v,\bullet}^0\right)^\frac{1}{n}+\frac{s}{t}\left(\Vol V_{v,\bullet}^t\right)^\frac{1}{n}=\left(1-\frac{s}{t}\right)\left(\VolL\right)^\frac{1}{n}+\frac{s}{t}\left(\Vol V_{v,\bullet}^s\right)^\frac{1}{n}.
    $$
So $
        \Vol V_{v,\bullet}^s\geq\VolL>\Vol V_{v,\bullet}^s,
    $
     a contradiction.
\end{proof}

\subsection{A 1-parameter family of Okounkov bodies}

A key object in this article is the 1-parameter family of filtered Okounkov bodies
associated to $v$ introduced by Boucksom--Chen \cite{LM09,BC11,BJ20}.

\begin{definition}\label{G def}
Recall \S\ref{DiscreteOkBSubsec}.    Fix an admissible flag on $X$.
    For $v\in\ValXfin$, denote by 
    $$\Delta_v^t, \q t\geq0,$$ 
    the Okounkov body associated to $V_{v,\bullet}^t$ \cite[Definition 1.16]{LM09}, where
(recall Definition \ref{GradedLinearSerierDef})
    $$
        V_{v,k}^t:=\cF_v^{kt}R_k, \q k\in\ZZ_+.
    $$
    In particular,
    $$
        \Delta:=\Delta_v^0
    $$
    is an Okounkov body of $L$,
    and $\Delta_v^t=\emptyset$ for $t>\cS_0(v)$ \eqref{def T}.
\end{definition}
\noindent
By Lazarsfeld--Musta\c t\u a (recall \eqref{VolGradedEq}) \cite[Theorem 2.13]{LM09},
    \begin{equation}\label{LM}
        \left|\Delta_v^t\right|=\frac{1}{n!}\Vol V_{v,\bullet}^t.
    \end{equation}

Denote by
\begin{equation}\label{p1Eq}
    p_1:\RR^n\ra\RR
\end{equation}
the projection $x=(x_1,\ldots,x_n)\mapsto x_1$.
On $\ValXdiv$, the Okounkov bodies $\Delta_v^t$ can be studied via
their slices, i.e., $p_1$-fibers.

\blem
\lb{sliceDivOkBodyLem}
Suppose $v\in\ValXdiv$.
There exists an admissible flag such that
\beq
\lb{Deltavt}
    \Delta_v^t
    =
    \Delta\cap p_1^{-1}\left(\left[t,\cS_0\left(v\right)\right]\right),
    \q \hbox{\ for $t\in[0,\cS_0(v)]$}.
\eeq
\elem

Note that $\Delta_v^t$ are convex bodies in $\RR^n$ 
for $t<\cS_0(v)$ and in the limit $t\ra\cS_0(v)^-$
{\it collapse} to a lower-dimensional convex set.
Note that $\Delta_v^t=\Delta_v^{\sigma(v)}$ for $t\in[0,\sigma(v)]$.

\bpf
We may suppose $v\in\ValXdiv$ is of the form 
$v=v_{F,\sigma}$ for some $F\subset Y$ and 
a resolution $\sigma:Y\to X$ (Definition \ref{DivValDef}). 
Choose an admissible flag contained in $F$:
\beq
\lb{flagFEq}
Y\supset F\supset\cdots\supset Y_n,
\eeq
and let $\Delta$ denote the Okounkov body of $\sigma^*L$. 

Assume $t$ is rational with denominator $q_t$. Note that 
$
\Delta\cap p_1^{-1}\left(\left[t,\cS_0\left(v\right)\right]\right)
=
\Delta(\sigma^*L-tF)+(t,0,\ldots,0)
$
\cite[Theorem 4.26]{LM09}.
The body $q_t\Delta(\sigma^*L-tF)$ is constructed from the graded ring whose
$k$-th graded piece is
$R_k'=H^0(q_tk(\sigma^*L-tF))=H^0(q_tk\sigma^*L-q_tktF)$.
It can be identified with $V^t_{v,q_tk}$ by dividing by 
a local defining section of $F$ to the power $q_tt$; thus,
by Definition \ref{G def} it equals 
$q_t\Delta_v^t$ translated by the vector $(-q_tt,0,\ldots,0)$.
Dividing by $q_t$ proves \eqref{Deltavt} for rational $t$. 
By continuity, it follows also for irrational $t$.
\epf

Note that for this choice of admissible flag, by closedness of the Okounkov body,
$
    p_1^{-1}\left(\cS_0\left(v\right)\right)\cap\Delta\neq\emptyset,
$
i.e., $\max_\Delta p_1=\cS_0(v)$. However, for $t$ in a neighborhood of $0$ it is possible that $p_1^{-1}(t)\cap\Delta=\emptyset$, i.e., $\min_\Delta p_1\ge 0$ with equality if (but not only if)
$L$ is ample (see below). Set $|\emptyset|=0$.

\blem
\lb{muvdivLem}
Suppose $v\in\ValXdiv$. 
There exists an admissible flag such that
$\mu_v$ is 
absolutely continuous with respect to 
Lebesgue measure on $\RR$. Moreover, its density is
$$d\mu_v=
   \frac{1}{\left|\Delta\right|}\left|\Delta\cap p_1^{-1}\left(s\right)\right|ds,
\q s\in[0,\cS_0(v)].
$$
In particular, $d\mu_v(s)=f(s)^{n-1}ds$ with $f$ concave on $p_1(\Delta)=[\sigma(v),\cS_0(v)]$
and zero on $[0,\sigma(v)]$.

\elem

\bpf
As in the proof of Lemma \ref{sliceDivOkBodyLem},
let 
$v=v_{F,\sigma}$, and let $\Delta$ denote the Okounkov body of $\sigma^*L$
with respect to the flag \eqref{flagFEq}. 
By Lemma \ref{sliceDivOkBodyLem},
$$
    \left|\Delta_v^t\right|=\left|\Delta\cap p_1^{-1}\left(\left[t,\cS_0\left(v\right)\right]\right)\right|=\int_t^{\cS_0\left(v\right)}\left|\Delta\cap p_1^{-1}\left(s\right)\right|ds.
$$
Differentiating,
$
    d\mu_v=\frac{n!}{\VolL}\left|\Delta\cap p_1^{-1}\left(s\right)\right|ds=\frac{1}{\left|\Delta\right|}\left|\Delta\cap p_1^{-1}\left(s\right)\right|ds.
$
In particular, on its support $\mu_v$ is absolutely continuous with respect to Lebesgue measure, hence also on $\RR$.
(cf. \cite[Corollary 4.27]{LM09}).
Finally, by Brunn's concavity,
 $s\mapsto|\Delta\cap p_1^{-1}(s)|^\frac{1}{n-1}$ is concave on 
 $p_1(\Delta)$.
\epf

Results of Ein et al. \cite{ELMNP09} and an observation of Fujita
 \cite[Proposition 2.1]{Fuj19}~(cf.~\cite[Proposition 3.11]{BJ20})~give:
\blem
\lb{ELMNPFujLem}
Suppose $v\in\ValXdiv$ and consider the flag guaranteed by Lemma \ref{muvdivLem}.
If, in addition,
$L$ is ample, $p_1(\Delta)=[0,\cS_0(v)]$, i.e., $\sigma(v)=0$.
\elem

\begin{corollary}\label{FujLem}
    Suppose $L$ is ample. For $v\in\ValXdiv$,
    $
        d\mu_v=f\left(s\right)^{n-1}ds
    $
    where $f$ is a concave function on $[0,\cS_0(v)]$.
\end{corollary}

Boucksom--Chen's concave transform 
associated to $v$ is
\cite[Definition 1.8]{BC11},
    \beq\lb{GvDef}
        G_v\left(x\right):=\sup\left\{t\geq0\,:\,x\in\Delta_v^t\right\},
        \q x\in \Delta.
    \eeq
    In other words,
    $
        \left\{G_v\geq t\right\}=\Delta_v^t.
    $
The measure $\mu_v$ \eqref{muvEq} is related to $G_v$
\cite[Theorem 1.11]{BC11},
 \beq\label{muG}
 d\mu_v=\frac{n!}{\Vol L}\left(G_v\right)_*dx,
\eeq
 where $dx$ denote the Lebesgue measure on $\Delta$. 
Note that 
\beq\lb{Gp1Eq}
G=p_1, \q \hbox{for $v\in\ValXdiv$ with the conventions of Lemma \ref{sliceDivOkBodyLem}.} 
\eeq

\section{Estimates on discrete Okounkov bodies}\label{idealized sec}

\subsection{Classical Weierstrass gap sequences and 1-D discrete Okounkov bodies}
\lb{classicalWeierSubSec}

Proposition \ref{WeierstrassIntroProp} is contained in the following:

\bprop
\lb{WeierstrassProp}
Let $C$ be a smooth curve of genus $g$ and $p$ a 
point in $C$. Then:

\smallskip
\noindent
(i)
The Weierstrass gap sequence of $p$ is $1=N_1<\cdots<N_g\le 2g-1$, with
$$
N_i=\min\{k\in\NN\,:\,\#(\Delta\cap \ZZ/k)-\#(\Delta_k)=i\},
$$
where $\Delta_k$ is $k$-th discrete Okounkov body of 
$\cO_C(p)$ associated 
to the flag $C\supset\{p\}$.

\smallskip
\noindent
(ii)
The Weierstrass gap sequence is determined by the first $N_g$ discrete Okounkov bodies.
Conversely, the Weierstrass gap sequence determines all discrete Okounkov bodies by
$$
    k\Delta_k=\left\{0,\ldots,k\right\}\setminus\left\{k-N_1,\ldots,k-N_g\right\},  \q k\in\NN.
$$

\smallskip
\noindent
(iii)
Equivalently, $k$ is a Weierstrass gap if and only if
$\min\Delta_k>\min\Delta=0$, and
$$
\{N_1,\ldots,N_g\}=\{k\in\NN\,:\, \min\Delta_k>0\}.
$$


\smallskip
\noindent
(iv) $p$ is a non-Weierstrass point, i.e., 
$\{N_1,\ldots,N_g\}=\{1,\ldots,g\}$,  
if and only if $\Delta_1=\ldots=\Delta_g=\{1\}$.

\smallskip
\noindent
(v) 
$\#(\Delta\cap \ZZ^n/k)-\#(\Delta_k)=g$ for $k\ge 2g-1$.
Moreover, given
$\tau\in(0,1]$, then for $k$ sufficiently large,
$\#\big(\Delta^{Q_v(\tau)}\cap \ZZ^n/k\big)-\#\big((\Delta^{Q_v(\tau)})_k\big)=g$. 



\eprop

\bpf
To start, recall the classical definition. For a curve $C$ of genus $g\geq1$,  $N\in\NN$ is a {\it Weierstrass gap} of $p\in C$ if there is no meromorphic function on $C$ that is holomorphic away from $p$ and has a pole of order $N$ at $p$ 
\cite{Ver}. 
Practically, the gaps are determined as follows.
Consider the map $\cW: 
\NN\longmapsto\{0,1\}$,
$
i\;\stackrel\cW{\longmapsto}\; h^0(ip)-h^0\big((i-1)p\big). 
$
Using Riemann--Roch 
$2g-1$ times, i.e., $h^0(kp)-h^1(kp)=k+1-g$ for $k=1,\ldots,2g-1$,
together with $g=h^1\ge h^1(p)\ge h^1(2p)\ge\cdots\ge h^1((2g-1)p)=h^0(K_C-(2g-1)p)=0$
and $1=h^0\le h^0(p)\le h^0(2p)\le\cdots\le h^0((2g-1)p)=g$, 
implies that $\#\cW^{-1}(0)=g$. The set of Weierstrass gaps
is $\cW^{-1}(0)=\{N_1,\ldots,N_g\}\subset\{1,\ldots,2g-1\}$:
if $N\in \cW^{-1}(0)$, i.e., $h^0(Np)=h^0((N-1)p)$, the
injection $H^0((N-1)p)\hookrightarrow H^0(Np)$ gives 
$H^0((N-1)p)= H^0(Np)$, i.e., every section in $H^0(Np)$
vanishes at $p$, so $N$ is a Weierstrass gap, and vice versa. 
It follows that there are precisely $g$ gaps and $1=N_1<\cdots<N_g\leq 2g-1$.

This can be translated 
to the language of discrete Okounkov bodies as follows.
First observe that $N$ is a Weierstrass gap of $p$ precisely when there is no section $s\in H^0(C,kp)$ 
with $\ord_ps=k-N$, for a(ny) fixed $k\geq N$. Equivalently,
$
    k\Delta_k=\left\{0,\ldots,k\right\}\setminus\left\{k-N_1,\ldots,k-N_g\right\},  \q k\in\NN,
$
where $\Delta_k$ is as in the statement. This proves (ii).
Note that (iii) is a restatement of  (ii). 
The point $p$ is called a {\it Weierstrass point} if $\cW^{-1}(0)\not=\{1,\ldots,g\}$
and so (iv) follows from (ii),
and so does (v).

While (i) follows from (ii), we prove (i) directly by induction to illustrate how 
the $\Delta_k$ can
be computed in practice (see Figure \ref{VerFig}). Note that $\Delta_1=\{1\}$, so $\#\Delta\cap\ZZ-\#\Delta_1=1$.
Suppose (i) holds for $N_1,\ldots,N_{\ell-1}$ for some $\ell\le g$. 
Note that $\Delta=[0,1]$ so
$\#(\Delta\cap\ZZ/k)=k+1.$
Thus, $(N_{\ell-1}+1)(\Delta\cap\ZZ/(N_{\ell-1}+1))$ has one more point than 
$N_{\ell-1}(\Delta\cap\ZZ/N_{\ell-1})$.
There are two cases. First, that new point is not contained in $(N_{\ell-1}+1)(\Delta_{N_{\ell-1}+1})$,
meaning
$N_{\ell-1}+1$ is a Weierstrass
gap by (i), i.e., $N_\ell=N_{\ell-1}+1$; but it also means  
that $\#(\Delta\cap \ZZ/N_\ell)-\#(\Delta_{N_\ell})=\ell-1+1$ by the induction
hypothesis, thus proving (i) in this case.
Second, if  the new point is contained in $(N_{\ell-1}+1)(\Delta_{N_{\ell-1}+1})$
then by (ii), $N_\ell>N_{\ell-1}+1$. Since there are precisely $g$ gaps, we repeatedly
increase the size of the discrete Okounkov body by 1, and are guaranteed 
that for some minimal $N_\ell>N_{\ell-1}$ the new point added in the last step will not be  contained in 
$N_{\ell}\Delta_{N_{\ell}}$ and by minimality and the induction hypothesis once again 
$\#(\Delta\cap \ZZ/N_\ell)-\#(\Delta_{N_\ell})=\ell-1+1$, proving (i).
\epf

Proposition \ref{Weierstrass2IntroProp} is a special case of the following:

\bprop
\lb{Weierstrass2Prop}
Let $C$ be a smooth curve of genus $g\geq2$ and $p$ 
point in $C$. Let $d_1=g$ and $d_k=k(2g-2)+1-g$ for $k\ge2$. Then:

\smallskip
\noindent
(i) The $k$-Weierstrass gap sequence of $p$ is $1=N^{(k)}_1<\cdots<N^{(k)}_{d_k}\le k(2g-2)+1$, with
$
\{N^{(k)}_1,\ldots,
N^{(k)}_{d_k}\}=k\Delta_k+1,
$
where $\Delta_k$ is $k$-th discrete Okounkov body of 
$K_C$ associated 
to the flag $C\supset\{p\}$.



\smallskip
\noindent
(ii)
For all $k\ge2$, $\#(\Delta\cap \ZZ^n/k)-\#(\Delta_k)=g$, while $\#(\Delta\cap \ZZ^n)-\#(\Delta_1)=g-1$.
The point $p$ is a $k$-Weierstrass point if and only if
$
k\Delta_k\not=\{0,1,\ldots,d_k-1\}.
$

\smallskip
\noindent
(iii) Equivalently, $\max\Delta-\max\Delta_k\le \frac{g}k$ ($\max\Delta-\max\Delta_1\le g-1$)
with equality if and only if $p$ is not a $k$-Weierstrass point for $k\ge2$ (respectively, $k=1$).




\eprop

\bpf
Recall that 
$N$ is a $k$-Weierstrass gap when there is
$s\in H^0(C,kK_C)$ with $\ord_ps=N-1$,
and that $p$ is a $k$-Weierstrass point when there
is $s\in H^0(C,kK_C)$ with $\ord_ps\ge d_k=h^0(C,kK_C)$.
For $k=1$, 
$N$ is a Weierstrass gap when 
$h^0(Np)=h^0((N-1)p)$, equivalently
when $h^1((N-1)p)-h^1(Np)=1$, or by
Kodaira--Serre duality 
$h^0(K_C-(N-1)p)-h^0(K_C-Np)=1$,
i.e., there exists a section $s\in H^0(K_C)$
with $\ord_p(s)=N-1$, i.e., $N-1\in \Delta_1$.
The proof for $k=1$ now follows from that of 
Proposition \ref{WeierstrassProp}.
For $k\ge2$, the proof follows
similarly by noting that
$d_k=k(2g-2)+1-g$ as $h^1(kK_C)=0$. 
\epf

\subsection{Idealized jumping function}

In this subsection we introduce the idealized jumping function and compare it with the usual one.

\begin{definition}\label{iklDef}
    Recall Definition \ref{G def}. Let
    \begin{equation}
        D_k:=\#\left(\Delta\cap\ZZ^n/k\right).
    \end{equation}
    The {\it idealized jumping function of $R$}, $i=i_{\bullet,\bullet}:
    \ValX\ra\cup_{k\in\ZZ_+}\ZZ_+^{D_k}$,
    is \newline 
    $i(v)=\{\{i_{k,m}(v)\}_{m\in\{1,\ldots,D_k\}}\}
    _{k\in\ZZ_+},\;$ 
    where
    \begin{equation}\label{bk def}
        i_{k,\ell}\left(v\right):=k\max\left\{t\geq0\,:\,\#\left(\Delta_v^t\cap\ZZ^n/k\right)\geq\ell\right\}.
    \end{equation}
\end{definition}

Similarly to \eqref{jumpingEq},
\beq\lb{ijumpingEq}
    i_{k,1}\geq\cdots\geq i_{k,D_k}.
\eeq

\blem
\label{ikljkl1stLem}
    For $k\in\NN(L)$ and $\ell\in\{1,\ldots,d_k\},\;$
    $
       j_{k,\ell}\left(v\right)\leq i_{k,\ell}\left(v\right).
    $
\elem
\bpf
Recall the equality \cite[Lemma 1.4]{LM09},
\begin{equation}\label{numDiscOkB}
    \#\left((\Delta^t_v)_k\right)=V_{v,k}^t,
\end{equation}
and that the latter equals $\cF_v^{kt}R_k$ by 
Definition  \ref{G def}.
By \eqref{numDiscOkB} and Definition \ref{iklDef}, for any $a>i_{k,\ell}(v)$,
    $$
        \dim\cF_v^aR_k=\#\left(\Delta_{v,k}^{{a}/{k}}\right)\leq\#\left(\Delta_v^{{a}/{k}}\cap\ZZ^n/k\right)<\ell.
    $$
    By \eqref{jumping number} (note the maximum is attained),
    $
       j_{k,\ell}\left(v\right)<a.
    $
    Hence,
    $
       j_{k,\ell}\left(v\right)\leq i_{k,\ell}\left(v\right).
    $
\epf

A reverse inequality holds for divisorial valuations:

\blem
\label{ikljkl2ndLem}
Let $v\in\ValXdiv$. For $k\in\NN(L)$ and $\ell\in\{1,\ldots,d_k\},\;$
$
        j_{k,d_k+1-\ell}\left(v\right)\geq i_{k,D_k+1-\ell}\left(v\right).
$
\elem
\bpf
Follow the notation in the proof of Lemma \ref{sliceDivOkBodyLem}
for the Okounkov body associated to the flag \eqref{flagFEq}.
    Recall Definition \ref{iklDef}. By slight abuse we will use interchangeably the same
    notation for sequences and finite sets in $\RR$. The non-decreasing sequence $\{i_{k,D_k+1-\ell}(v)\}_{\ell=1}^{D_k}$ is equal to $p_1(k\Delta\cap\ZZ^n)$, i.e., the projection of all lattice points in $k\Delta$ onto their first coordinate. Since the sequence of jumping numbers $\{j_{k,d_k+1-\ell}(v)\}_{\ell=1}^{d_k}$ consists of the projection of (part of) the lattice points, namely of $k\Delta_k$ \eqref{DeltakEq}, it 
    is a subsequence of $\{i_{k,D_k+1-\ell}(v)\}_{\ell=1}^{D_k}$. In particular 
    (by the monotone arrangement \eqref{jumpingEq},\eqref{ijumpingEq}),
    $
        j_{k,d_k+1-\ell}\left(v\right)\geq i_{k,D_k+1-\ell}\left(v\right),
$ for 
$\ell\in\{1,\ldots,d_k\}.
    $
\epf

\subsection{Idealized expected vanishing}

So far we have associated various quantities to the discrete
Okounkov bodies $\Delta_k$. 
In the spirit of \S\ref{DiscreteOkBSubsec} and  Problem \ref{DkdkProb} 
 we associate analogues 
quantities to the idealized discrete Okounkov bodies $\Delta\cap \ZZ^n/k$:
i.e., quantities that roughly emulate the body being ``toric."

\begin{definition}\label{upper S}
    The {\it idealized expected vanishing order} is the function
    $\overline{S}_{k,m}:\ValX\ra[0,\infty)$,
    $$
        \overline{S}_{k,m}
        :=\frac{1}{km}\sum_{\ell=1}^{m}i_{k,\ell}.
    $$
\end{definition}

When $v\in\ValXfin$, $i_{k,\ell}(v)/k\in[0,\cS_0(v)]$, so also $\overline{S}_{k,m}(v)$
lies in that range.

\begin{example}
    Let $C$ be a smooth projective curve of genus $g$, and fix a point $P\in C$, yielding the flag $C\supset\{P\}$. Let $L$ be a line bundle over $C$ with $\deg L\geq 2g+1$, and $v=\ord_P$ \cite[Example 1.3]{LM09}. For a very general choice of $P$,
    $d_k=k\deg L+1-g>0$, so $\NN(L)=\NN$, and 
    $
        j_{k,\ell}\left(v\right)=d_k-\ell,
    $
    i.e.,
    $
        \left\{j_{k,\ell}\left(v\right)\right\}_{\ell=1}^{d_k}=\left\{k\deg L-g,k\deg L-g-1,\ldots,0\right\},
    $
    and 
    $
        S_{k,m}\left(v\right)=\deg L-({m-1})/{2k}-{g}/{k},
    $
    It follows that $\Delta=[0,\deg L]$ and $\Delta_v^t=[t,\deg L]$ for $0\leq t\leq \cS_0(v)=\deg L$ \cite[Example 1.14]{LM09}. By Definition \ref{upper S}, $D_k=k\deg L+1$, and 
    $
        i_{k,\ell}\left(v\right)=D_k-\ell,
    $
    i.e.,
    $
        \left\{i_{k,\ell}\left(v\right)\right\}_{\ell=1}^{D_k}=\left\{k\deg L,k\deg L-1,\ldots,0\right\},
    $
    and
    $
        \overline{S}_{k,m}\left(v\right)=\deg L-({m-1})/{2k},
    $
\end{example}

By Lemma  \ref{ikljkl1stLem},
\beq\lb{S leq S upper}
S_{k,m}\leq\overline{S}_{k,m}.
\eeq
A simple and key observation is that a weak reverse estimate holds on $\ValXdiv$.

\blem
\lb{weakSkmIneq}
On $\ValXdiv$,   
\begin{equation}\label{first est}
S_{k,m_k}\geq\frac{1}{km_k}\sum_{\ell=1+D_k-d_k}^{m_k+D_k-d_k}i_{k,\ell}.
\end{equation}
\elem
\bpf
By Definition \ref{SkmDef}, \eqref{jumpingEq}, and Lemma \ref{ikljkl2ndLem},
$$
\baligned
S_{k,m_k}
&=
\frac{1}{km_k}\sum_{\ell=1}^{m_k}j_{k,\ell}
=
\frac{1}{km_k}\sum_{\ell=d_k+1-m_k}^{d_k}j_{k,d_k+1-\ell}
\cr
&\ge
\frac{1}{km_k}\sum_{\ell=d_k+1-m_k}^{d_k}i_{k,D_k+1-\ell}
=
\frac{1}{km_k}\sum_{\ell=1+D_k-d_k}^{m_k+D_k-d_k}i_{k,\ell}
.
\ealigned
$$
\epf
Lemma \ref{weakSkmIneq} will be improved to a strong reverse inequality for \eqref{S leq S upper}
in \S\ref{SkmnoncollapsingSubSec}, under the non-collapsing assumption $\tau>0$.
In the collapsing case, a microlocal estimate will be combined with Lemma \ref{weakSkmIneq} instead.

\blem
\lb{barSmkdecreasingRem}
For fixed $k\in\NN(L)$, $\overline{S}_{k,m}$ is non-increasing in $m$.
\elem
\bpf
Recall \eqref{GvDef}.
Write $\Delta\cap\ZZ^n/k=\{x_1,\ldots,x_{D_k}\}$ with $G_v(x_1)\geq\cdots\geq G_v(x_{D_k})$. Then $i_{k,\ell}\left(v\right)=kG_v(x_\ell)$ and 
$\overline{S}_{k,m}=\frac{1}{m}\sum_{\ell=1}^{m}G_v(x_\ell)$. The conclusion follows from
monotonicity~of~$G_v(x_\ell)$. 
\epf

\subsection{Partial discrete Okounkov bodies}

The doubly-indexed $(k,m_k)$ setting of this article with $m_k\le d_k$ means
that at each level $k$ we are considering the 
subset of {\it partial discrete Okounkov bodies} of cardinality $m_k$ (there are
$\binom{d_k}{m_k}$ such).
By Lemma \ref{compatible basis lem}, there will be a distinguished subset of this 
``discrete
Grassmannian"
that compute $S_{k,m}$. These are the so-called compatible $m$-bases. It 
turns out that they are closely related to the 1-parameter family of Okounkov
bodies $\Delta_v^{Q_v(\tau)}$ indexed by the volume quantile $Q_v(\tau)$
and $\tau\in[0,1]$. The following result connects many of the notions
introduced in this article. It is a generalization of Boucksom--Chen's Theorem 
\ref{mu}. Since it is not needed for our main
applications in this article, we postpone its proof to a sequel.

\bthm\lb{empiricalQThm}
For $k\in\NN(L)$ and $m_k\in\{1,\ldots,d_k\}$ with $\lim_{k\to\infty}{m_k}/{d_k}=\tau\in\left[\mu_v(\{\cS_0(v)\}),1\right]$, choose a linearly independent family $\{s_1^{(k)},\ldots,s_{m_k}^{(k)}\}$ compatible with $\cF_v$ (Recall Definition \ref{compat def}). The sequence of empirical measures on $\RR^n$,
$
    \frac{1}{m_k}\sum_{i=1}^{m_k}\delta_{\frac{1}{k}\nu\left(s_i^{\left(k\right)}\right)},
$
converges weakly to 
$
\boldsymbol{1}_{\Delta_v^{Q_v(\tau)}}dx/{\big| \Delta_v^{Q_v(\tau)} \big|}
$.
\ethm

This motivates the following definition.

\bdefn
The {\it $\tau$-partial idealized discrete Okounkov body} is $\Delta_v^{Q_v(\tau)}\cap\ZZ^n/k$.
\edefn

Problem \ref{DkdkQuantileProb} amounts to estimating $|M_k-m_k|$ where
    \begin{equation}\lb{MkEq}
        M_k:=\#\left(\Delta_v^{Q_v(\tau)}\cap\ZZ^n/k\right).
    \end{equation}
There is an initial estimate in terms of the idealized expected vanishing.

\blem
\lb{mkMkLemma}
    $
    \min\left\{1,
        {M_k}/{m_k}\right\}\overline{S}_{k,M_k}\left(v\right)
    \le
        \overline{S}_{k,m_k}\left(v\right)\leq
        \max\left\{1,
        {M_k}/{m_k}\right\}\overline{S}_{k,M_k}\left(v\right).
    $
\elem

\bpf
If $m_k\leq M_k$, by Definition \ref{upper S},
    $$
        \overline{S}_{k,m_k}\left(v\right)\leq\frac{1}{km_k}\sum_{\ell=1}^{M_k}i_{k,\ell}\left(v\right)=\frac{M_k}{m_k}\overline{S}_{k,M_k}\left(v\right).
    $$
    If $m_k\geq M_k$, then $\overline{S}_{k,m_k}(v)\leq\overline{S}_{k,M_k}(v)$ by Lemma  \ref{barSmkdecreasingRem}, giving
    the upper bound. 
    The~lower~bound~is~similar.
\epf

\begin{figure}[ht]
    \centering
    \begin{tikzpicture}
        \draw[->](-.5,0)--(11.5,0)node[below]{$\RR^n$};
        \draw[->](.5,-1)--(.5,5)node[left]{$t$};
        \draw[domain=2:6]plot(\x,{(\x+2)^2*(10-\x)/32-4});
        \draw(6,4)--(8,4)node[above right]{$G_v(x)$};
        \draw[domain=8:10]plot(\x,{(\x-4)^2*(10-\x)/8});
        \draw[dashed](.5,4)node[above left]{\scriptsize$i_{k,1}/k=i_{k,2}/k$}--(6,4);
        \draw[<-](.5,4)--(.4,4.2);
        \draw[dashed](6.5,0)--(6.5,4);
        \draw[dashed](7.5,0)--(7.5,4);
        \draw[dashed](5.5,0)--(5.5,{(5.5+2)^2*(10-5.5)/32-4})--(.5,{(5.5+2)^2*(10-5.5)/32-4})node[left]{\scriptsize$i_{k,3}/k$};
        \draw[dashed](8.5,0)--(8.5,{(8.5-4)^2*(10-8.5)/8})--(.5,{(8.5-4)^2*(10-8.5)/8})node[below left]{\scriptsize$i_{k,4}/k$};
        \draw[<-](.5,{(8.5-4)^2*(10-8.5)/8})--(.4,{(8.5-4)^2*(10-8.5)/8-.2});
        \draw[dashed](4.5,0)--(4.5,{(4.5+2)^2*(10-4.5)/32-4})--(.5,{(4.5+2)^2*(10-4.5)/32-4})node[left]{\scriptsize$i_{k,5}/k$};
        \draw[dashed](3.5,0)--(3.5,{(3.5+2)^2*(10-3.5)/32-4})--(.5,{(3.5+2)^2*(10-3.5)/32-4})node[left]{\scriptsize$i_{k,6}/k$};
        \draw[dashed](9.5,0)--(9.5,{(9.5-4)^2*(10-9.5)/8})--(.5,{(9.5-4)^2*(10-9.5)/8})node[left]{\scriptsize$i_{k,7}/k$};
        \draw[dashed](2.5,0)--(2.5,{(2.5+2)^2*(10-2.5)/32-4})--(.5,{(2.5+2)^2*(10-2.5)/32-4})node[left]{\scriptsize$i_{k,8}/k$};
        \draw[line width=1pt](5.5,0)--(5.5,3)node[right]{\scriptsize$j_{k,1}/k$};
        \draw[line width=1pt](6.5,0)--(6.5,3)node[right]{\scriptsize$j_{k,2}/k$};
        \draw[line width=1pt](7.5,0)--(7.5,1)node[right]{\scriptsize$j_{k,3}/k$};
        \filldraw(3.5,0)circle(1pt)node[above right]{\scriptsize$j_{k,4}/k$};
        \draw[->](-.5,-1.62)--(11.5,-1.62)node[below]{$\RR^n$};
        \draw[line width=2pt](2,-1.62)--(10,-1.62)node[below]{$\Delta$};
        \filldraw[fill=white](2.5,-1.62)circle(2.5pt);
        \filldraw(3.5,-1.62)circle(2.5pt);
        \filldraw[fill=white](4.5,-1.62)circle(2.5pt);
        \filldraw(5.5,-1.62)circle(2.5pt);
        \filldraw(6.5,-1.62)circle(2.5pt);
        \filldraw(7.5,-1.62)node[above right]{$\Delta_k$}circle(2.5pt);
        \filldraw[fill=white](8.5,-1.62)circle(2.5pt);
        \filldraw[fill=white](9.5,-1.62)node[above right]{$\Delta\cap\ZZ^n/k$}circle(2.5pt);
    \end{tikzpicture}
    \caption{The graph of $G_v:\RR^n\ra\RR$ and the jumping numbers $j_{k,\ell}(v)$ and idealized jumping numbers $i_{k,\ell}(v)$ in an example with $d_k=4,D_k=8$. 
    Here, $j_{k,4}=0$ and
    $\cF_v^{j_{k,4}}R_k=H^0(X,kL)\cong \CC^4$,
    $\cF_v^{j_{k,3}}R_k\cong \CC^3$,
    and $j_{k,2}=j_{k,1}$ so $\cF_v^{j_{k,1}}R_k\cong \CC^2$.
    Each super-level set $G_v^{-1}([t,\infty))\subset\RR^n$ is the Okounkov body $\Delta_v^t$. The Okounkov body $\Delta\subset\RR^n$ and $\Delta_k\subset\Delta\cap\ZZ^n/k\subset\Delta$ are drawn below for comparison
    (the former is represented by the solid dots).}
\end{figure}
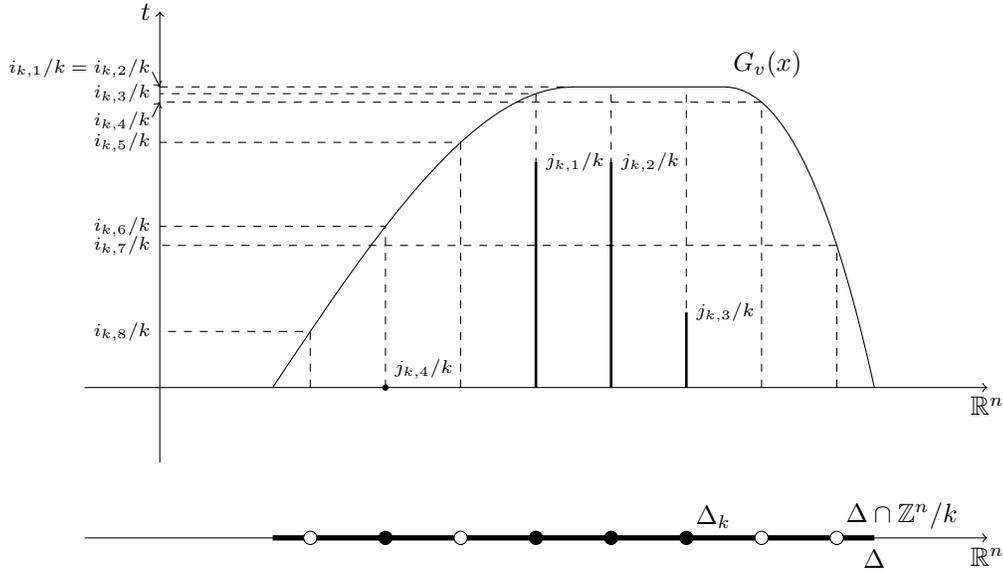

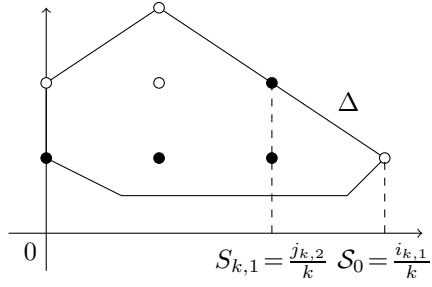
\begin{figure}[ht]
    \centering
    \begin{tikzpicture}
        \draw[->](-.5,0)--(5,0);
        \draw[->](0,-.5)--(0,3);
        \draw(0,0)node[below left]{$0$};
        \draw(0,2)--(1.5,3)--(3,2)--(4.5,1)node[midway,above right]{$\Delta$}--(4,.5)--(1,0.5)--(0,1)--cycle;
        \filldraw(0,1)circle(2pt);
        \filldraw[fill=white](0,2)circle(2pt);
        \filldraw(1.5,1)circle(2pt);
        \filldraw[fill=white](1.5,2)circle(2pt);
        \filldraw[fill=white](1.5,3)circle(2pt);
        \filldraw(3,1)circle(2pt);
        \filldraw(3,2)circle(2pt);
        \filldraw[fill=white](4.5,1)circle(2pt);
        \draw[dashed](3,0)node[below]{$S_{k,1}\!=\!\frac{j_{k,2}}k$}--(3,2);
        \draw[dashed](4.5,0)node[below]{$\cS_0\! =\! \frac{i_{k,1}}k$}--(4.5,1);
    \end{tikzpicture}
    \caption{An example of a discrete Okounkov body associated to a
    divisorial valuation with $d_k=4$ and $j_{k,1}=j_{k,2}>j_{k,3}>j_{k,4}=0$, and $D_k=8$ and $i_{k,1}>i_{k,2}=i_{k,3}>i_{k,4}=i_{k,5}=i_{k,6}>i_{k,7}=i_{k,8}=0$.
    In this example, $j_{k,2}=i_{k,2}$, $j_{k,3}=i_{k,4}$, $j_{k,4}=i_{k,7}$.}
\lb{FigDkdkij}
\end{figure}

\section{Joint asymptotics of
Grassmannian thresholds: lower bound}\label{lb section}

In this section we construct the invariants $\boldsymbol\delta_\tau$ (Corollary \ref{deltadoubleasymCor}) generalizing $\alpha$ and $\delta$. 
A valuative characterization of $\delta_{k,m}$ is stated in \S \ref{valuative}. 
The joint convergence in $(k,m)$ (but not asymptotics; those are derived later) of the end-point cases $\tau=0,1$ are the subject of \S\ref{End-point} involving the degenerate ($m_k/d_k\ra0$) and full-measure case ($m_k/d_k\ra1$). In \S \ref{general double} we construct the invariants $\cS_\tau$ generalizing $\cS_0$ and $\cS_1$. In \S \ref{lattice point estimates} asymptotic estimates on lattice points in a convex body are derived. These are applied in \S \ref{the concave transform}, culminating in Corollary \ref{deltadoubleasymCor}
on well-definedness of the invariants $\boldsymbol\delta_\tau$.

\subsection{Valuative characterization of Grassmannian stability thresholds}\label{valuative}

First, a straightforward generalization of the valuative
characterization of $\alpha_k$ ($m=1$) and $\delta_k$ ($m=d_k$)
 \cite[Propositions 4.1 and 4.3]{BJ20}.

\begin{lemma}\lb{deltakminfProp}
    For $k\in\NN(L)$ and $m\in\{1,\ldots,d_k\}$, 
    $
    \disp
        \delta_{k,m}=\inf_{v\in\ValXdiv}{A}/{S_{k,m}}=\inf_{v\in\ValXfin}{A}/{S_{k,m}}.
    $
\end{lemma}
\begin{proof}
    By \eqref{delta def} and Theorem \ref{lct def},
    \begin{align*}
        \delta_{k,m}&=\inf_{\substack{\left\{s_\ell\right\}_{\ell=1}^m\subset R_k\\\mbox{\smlsev linearly independent}}}km\lct\left(\sum_{\ell=1}^m\left(s_\ell\right)\right)
        =\inf_{\substack{\left\{s_\ell\right\}_{\ell=1}^m\subset R_k\\\mbox{\smlsev linearly independent}}}\inf_{v\in\ValXdiv}km\frac{A\left(v\right)}{v\left(\sum\limits_{\ell=1}^m\left(s_\ell\right)\right)}\\
        &=\inf_{v\in\ValXdiv}km\frac{A\left(v\right)}{\sup\limits_{\substack{\left\{s_\ell\right\}_{\ell=1}^m\subset R_k\\\mbox{\smlsev linearly independent}}}v\left(\sum\limits_{\ell=1}^m\left(s_\ell\right)\right)}
        =\inf_{v\in\ValXdiv}\frac{A\left(v\right)}{S_{k,m}\left(v\right)}.
    \end{align*}
    The same argument applies to $v\in\ValXfin$.
\end{proof}

\begin{lemma}\label{birlem}
    For $k\in\NN(L)$, $\alpha_k=\delta_{k,1}$ is computed by a divisorial valuation, i.e.,
    $$
        \delta_{k,1}=
        \frac{A(v)}{S_{k,1}(v)} \q \hbox{for some $v\in\ValXdiv$}.
    $$
\end{lemma}
\begin{proof}
    By \eqref{delta def},
    $
        \delta_{k,1}=\inf_{s\in R_k}k\lct\left(s\right).
    $
    By the lower semi-continuity of complex singularity exponents \cite[Theorem 0.2(3)]{DK01}, we can find $s\in R_k$ such that
        $\delta_{k,1}=k\lct\left(s\right).$
    By Theorem \ref{lct def}, there is $v\in\ValXdiv$ such that
        $\lct\left(s\right)={A\left(v\right)}/{v\left(s\right)}.$
    By Lemma \ref{compatible basis lem},
        $v\left(s\right)\leq kS_{k,1}\left(v\right).$
    Combining these four equations,
    $
        \frac{A\left(v\right)}{S_{k,1}\left(v\right)}\leq\delta_{k,1}.
    $
By Lemma \ref{deltakminfProp}
    equality is achieved.
\end{proof}

\subsection{End-point convergence cases}\label{End-point}

In this subsection we prove Theorem \ref{sub main thm}.

\begin{definition}\label{old notion}
    The {\it global log canonical threshold ($\alpha$-invariant)}
    $
        \alpha=\hbox{\rm glct}
        :=\inf_{\NN(L)}\hbox{\rm glct}_k,
    $
    where
    $
        \alpha_k=\hbox{\rm glct}_k
        :=k\inf_{s\in R_k}\lct\left(s\right).
    $
    The {\it basis log canonical threshold ($\delta$-invariant, {\rm or}  stability threshold)}
    $
        \delta:=\lim_{k\to\infty}\delta_k
    $
    (the limit exists \cite[Theorem A]{BJ20}), where
    $
        \delta_k=kd_k\inf_{\substack{\left\{s_\ell\right\}_{\ell=1}^{d_k}\\\mbox{\smlsev basis of }R_k}}\lct\left(\sum_{\ell=1}^{d_k}\left(s\right)\right).
    $
\end{definition}

The next result generalizes a theorem of Demailly \cite[Theorem A.3]{CS08}~and~Shi~\cite[Theorem 2.2]{Shi10}.
\begin{theorem}\label{sub main thm}
    Suppose $m_k\in\{1,\ldots,d_k\}$ for $k\in\NN(L)$.
If $m_k=o(k^n)$, then
    $
        \lim_{k\to\infty}\delta_{k,m_k}=\inf_{k}\delta_{k,m_k}=\alpha.
    $
    \newline
    If 
    $m_k=d_k-o(k^n)$,
    then $
        \lim_{k\to\infty}\delta_{k,m_k}=\delta.
    $
\end{theorem}
The proof relies on Proposition \ref{Strong version} and Lemma
\ref{tau=max} proved below.
\begin{proof}
    For a sequence $\{m_k\}_{k\in\NN(L)}$ such that $1\leq m_k=o(k^n)$, we rely on to Proposition \ref{Strong version} and its proof. Namely, Lemma \ref{deltakminfProp} implies that for any fixed $v_0\in\ValXfin$, 
    $$
        \limsup_{k\to\infty}\delta_{k,m_k}=\limsup_{k\to\infty}\inf_{v\in\ValX}\frac{A\left(v\right)}{S_{k,m_k}\left(v\right)}\leq\limsup_{k\to\infty}\frac{A\left(v_0\right)}{S_{k,m_k}\left(v_0\right)}=\frac{A\left(v_0\right)}{\lim\limits_{k\to\infty}S_{k,m_k}\left(v_0\right)}=\frac{A\left(v_0\right)}{\cS_0\left(v_0\right)}.
    $$
    Since $v_0$ is arbitrary, by Remark \ref{Fkt rmk} and Lemma \ref{deltakminfProp} (see also \cite[Theorem C]{BJ20}),
    $$
        \limsup_{k\to\infty}\delta_{k,m_k}\leq\inf_{v\in\ValXfin}\frac{A\left(v\right)}{\cS_0\left(v\right)}=\alpha.
    $$
    On the other hand, by Lemma \ref{deltakminfProp}, \eqref{m monotonicity}, and Remark \ref{Fkt rmk},
    $
        \delta_{k,m_k}=\inf_{v\in\ValXfin}\frac{A\left(v\right)}{S_{k,m_k}\left(v\right)}\geq\inf_{v\in\ValXfin}\frac{A\left(v\right)}{\cS_0\left(v\right)}=\alpha.
    $
    In conclusion,
    $
        \limsup_{k\to\infty}\delta_{k,m_k}\leq\alpha\leq\inf_k\delta_{k,m_k}\leq\liminf_{k\to\infty}\delta_{k,m_k}.
    $

    For a sequence $\{m_k\}_{k\in\NN(L)}$ such that $0\leq d_k-m_k=o(k^n)$, by \eqref{S squeeze},
    $
        \frac{m_k}{d_k}\delta_k\leq\delta_{k,m_k}\leq\delta_k.
    $
    By Definition \ref{old notion} this completes the proof.
\end{proof}

\begin{lemma}\label{tau=max}
    For any sequence $\{m_k\}_{k\in\NN(L)}$ such that $0\leq d_k-m_k=o(k^n)$,
    $\disp
        \lim_{k\to\infty}S_{k,m_k}=\cS_1.
    $
\end{lemma}
\begin{proof}
    By \eqref{m monotonicity},
    \begin{equation}\label{S squeeze}
        S_{k,d_k}\leq S_{k,m_k}=\frac{1}{km_k}\sum_{\ell=1}^{m_k}j_{k,\ell}\leq\frac{1}{km_k}\sum_{\ell=1}^{d_k}j_{k,\ell}=\frac{d_k}{m_k}S_{k,d_k}.
    \end{equation}
    Taking the limit in $k$ completes the proof.
\end{proof}

\begin{proposition}\label{Strong version}

    For any sequence $\{m_k\}_{k\in\NN(L)}$ such that $1\leq m_k=o(k^n)$,
    $
        \lim_{k\to\infty}S_{k,m_k}=\sup_kS_{k,m_k}=\cS_0.
    $
\end{proposition}
\begin{proof}
    For $\varepsilon\in(0,1)$, there is $k_1$ so that for $\NN(L)\ni k\geq k_1$,
    $
 \disp       m_k\leq\frac{\varepsilon^n\VolL}{2n!}k^n.
    $
    Fix $v\in\ValXfin$.
    By~Claim~\ref{conv lem 1},
    $$
        \lim_{k\to\infty}\frac{n!}{k^n}\dim\cF_v^{k\left(1-\varepsilon\right)\cS_0\left(v\right)}R_k=\Vol V_{v,\bullet}^{\left(1-\varepsilon\right)\cS_0\left(v\right)}\geq\varepsilon^n\VolL.
    $$
    Hence there is $k_2$ such that for $k\geq k_2$,
    $$
        \dim\cF_v^{k\left(1-\varepsilon\right)\cS_0\left(v\right)}R_k\geq\frac{\varepsilon^n\VolL}{2n!}k^n.
    $$
    (In particular, $k\in\NN(L)$ whenever $k\geq k_2$.) By \eqref{jumping number}, for $k\geq\max\{k_1,k_2\}$,
    $
       j_{k,m_k}\left(v\right)\geq k\left(1-\varepsilon\right)\cS_0\left(v\right).
    $
    By \eqref{m monotonicity},
    $$
        S_{k,m_k}\left(v\right)=\frac{1}{km_k}\sum_{\ell=1}^{m_k}j_{k,\ell}\left(v\right)\geq\frac{1}{k}j_{k,m_k}\left(v\right)\geq\left(1-\varepsilon\right)\cS_0\left(v\right).
    $$
    Since $\varepsilon$ is arbitrary,
    \begin{equation}\label{liminf geq}
        \liminf_{k\to\infty}S_{k,m_k}\left(v\right)\geq\cS_0\left(v\right).
    \end{equation}
    On the other hand, by \eqref{def T} and \eqref{m monotonicity},
    \begin{equation}\label{limsup leq}
        S_{k,m_k}\left(v\right)\leq S_{k,1}\left(v\right)
        \leq\cS_0\left(v\right).
    \end{equation}
    Combining \eqref{liminf geq}--\eqref{limsup leq},
    $
        \limsup_{k\to\infty}S_{k,m_k}\left(v\right)\leq\sup_kS_{k,m_k}\left(v\right)\leq\cS_0\left(v\right)\leq\liminf_{k\to\infty}S_{k,m_k}\left(v\right).
    $
\end{proof}

\subsection{Joint monotonicity}
This subsection collects basic monotonicity properties of $S_{k,m}$ and $\delta_{k,m}$.
\begin{lemma}\label{superadditive}
    For $k,k'\in\NN(L)$,
    $
        \left(k+k'\right)S_{k+k',m}\geq kS_{k,m}+k'S_{k',1}.
    $
\end{lemma}
\begin{proof}
    For $v\in\ValX$, let $s\in H^0\left(X,k'L\right)$ be such that $v(s)=k'S_{k',1}\left(v\right)$ (recall Lemma \ref{compatible basis lem}). The map
    $$
        \cF_v^\lambda R_k\xrightarrow{\otimes s}\cF_v^{\lambda+k'S_{k',1}\left(v\right)} R_{k+k'}
    $$
    is an injection. In particular,
    \begin{align*}
        j_{k,\ell}\left(v\right)&=\max\left\{\lambda\,:\,\dim\cF_v^\lambda R_k\geq\ell\right\}
        \leq\max\left\{\lambda\,:\,\dim\cF_v^{\lambda+k'S_{k',1}\left(v\right)}R_{k+k'}\geq\ell\right\}\\
        &=\max\left\{\lambda\,:\,\dim\cF_v^\lambda R_{k+k'}\geq\ell\right\}-k'S_{k',1}\left(v\right)
        =j_{k+k',\ell}\left(v\right)-k'S_{k',1}\left(v\right).
    \end{align*}
    It follows that
    $$
        S_{k,m}=\frac{1}{km}\sum_{\ell=1}^mj_{k,\ell}\leq\frac{1}{km}\sum_{\ell=1}^mj_{k+k',\ell}-\frac{k'}{k}S_{k',1}=\frac{k+k'}{k}S_{k+k',m}-\frac{k'}{k}S_{k',1}.
    $$
\end{proof}
\begin{lemma}\label{multiple lem}
    For $k\in\NN(L)$, $m\in\{1,\ldots,d_k\}$, and $\ell\in\NN$,
    $
        S_{k,m}\leq S_{\ell k,m}.
    $
\end{lemma}
\begin{proof}
    Let $k_\ell=(\ell-1)k$. Notice that for any $s\in R_k$,
    $s^{\otimes(\ell-1)}\in R_{k_\ell}$. In particular, for $v\in\ValX$,
    $$
        \max_{s\in R_{k_\ell}}v\left(s\right)\geq\left(\ell-1\right)\max_{s\in R_k}v\left(s\right).
    $$
    By \eqref{m monotonicity},
    $$
        S_{k,m}\left(v\right)\leq S_{k,1}\left(v\right)=\frac{1}{k}\max_{s\in R_k}v\left(s\right)\leq\frac{1}{k_\ell}\max_{s\in R_{k_\ell}}v\left(s\right)=S_{k_\ell,1}\left(v\right).
    $$
    By Lemma \ref{superadditive} and \eqref{m monotonicity},
    $$
        S_{k,m}\left(v\right)\leq\frac{k+k_\ell}{k}S_{k+k_\ell,m}\left(v\right)-\frac{k_\ell}{k}S_{k_\ell,1}\left(v\right)\leq\frac{k+k_\ell}{k}S_{k+k_\ell,m}\left(v\right)-\frac{k_\ell}{k}S_{k,m}\left(v\right),
    $$
    i.e.,
    $
        S_{k,m}\left(v\right)\leq S_{k+k_\ell,m}\left(v\right).
    $
\end{proof}

Lemma \ref{multiple lem} together with 
Lemma \ref{deltakminfProp} and \eqref{m monotonicity}
imply:

\begin{corollary}
Let $k\in\NN(L)$, $\ell\in\NN$, $m,m'\in\{1,\ldots,d_k\}$ with $m\leq m'$.
Then
    $
        \delta_{k,m}\leq\delta_{k,m'},
    $
    and
    $
        \delta_{\ell k,m}\leq\delta_{k,m}.
    $
\end{corollary}
\begin{proof}
    By Lemma \ref{deltakminfProp} and \eqref{m monotonicity},
    $
    \disp
        \delta_{k,m}=\inf_{v\in\ValX}\frac{A\left(v\right)}{S_{k,m}\left(v\right)}\leq\inf_{v\in\ValX}\frac{A\left(v\right)}{S_{k,m'}\left(v\right)}=\delta_{k,m'}.
    $
    By Lemma \ref{multiple lem} and Lemma \ref{deltakminfProp},
    $
    \disp    \delta_{\ell k,m}=\inf_{v\in\ValX}\frac{A\left(v\right)}{S_{\ell k,m}\left(v\right)}\leq\inf_{v\in\ValX}\frac{A\left(v\right)}{S_{k,m}\left(v\right)}=\delta_{k,m}.
    $
\end{proof}

\begin{corollary}
\lb{BirkarCor}
    Recall Definition \ref{old notion}. Suppose $\alpha=\alpha_\ell$ for some $\ell\in\NN(L)$. Then
    $
        0\leq\alpha_k-\alpha=O\left(1/k\right), k\in\NN(L).
    $
\end{corollary}
\begin{proof}
    By Lemma \ref{birlem}, there is $w\in\ValXdiv$ such that
    $
        \alpha_\ell=\frac{A\left(w\right)}{S_{\ell,1}\left(w\right)}.
    $
    By Remark \ref{Fkt rmk} and Lemma \ref{deltakminfProp},
    $
        \alpha=\inf_k\alpha_k=\inf_{v\in\ValXdiv}\frac{A\left(v\right)}{\cS_0\left(v\right)}.
    $
    Recall \eqref{def T}. Combining,
    $$
        \alpha_\ell=\frac{A\left(w\right)}{S_{\ell,1}\left(w\right)}\geq\frac{A\left(w\right)}{\cS_0\left(w\right)}\geq\inf_{v\in\ValXdiv}\frac{A\left(v\right)}{\cS_0\left(v\right)}=\alpha.
    $$
    Since $\alpha_\ell=\alpha$,
    \begin{equation}\label{w S0}
        S_{\ell,1}\left(w\right)=\cS_0\left(w\right),
    \end{equation}
    and
    \begin{equation}\label{w alpha}
        \alpha=\frac{A\left(w\right)}{\cS_0\left(w\right)}.
    \end{equation}
    Pick $\ell_0$ sufficiently large so that for any $k\geq\ell_0$, $k\in\NN(L)$. Using Euclidean division, for $k\geq\ell_0$,
    $
        k-\ell_0=q\left(k\right)\ell+r\left(k\right),
    $
    where $q(k)\geq0$ and $0\leq r(k)<\ell$. By Remark \ref{Fkt rmk}, Lemma \ref{multiple lem} (with $m=1$), and \eqref{w S0},
    \begin{align*}
        kS_{k,1}\left(w\right)&\geq q\left(k\right)\ell S_{q\left(k\right)\ell,1}\left(w\right)+\left(\ell_0+r\left(k\right)\right)S_{\ell_0+r\left(k\right),1}\left(w\right)\\
        &\geq q\left(k\right)\ell S_{\ell,1}\left(w\right)+\left(\ell_0+r\left(k\right)\right)S_{\ell_0+r\left(k\right),1}\left(w\right)\\
        &\geq q\left(k\right)\ell S_{\ell,1}\left(w\right)
        =\left(k-r\left(k\right)-\ell_0\right)\cS_0\left(w\right)
        >\left(k-\ell-\ell_0\right)\cS_0\left(w\right).
    \end{align*}
    By \eqref{w alpha}, for $k>\ell+\ell_0$,
    $$
        \alpha_k\leq\frac{A\left(w\right)}{S_{k,1}\left(w\right)}<\frac{k}{k-\ell-\ell_0}\frac{A\left(w\right)}{\cS_0\left(w\right)}=\frac{k}{k-\ell-\ell_0}\alpha,
    $$
    i.e.,
    $\disp
        \alpha_k-\alpha<\frac{\ell+\ell_0}{k-\ell-\ell_0}\alpha=O\left(k^{-1}\right).
    $
\end{proof}

\subsection{The ccdf, quantile, and tail distribution associated to a valuation}

In this subsection we introduce the ccdf, quantile, and tail distribution associated to a valuation.

\bdefn
\lb{ccdfDef}
Let $v\in \ValXfin$.
The {\it complementary cumulative distribution function (ccdf)}
of $v$ is~(recall~\eqref{muvEq})
$$
F_v(t):=\mu_v\big([t,\cS_0(v)]\big), \q t\in[0,\cS_0(v)].
$$
The {\it quantile function associated to $v$} is its
generalized inverse
$$
Q_v:= \sup F_v^{-1}, \q \hbox{\ on } [0,1],
$$
where $F_v^{-1}$ is a set-valued map and 
$Q_v(\tau)$
equals the supremum of the set $F_v^{-1}(\tau)\subset\RR$. 
The {\it $\tau$-tail distribution associated to $v$} is
$\boldsymbol{1}_{[Q_v(\tau),\cS_0(v)]}\mu_v$.
\edefn

The barycenter of a measure $\nu$ on $\RR$ is
$
b\mskip1mu(\nu):=\int_\RR td\nu(t)\big/\int_\RR d\nu.
$

\begin{proposition}\label{S_tau def}
    Suppose $m_k\in\{1,\ldots,d_k\}$ for $k\in\NN(L)$,
    and
    $
        \lim_{k\to\infty}{m_k}/{d_k}=\tau\in\left[0,1\right].
    $
    The limit
    $$
        \cS_\tau:=\lim_{k\to\infty}S_{k,m_k}
    $$
    exists on $\ValXfin$ and is independent of the choice of $\{m_k\}_{k\in\NN(L)}$. Moreover, $\cS_\tau$ 
    is the barycenter of the $\tau$-tail distribution of $v$  (recall (\ref{muvEq})
    and Definition \ref{ccdfDef}),
    $$
\cS_\tau=b\mskip1mu\left(\boldsymbol{1}_{[Q_v(\tau),\cS_0(v)]}\mu_v\right).
    $$
    In particular, $\cS_\tau$ is non-increasing with respect to $\tau$.
\end{proposition}

The key for the proof of Proposition \ref{S_tau def} is to associate quantum analogues of the ccdf and tail distribution to $v$
and prove some asymptotic estimates for them that guarantee that the quantum barycenters converge to the classical ones.

\begin{definition}\label{iota defn}
    Let $v\in \ValXfin$,  $k\in\NN(L)$ and $t\in[0,S_{k,1}(v)]$. 
    The {\it $k$-th quantum ccdf of $v$}~is~(recall~\eqref{muvkEq})
$$
F_{v,k}(t):=\mu_{v,k}\big([t,S_{k,1}(v)]\big)\in\{j/d_k\}_{j=0}^{d_k}, \q t\in[0,S_{k,1}(v)].
$$
The {\it $k$-th quantum quantile function of $v$} is its inverse
$$
Q_{v,k}(\tau):= \sup F_{v,k}^{-1}\left(\lfloor \tau d_k\rfloor/d_k\right), \q \hbox{\ on } [0,1].
$$
The {\it $k$-th quantum $\tau$-tail distribution associated to $v$} is
$\boldsymbol{1}_{[Q_{v,k}(\tau),S_{k,1}(v)]}\mu_{v,k}$.
\end{definition}
Note that 
by Theorem \ref{mu} and Definition \ref{ccdfDef},
$
\lim_kF_{v,k}=F_v,
$
and hence also
$
\lim_kQ_{v,k}=Q_v.
$
In fact, there is a useful expression for the quantum ccdf:

\blem
    Let $v\in \ValXfin$,  $k\in\NN(L)$ and $t\in[0,S_{k,1}(v)]$. 
Then, $\disp F_{v,k}(t)=\frac1{d_k}\dim\cF_v^{kt}R_k$.
\elem

Set
\beq\lb{iotavkEq}
        \iota_{v,k}\left(t\right):=\left\{1\leq\ell\leq d_k\,:\,j_{k,\ell}\left(v\right)\geq kt\right\}.
\eeq

\bpf
First, we record a useful expression for $\iota_{v,k}$ \eqref{iotavkEq}.
By \eqref{jumping number}, $j_{k,\ell}(v)\geq kt$ if and only if $\dim\cF_v^{kt}R_k\geq\ell$. Therefore (recall \eqref{iotavkEq}),
\begin{align}
    \iota_{v,k}\left(t\right)&=\left\{1\leq\ell\leq d_k\,:\,j_{k,\ell}\left(v\right)\geq kt\right\}
    =\left\{1\leq\ell\leq d_k\,:\,\dim\cF_v^{kt}R_k\geq\ell\right\}\nonumber\\
    &=\left\{1,\ldots,\dim\cF_v^{kt}R_k\right\}.
    \label{iota def}
\end{align}
Now, note that by Definition \ref{ccdfDef} and \eqref{iotavkEq},
    \beq\lb{FvkiotaEq}
    F_{v,k}=\frac1{d_k}\max\iota_{v,k}=\frac1{d_k}\#\left(\iota_{v,k}\right),
    \eeq
and the lemma follows.
\epf

Towards the proof of Proposition \ref{S_tau def} we prove first a 
quantum analogue:

\begin{lemma}
    For $k\in\NN(L), m_k\in d_kF_{v,k}([0,S_{k,1}(v)])$, 
    $
S_{k,m_k}\left(v\right)
        =b\mskip1mu
        \left( \boldsymbol{1}_{\left[Q_{v,k}(m_k/d_k),S_{k,1}\left(v\right)\right]}\mu_{v,k}\right).
$
\end{lemma}

\begin{proof}
    By \eqref{FvkiotaEq}, 
    the statement is equivalent to showing that
\beq\lb{SkmccdfequivEq}
        S_{k,\#\left(\iota_{v,k}\left(s\right)\right)}\left(v\right)
        =b\mskip1mu
        \left( \boldsymbol{1}_{\left[s,\cS_0\left(v\right)\right]}\mu_{v,k}\right),
     \q \hbox{\ for    $s\in[0,S_{k,1}(v)]\subset [0,\cS_0(v)]$}.
    \eeq
   By
    \eqref{muvkEq} and \eqref{iotavkEq},
    $$
        \int_{\left[s,\cS_0\left(v\right)\right]}td\mu_{v,k}=\frac{1}{kd_k}\sum_{\ell\in\iota_{v,k}\left(s\right)}j_{k,\ell}\left(v\right),
    $$
    and by \eqref{FvkiotaEq}
    $$
        \int_{\left[s,\cS_0\left(v\right)\right]}d\mu_{v,k}=
        F_{v,k}(s)=\frac{1}{d_k}\#\left(\iota_{v,k}\left(s\right)\right).
    $$
Thus,
    \begin{equation}\label{iota}
        S_{k,\#\left(\iota_{v,k}\left(s\right)\right)}\left(v\right)=\frac{1}{k\#\left(\iota_{v,k}\left(s\right)\right)}\sum_{\ell\in\iota_{v,k}\left(s\right)}j_{k,\ell}\left(v\right)=\frac{\int_{\left[s,\cS_0\left(v\right)\right]}td\mu_{v,k}}{\int_{\left[s,\cS_0\left(v\right)\right]}d\mu_{v,k}},
    \end{equation}
    proving \eqref{SkmccdfequivEq}.
\end{proof}

\subsection{General double limits: tail expectations}
\label{general double}

We are now in position to prove Proposition \ref{S_tau def}.

\begin{proof}[Proof of Proposition \ref{S_tau def}]
    When $\tau=0,1$, this is Lemma \ref{tau=max} and Proposition \ref{Strong version}.
Thus, suppose $m_k\in\{1,\ldots,d_k\}$ for $k\in\NN(L)$,
    and
    $
        \lim_{k\to\infty}\frac{m_k}{d_k}=\tau\in\left(0,1\right).
    $
Note that (recall \eqref{muvEq} and Definition \ref{ccdfDef})
\begin{align}
        Q_v(\tau)&=\sup\left\{t\leq\cS_0\left(v\right)\,:\,\mu_v\left(\left[t,\cS_0\left(v\right)\right]\right)\geq\tau\right\}\label{t_0 def'}\\
        &=\sup\left\{t\leq\cS_0\left(v\right)\,:\,\Vol V_{v,\bullet}^t\geq\tau\VolL\right\}\in\left(0,\cS_0\left(v\right)\right].\label{t_0 def}
    \end{align}

    \begin{lemma}
    \lb{VvvolumestrictLem}
        For $\tau\in[0,1)$ and $t\in[0,Q_v(\tau))$, $\Vol V_{v,\bullet}^t>\tau\VolL$.
    \end{lemma}
    \begin{proof}
        Fix $\tau\in[0,1)$ and $t\in[0,Q_v(\tau))$, and suppose
    \beq\lb{VolVtVolEq}
    \Vol V_{v,\bullet}^t\leq\tau\VolL.
    \eeq
        Pick
        \begin{equation}\label{contradiction}
            s\in\left(t,Q_v\left(\tau\right)\right).
        \end{equation}
By Claim \ref{conv lem 2} (whose assumptions are fulfilled by \eqref{VolVtVolEq}
as $\tau<1$),
        $
            \Vol V_{v,\bullet}^s<\Vol V_{v,\bullet}^t\leq\tau\VolL.
        $
        By \eqref{t_0 def}, $Q_v(\tau)\leq s$, contradicting \eqref{contradiction}.
    \end{proof}
    By \eqref{VolGradedEq} and Lemma \ref{VvvolumestrictLem},
    $$\baligned
        \lim_{k\to\infty}\frac{1}{k^n}\dim\cF_v^{k\left(Q_v\left(\tau\right)-\varepsilon\right)}R_k
        =\frac{1}{n!}\Vol V_{v,\bullet}^{Q_v\left(\tau\right)-\varepsilon}
        >\frac{\tau{\VolL}}{n!}
=\frac{\VolL}{n!}\lim_{k\to\infty}\frac{m_k}{d_k}=\lim_{k\to\infty}\frac{m_k}{k^n}
        ,\q \hbox{for any $\varepsilon\in(0,Q_v(\tau)]$}.
    \ealigned
    $$
    Therefore, 
    $
        m_k\leq\dim\cF_v^{k\left(Q_v\left(\tau\right)-\varepsilon\right)}R_k,\q
        \hbox{\ for sufficiently large $k$}.
    $
By \eqref{iota def},
    \begin{align*}
        \#\Big(\iota_{v,k}\big(Q_v(\tau)-\varepsilon\big)\Big)&=\#\left\{\ell\in\NN\,:\,1\leq\ell\leq\dim\cF_v^{k\left(Q_v\left(\tau\right)-\varepsilon\right)}R_k\right\}\\
        &=\dim\cF_v^{k\left(Q_v\left(\tau\right)-\varepsilon\right)}R_k
        \geq m_k.
    \end{align*}
    By \eqref{m monotonicity}, \eqref{iota}, and Theorem \ref{mu},
    \begin{align*}
        \liminf_{k\to\infty}S_{k,m_k}\left(v\right)&\geq\liminf_{k\to\infty}S_{k,\#\left(\iota_{v,k}\left(Q_v\left(\tau\right)-\varepsilon\right)\right)}\left(v\right)\\
        &=\liminf_{k\to\infty}\frac{\int_{\left[Q_v\left(\tau\right)-\varepsilon,\cS_0\left(v\right)\right]}td\mu_{v,k}\left(t\right)}{\int_{\left[Q_v\left(\tau\right)-\varepsilon,\cS_0\left(v\right)\right]}d\mu_{v,k}\left(t\right)}
        =\frac{\int_{\left[Q_v\left(\tau\right)-\varepsilon,\cS_0\left(v\right)\right]}td\mu_v\left(t\right)}{\int_{\left[Q_v\left(\tau\right)-\varepsilon,\cS_0\left(v\right)\right]}d\mu_v\left(t\right)}.
    \end{align*}
    Since $\varepsilon\in(0,Q_v(\tau)]$ is arbitrary,
    \begin{equation}\label{S_k,m_k}
        \liminf_{k\to\infty}S_{k,m_k}\left(v\right)\geq\frac{\int_{\left[Q_v\left(\tau\right),\cS_0\left(v\right)\right]}td\mu_v\left(t\right)}{\int_{\left[Q_v\left(\tau\right),\cS_0\left(v\right)\right]}d\mu_v\left(t\right)}.
    \end{equation}
    
    Now, there are two cases to consider. \newline
    First, suppose $\tau\leq\mu_v(\{\cS_0(v)\})$. By \eqref{t_0 def'}, $Q_v(\tau)=\cS_0(v)$. By \eqref{S_k,m_k} (and as $\tau>0$),
    $$
        \liminf_{k\to\infty}S_{k,m_k}\left(v\right)\geq\frac{\int_{\left\{\cS_0\left(v\right)\right\}}td\mu_v\left(t\right)}{\int_{\left\{\cS_0\left(v\right)\right\}}d\mu_v\left(t\right)}=\cS_0\left(v\right).
    $$
    By \eqref{def T} and \eqref{m monotonicity},
    $
        S_{k,m_k}(v)\leq S_{k,1}(v)\leq\cS_0(v).
    $
    Altogether,
    $
        \lim_{k\to\infty}S_{k,m_k}\left(v\right)=\cS_0\left(v\right).
    $
    \newline
    Second, suppose $\tau>\mu_v(\{\cS_0(v)\})$. By \eqref{mu charge}, there exists $s<\cS_0(v)$ such that
    $$
        \Vol V_{v,\bullet}^{s}<\tau\VolL.
    $$
    By \eqref{t_0 def}, $Q_v(\tau)\leq s<\cS_0(v)$. Again by \eqref{t_0 def}, for $\varepsilon\in(0,\cS_0(v)-Q_v(\tau))$, $\Vol V_{v,\bullet}^{Q_v(\tau)+\varepsilon}<\tau\VolL$, so,
    $$
        \lim_{k\to\infty}\frac{1}{d_k}\dim\cF_v^{k\left(Q_v\left(\tau\right)+\varepsilon\right)}R_k=\frac{\Vol V_{v,\bullet}^{Q_v\left(\tau\right)+\varepsilon}}{\Vol L}<{\tau}=\lim_{k\to\infty}\frac{m_k}{d_k}.
    $$
    Therefore for sufficiently large $k$,
    $
        \dim\cF_v^{k\left(Q_v\left(\tau\right)+\varepsilon\right)}R_k\le m_k.
    $
    By \eqref{iota def},
    \begin{align*}
        \#\Big(\iota_{v,k}\big(Q_v\left(\tau\right)+\varepsilon\big)\Big)
        =\dim\cF_v^{k\left(Q_v\left(\tau\right)+\varepsilon\right)}R_k
        \leq m_k.
    \end{align*}
    By \eqref{m monotonicity} and \eqref{iota},
    \begin{align*}
        \limsup_{k\to\infty}S_{k,m_k}\left(v\right)&\leq\limsup_{k\to\infty}S_{k,\#\left(\iota_{v,k}\left(Q_v\left(\tau\right)+\varepsilon\right)\right)}\left(v\right)\\
        &=\limsup_{k\to\infty}\frac{\int_{\left[Q_v\left(\tau\right)+\varepsilon,\cS_0\left(v\right)\right]}td\mu_{v,k}\left(t\right)}{\int_{\left[Q_v\left(\tau\right)+\varepsilon,\cS_0\left(v\right)\right]}d\mu_{v,k}\left(t\right)}
        =\frac{\int_{\left[Q_v\left(\tau\right)+\varepsilon,\cS_0\left(v\right)\right]}td\mu_v\left(t\right)}{\int_{\left[Q_v\left(\tau\right)+\varepsilon,\cS_0\left(v\right)\right]}d\mu_v\left(t\right)}.
    \end{align*}
    Since $\varepsilon$ is arbitrary,
    $$
        \limsup_{k\to\infty}S_{k,m_k}\left(v\right)\leq\frac{\int_{\left[Q_v\left(\tau\right),\cS_0\left(v\right)\right]}td\mu_v\left(t\right)}{\int_{\left[Q_v\left(\tau\right),\cS_0\left(v\right)\right]}d\mu_v\left(t\right)}.
    $$
    Combining with \eqref{S_k,m_k},
    $$
        \lim_{k\to\infty}S_{k,m_k}\left(v\right)=\frac{\int_{\left[Q_v\left(\tau\right),\cS_0\left(v\right)\right]}td\mu_v\left(t\right)}{\int_{\left[Q_v\left(\tau\right),\cS_0\left(v\right)\right]}d\mu_v\left(t\right)}
    $$
    also in this case.
Finally, monotonicity of $\cS_\tau$ follows from \eqref{m monotonicity},
concluding the~proof~of~Proposition~\ref{S_tau def}.
\end{proof}

\begin{remark}
Recall \eqref{GvDef}. Proposition \ref{S_tau def} can be rephrased as
    \begin{equation}\label{S_infty}
        \cS_\tau\left(v\right)=\frac{\int_{\left[Q_v\left(\tau\right),\cS_0\left(v\right)\right]}td\mu_v\left(t\right)}{\int_{\left[Q_v\left(\tau\right),\cS_0\left(v\right)\right]}d\mu_v\left(t\right)}=\frac{\int_{\Delta_v^{Q_v\left(\tau\right)}}G_v\left(x\right)dx}{\int_{\Delta_v^{Q_v\left(\tau\right)}}dx},
    \end{equation}
    i.e., $\cS_\tau(v)$ is
    the average value of $G_v$ on $\Delta_v^{Q_v\left(\tau\right)}$.
\end{remark}

\subsection{Uniform estimates on the Ehrhart function in convex sub-bodies}
\label{lattice point estimates}
In this subsection we prove estimates on lattice points in a convex body $K$.
Specifically, we are interested to show estimates that
hold {\it uniformly  for any convex body inside $K$ and simultaneously
for any integer dilation of $K$}.
Lemma \ref{lattice upper bound}
can be thought of as a one-sided generalized Euler--Maclaurin type formula
(a typical Euler-Maclaurin formula applies to a {\it fixed} body
\cite{KSW,PS}).
Both Lemma \ref{lattice upper bound}
and Proposition \ref{lattice lower bound} can be thought of 
as {\it uniform} versions of estimates on the 
{\it discrepancy} (also called {\it rest}) or 
{\it Ehrhart function}. These estimates seem to be new even in the case
of convex lattice polytopes, though in that case one can 
derive much more refined estimates using Ehrhart theory \cite{JR2}.

The main result of this subsection is:

\bthm\lb{MainEMacThm}
 For any convex body $K\subset\RR^n$ and $\nu\in(0,|K|]$, there exists $C(\nu,K)>0$ such that for any convex body $P\subset K$ with $|P|\geq \nu$,
    $$
        \Big|\#\left(P\cap\ZZ^n/k\right)
        -\left|P\right|k^n\Big|\le Ck^{n-1}, 
        \q\hbox{for all $k\in\NN$}.
    $$
\ethm

Theorem \ref{MainEMacThm} follows from Lemma \ref{lattice upper bound}
and Proposition \ref{lattice lower bound}. Note the rate $O(k^{n-1})$ is optimal
(cf. \cite[Proposition 38]{BL18} where a non-uniform version is obtained 
with rate $o(k^n)$).

\begin{lemma}\label{lattice upper bound}
    For a convex body $K\subset\RR^n$, there exists $C=C(K)>0$ such that for any non-negative concave function $G$ defined on a convex body $P\subset K$ and $k\in\NN$,
    $$
        \frac{1}{k^n}\sum_{x\in P\cap\ZZ^n/k}G\left(x\right)\leq\int_PG\left(x\right)dx+\frac{C}{k}\sup G.
    $$
    In particular,
    $
        \#\left(P\cap\ZZ^n/k\right)\leq\left|P\right|k^n+Ck^{n-1}.
    $
\end{lemma}
\begin{proof}
    We may rescale and assume $\sup G=1$. Notice that
    \begin{equation}\label{cylinder}
        \frac{1}{k^n}\sum_{x\in P\cap\ZZ^n/k}G\left(x\right)
        =\left|\left(\bigcup_{x\in P\cap\ZZ^n/k}\left\{x\right\}\times\left[0,G\left(x\right)\right]\right)+\hat Q/k\right|,
    \end{equation}
    where
    \beq\lb{hatQEq}
        \hat Q=\left[-\frac{1}{2},\frac{1}{2}\right]^n\times\left\{0\right\}\subset\RR^{n+1}.
    \eeq
    Define
    \beq\lb{PGEq}
        P_G:=\left\{\left(x,y\right)\in P\times\left[0,1\right]\,:\,0\leq y\leq G\left(x\right)\right\}\subset K\times[0,1].
    \eeq
Then,
\begin{equation}\label{graph inclusion}
        \bigcup_{x\in P\cap\ZZ^n/k}\left\{x\right\}\times\left[0,G\left(x\right)\right]\subset P_G.
    \end{equation}
    Combining \eqref{cylinder}--\eqref{graph inclusion}, 
    and expanding the volume of a Minkowski sum of two convex bodies
    (note $P_G$ is convex since $G$ is concave),
    \begin{align*}
        \frac{1}{k^n}\sum_{x\in P\cap\ZZ^n/k}G\left(x\right)&\leq\left|P_G+\hat Q/k\right|\\
        &=\left|P_G\right|+\sum_{i=1}^n\binom{n+1}{i}V\left(P_G,n+1-i;\hat Q,i\right)k^{-i}\\
        &\leq\int_PG\left(x\right)dx+\sum_{i=1}^n\binom{n+1}{i}V\left(K\times\left[0,1\right],n+1-i;\hat Q,i\right)k^{-i},
    \end{align*}
    where $V(A,k;B,\ell)$ denotes the mixed volume $V(A,\ldots,A,B,\ldots,B)$ where $A$ appears $k$ times and $B$ appears $\ell$ times \cite[Theorem 5.1.7]{schneider}. This completes the proof of Lemma \ref{lattice upper bound}.
\end{proof}

\begin{prop}\label{lattice lower bound}
    For a convex body $K\subset\RR^n$ and $\nu\in(0,|K|]$, there exists $C=C(\nu,K)>0$ such that for any convex body $P\subset K$ with $|P|\geq \nu$,
    $$
        \#\left(P\cap\ZZ^n/k\right)\geq\left(1-\frac{C}{k}\right)\left|P\right|k^n,       \q\hbox{for all $k\in\NN$}.
    $$
\end{prop}

\begin{proof} First, a useful lemma.

\blem\lb{BallInsidePLemma}
    For a convex body $K\subset\RR^n$ and $\nu\in(0,|K|]$, there exists $r=(\nu,K)>0$ such that for any convex body $P\subset K$ with $|P|\geq \nu$,
 there is a Euclidean ball of radius $r$ contained in $P$.
\elem

\bpf
    Suppose such an $r$ does not exist, i.e., there is a sequence $P_k\subset K$ with $|P_k|\geq \nu$ such that
    $$
        r_k:=\sup\left\{r\,:\,B\left(p,r\right)\subset P_k\mbox{ for some }p\right\}\to0 \q \hbox{ as $k\ra\infty$},
    $$
    where $B(p,r)$ denotes a Euclidean ball of radius $r$ centered at $p\in\RR^n$.
    By the Blaschke selection theorem, up to a subsequence we may assume $P_k\to P_\infty$ in Hausdorff distance $d_H$
    \cite[p. 62]{Blaschke16Book}, \cite[Theorem 1.8.7]{schneider}. By $d_H$-continuity of volume, $|P_\infty|\geq \nu>0$. In particular, $P_\infty$ has non-empty interior, i.e., there is $p\in P_\infty$ and $r_0>0$ with $B(p,r_0)\subset P_\infty$. For sufficiently large $k$,
    $
        d_H\left(P_k,P_\infty\right)<{r_0}/{2}.
    $
    In particular,
    $$
        P_k+B\left(0,\frac{r_0}{2}\right)\supset P\supset B\left(p,r_0\right)=B\left(p,\frac{r_0}{2}\right)+B\left(0,\frac{r_0}{2}\right).
    $$
    It follows that
    $
        P_k\supset B\left(p,{r_0}/{2}\right),
    $
    i.e.,
    $
        r_k\geq{r_0}/{2}
    $
    for sufficiently large $k$. This is a contradiction.
\epf

    \begin{figure}[ht]
        \centering
        \begin{tikzpicture}
            \draw[domain=-180:180,samples=100]plot({2.5*cos(\x)},{1.5*sin(\x)-exp(.3*cos(\x))+.5});
            \draw(2.2,-1.1)node{$P'$};
            \draw[domain=-180:180,samples=100]plot({3*cos(\x)},{1.8*sin(\x)-1.2*exp(.3*cos(\x))+.6});
            \draw(2.64,-1.32)node{$P$};
            \filldraw(0,0)circle(1pt)node[below right]{$p$};
            \draw(0,0)circle(.85);
            \draw(1,-1)node{$B(p,r)$};
        \end{tikzpicture}
        \caption{The convex bodies $P$ and $P'$, and the ball $B(p,r)\subset P$
        (Lemma \ref{P'PlatticeLem} and its proof).}
        \lb{FigPPprime}
    \end{figure}
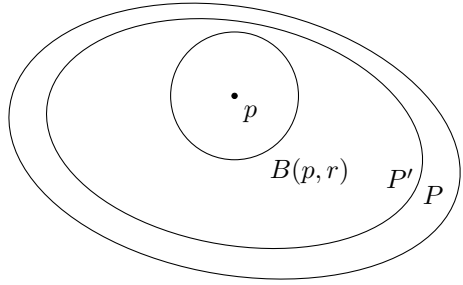

Next, we use Lemma \ref{BallInsidePLemma}  to construct an auxiliary convex sub-body that is useful for estimating from below the number of lattice points in $P$.

\blem\lb{P'PlatticeLem}
  Fix a convex body $K\subset\RR^n$ and a number $\nu\in(0,|K|]$.
  There exists $r=r(\nu,K)>0$ such that 
  for any convex sub-body $P\subset K$ satisfying $|P|\in(\nu,|K|]$, 
  there exists a point $p\in P$ with the property that 
    $$
        P_p:=\left(1-\frac{n^{3/2}}{2rk}\right)^\frac{1}{n}\left(P-p\right)+p\subset P
    $$
satisfies 
\begin{equation}\label{lattice lb}
        P_p\subset\left(P\cap\ZZ^n/k\right)+\left[-\frac{1}{2},\frac{1}{2}\right]^n\Big/k.
    \end{equation}
\elem

\bpf
Fix a convex body $P\subset K$ with $|P|\in(\nu,|K|]$.  
    Let $B(p,r)\subset P$ be given by Lemma \ref{BallInsidePLemma} 
    and let $C>0$  be a constant to be determined shortly.  Note,
    \begin{equation}\label{Bernoulli}
        \left(1-\frac{C}{k}\right)^\frac{1}{n}+\frac{C}{nk}\leq1,\q
        \hbox{for any $k>C$}.
    \end{equation}
    Set 
    $$
        P':=\left(1-\frac{C}{k}\right)^\frac{1}{n}\left(P-p\right)+p.
    $$
    Since $p\in P$ and $C<k$, $P'\subset P$ (see Figure \ref{FigPPprime}).
    Note that
    $$
        Q:=\left[-\frac{1}{2},\frac{1}{2}\right]^n\subset B\left(0,\frac{\sqrt{n}}{2}\right)\subset\RR^n.
    $$
    Thus, by \eqref{Bernoulli},
    \begin{equation*}
        P\supset P'+\frac{C}{kn}\left(P-p\right)\supset P'+\frac{C}{kn}B\left(0,r\right)\supset P'+\frac{2Cr}{kn^{3/2}}Q.
    \end{equation*}
Now, set $C=\frac{n^{3/2}}{2r}$, so that 
\beq\lb{PprimeQkEq}
P'+Q/k\subset P.
\eeq
Lemma \ref{P'PlatticeLem}
now follows from Claim \ref{PPrimeClaim}.
\epf

\bclaim\lb{PPrimeClaim}
 $P'\subset\left(P\cap\ZZ^n/k\right)+Q/k$.
\eclaim
\bpf
Let $x\in P'$. By \eqref{PprimeQkEq},
    \begin{equation}\label{cube inclusion}
        x+Q/k\subset P'+Q/k\subset P.
    \end{equation}
    Since $x+Q/k$ is a cube with edge length $1/k$, we can find
    \begin{equation}\label{lattice point}
        x'\in\left(x+Q/k\right)\cap\ZZ^n/k,
    \end{equation}
    i.e., a point $x'\in\ZZ^n/k$ such that $\|x-x'\|_\infty\leq\frac{1}{2k}$. Combining \eqref{cube inclusion} and \eqref{lattice point},
    $$
        x\in x'+Q/k\subset\left(\left(x+Q/k\right)\cap\ZZ^n/k\right)+Q/k\subset\left(P\cap\ZZ^n/k\right)+Q/k.
    $$
    This proves Claim \ref{PPrimeClaim}.
    \epf
By Lemma \ref{P'PlatticeLem},
    $$
    \baligned
    \frac{1}{k^n}\#\left(P\cap\ZZ^n/k\right)
    &=
    \left|\left(P\cap\ZZ^n/k\right)+Q/k\right|
    \ge
    \left|P_p\right|
    =\left(1-\frac{n^{3/2}}{2rk}\right)|P|,
    \ealigned
    $$
concluding the proof of Proposition \ref{lattice lower bound}.
\epf

\subsection{Idealized expected vanishing order}\label{the concave transform}

The main result of this subsection is a uniform one-sided estimate 
on the convergence~speed~of~$\overline{S}_{k,m}$. 

\begin{prop}\label{S upper}
    Suppose $m_k\in\{1,\ldots,d_k\}$ for $k\in\NN(L)$, and
    $
        \lim_{k\to\infty}\frac{m_k}{d_k}=\tau\in\left[0,1\right].
    $
    There exists $C,C'>0$ depending on $L$ and $\tau$ (independent of $v$ and $k$) such that     on $\ValXfin$,

    $$
        \overline{S}_{k,m_k}
        \leq
        \max\left\{
        \frac{\tau d_k}{m_k},1
        \right\}
        \left(
        \cS_\tau
        +\frac{C}{k}\cS_0
        \right),\q
        k\in\NN(L).
    $$
\end{prop}
\begin{proof}
Let $v\in\ValXfin$. Let $Q_v(\tau)$ be as in \eqref{t_0 def}, i.e.,
    $$
        Q_v(\tau)=\sup\left\{t\leq\cS_0\left(v\right)\,:\,\left|\Delta^t\right|\geq\tau{{\VolL}/{n!}}\right\}.
    $$
    If $Q_v(\tau)=\cS_0(v)$, by Proposition \ref{S_tau def},
    $$
        \cS_\tau\left(v\right)=\frac{\int_{\left\{\cS_0\left(v\right)\right\}}td\mu_v\left(t\right)}{\int_{\left\{\cS_0\left(v\right)\right\}}d\mu_v\left(t\right)}=\cS_0\left(v\right)\geq\overline{S}_{k,m_k}\left(v\right).
    $$
    Next, suppose $Q_v(\tau)<\cS_0(v)$. Then by Claim \ref{conv lem 1}, $\Vol V_{v,\bullet}^{Q_v(\tau)}>0$. By continuity of the concave function $t\mapsto(\Vol V_{v,\bullet}^t)^\frac{1}{n}$ (recall Proposition \ref{V.t}), supremum is attained, and (recall \eqref{LM})
    \beq\lb{QvtauvolEq}
        \left|\Delta^{Q_v(\tau)}\right|=\tau{{\VolL}/{n!}}>0.
    \eeq
    Recall \eqref{MkEq}. By Lemma \ref{barSmkdecreasingRem},  Lemma \ref{lattice upper bound}, \eqref{GvDef}, \eqref{LM}, and \eqref{S_infty},
    \begin{align*}
        \overline{S}_{k,M_k}\left(v\right)&=\frac{1}{M_k}\sum_{x\in\Delta_v^{Q_v(\tau)}\cap\,\ZZ^n/k}G_v\left(x\right)\\
        &\leq\frac{k^n}{M_k}\left(\int_{\Delta_v^{Q_v(\tau)}}G_v\left(x\right)dx+\frac{C_1}{k}\sup G_v\right)\\
        &=\frac{k^n}{M_k}\left(\tau\frac{\VolL}{n!} \cS_\tau\left(v\right)+\frac{C_1}{k}\cS_0\left(v\right)\right)
        =\frac{\tau k^n}{M_k}\frac{\VolL}{n!}\left(\cS_\tau\left(v\right)+\frac{C_2}{k}\cS_0\left(v\right)\right).
    \end{align*}
    Moreover, by Proposition \ref{lattice lower bound},
    $$
        M_k\geq\left(1-\frac{C_3}{k}\right)\tau\frac{\VolL}{n!}k^n.
    $$
    Thus, by Lemma \ref{mkMkLemma},
    \begin{align*}
        \overline{S}_{k,m_k}\left(v\right)&\leq\max\left\{\frac{\tau k^n}{m_k}\frac{\VolL}{n!},\left(1-\frac{C_3}{k}\right)^{-1}\right\}\left(\cS_\tau\left(v\right)+\frac{C_2}{k}\cS_0\left(v\right)\right)\\
        &\leq\left(1+\frac{C_4}{k}\right)\max\left\{\frac{\tau d_k}{m_k},1\right\}\left(\cS_\tau\left(v\right)+\frac{C_2}{k}\cS_0\left(v\right)\right)
        \leq\max\left\{\frac{\tau d_k}{m_k},1\right\}\left(\cS_\tau\left(v\right)+\frac{C_5}{k}\cS_0\left(v\right)\right).
    \end{align*}
\end{proof}

\subsection{Well-definedness of stability thresholds}

In this subsection we finish the construction of the invariants $\boldsymbol\delta_\tau$ (Corollary \ref{deltadoubleasymCor}) and give an asymptotic lower bound.

\bcor\label{S upper bound}
    Suppose $m_k\in\{1,\ldots,d_k\}$ for $k\in\NN(L)$, and
    $
        \lim_{k\to\infty}\frac{m_k}{d_k}=\tau\in\left[0,1\right].
    $
There exists $C>0$ depending on $L$ and $\tau$ (independent of $k$) such that on $\ValXfin$,
    $$
        S_{k,m_k}
        \leq\left(1+\frac{C}{k}\right)\max\left\{\frac{\tau d_k}{m_k},1\right\}\cS_\tau,
                \q k\in\NN(L).
    $$
\ecor
\begin{proof}
    By Remark \ref{S leq S upper} and Proposition \ref{S upper},
    $$
        S_{k,m_k}
        \leq\max\left\{\frac{\tau d_k}{m_k},1\right\}
        \left(\cS_\tau
        +\frac{C_1}{k}\cS_0
        \right).
    $$
    Since 
    $
        \cS_0
        \leq\left(n+1\right)\cS_1
    $
    \cite[Lemma 2.6]{BJ20}
    and as $\cS_\tau$ decreases in $\tau$ (Proposition \ref{S_tau def}),
    \begin{align*}
        S_{k,m_k}\left(v\right)&\leq\max\left\{\frac{\tau d_k}{m_k},1\right\}\left(\cS_\tau\left(v\right)+\frac{C_2}{k}\cS_1\left(v\right)\right)\\
        &\leq\max\left\{\frac{\tau d_k}{m_k},1\right\}\left(\cS_\tau\left(v\right)+\frac{C_2}{k}\cS_\tau\left(v\right)\right)
        =\left(1+\frac{C_2}{k}\right)\max\left\{\frac{\tau d_k}{m_k},1\right\}\cS_\tau\left(v\right).
    \end{align*}
\end{proof}

\bcor 
\lb{deltadoubleasymCor}
    Suppose $m_k\in\{1,\ldots,d_k\}$ for $k\in\NN(L)$, and
    $
        \lim_{k\to\infty}\frac{m_k}{d_k}=\tau\in\left[0,1\right].
    $
The limit
    $$
        \bm{\delta}_\tau:=\lim_{k\to\infty}\delta_{k,m_k}
    $$
    exists, and
    \begin{equation}\label{delta_tau alg}
        \bm{\delta}_\tau=\inf_{v\in\ValXfin}\frac{A\left(v\right)}{\cS_\tau\left(v\right)}=\inf_{v\in\ValXdiv}\frac{A\left(v\right)}{\cS_\tau\left(v\right)}.
    \end{equation}
    In particular, $\bm{\delta}_0=\alpha$, and $\bm{\delta}_1=\delta$.
    Moreover,
    $$
        \delta_{k,m_k}\geq\left(1-\frac{C}{k}\right)\min\left\{\frac{m_k}{\tau d_k},1\right\}\bm{\delta}_\tau.
    $$
\ecor
\begin{proof}
Let $v_0\in\ValXfin$. By Lemma \ref{deltakminfProp},
    $$
        \limsup_{k\to\infty}\delta_{k,m_k}=\limsup_{k\to\infty}\inf_{v\in\ValX}\frac{A\left(v\right)}{S_{k,m_k}\left(v\right)}\leq\limsup_{k\to\infty}\frac{A\left(v_0\right)}{S_{k,m_k}\left(v_0\right)}=\frac{A\left(v_0\right)}{\lim\limits_{k\to\infty}S_{k,m_k}\left(v_0\right)}=\frac{A\left(v_0\right)}{\cS_\tau\left(v_0\right)}.
    $$
    Since $v_0$ is arbitrary, by Corollary \ref{S upper bound},
    $$
        \limsup_{k\to\infty}\delta_{k,m_k}\leq\inf_{v\in\ValXfin}\frac{A\left(v\right)}{\cS_\tau\left(v\right)}\leq\inf_{v\in\ValXdiv}\frac{A\left(v\right)}{\cS_\tau\left(v\right)}\leq\left(1+\frac{C}{k}\right)\max\left\{\frac{\tau d_k}{m_k},1\right\}\delta_{k,m_k}.
    $$
    Letting $k\to\infty$,
    $$
\limsup_{k\to\infty}\delta_{k,m_k}\leq\inf_{v\in\ValXfin}\frac{A\left(v\right)}{\cS_\tau\left(v\right)}\leq\inf_{v\in\ValXdiv}\frac{A\left(v\right)}{\cS_\tau\left(v\right)}\leq\liminf_{k\to\infty}\delta_{k,m_k}.
    $$
    This completes the proof.
\end{proof}

\section{Joint asymptotics of
Grassmannian thresholds: upper bound}\label{ub section}

The main result of this section is the remaining one-sided
asymptotics of $S_{k,m}$ on $\ValXdiv$.

\begin{theorem}\label{SdoubleasymCor}
    Suppose $m_k\in\{1,\ldots,d_k\}$ for $k\in\NN(L)$, and
    $
        \lim_{k\to\infty}\frac{m_k}{d_k}=\tau\in\left[0,1\right].
    $
    For any $v\in\ValXdiv$, there exists $C>0$ depending on $L$, $\tau$, $\inf_k\{m_k/d_k\}$, and $v$ (independent~of~$k$)~such~that
    \begin{numcases}{S_{k,m_k}\left(v\right)\geq}
        \left(1-C\left(\max\left\{\frac{m_k}{d_k},\frac{1}{k}\right\}\right)^\frac{1}{n}\right)\cS_\tau\left(v\right),&$\tau=0$,\label{Case2}\\
        \nonumber\\
        \left(1-\frac{C}{k}\right)\min\left\{\frac{\tau d_k}{m_k},1\right\}\cS_\tau\left(v\right),&$\tau\in(0,1]$.\label{Case1}
    \end{numcases}
\end{theorem}

\subsection{Joint asymptotics of \texorpdfstring{$S_{k,m}$: non-collapsing case and the rooftop Okounkov body}{S{k,mk}}}
\label{SkmnoncollapsingSubSec}

The proof in the non-collapsing case relies on a key idea from \cite{JR2} in the toric case: we construct
a `rooftop Okounkov body' and then compare its volume to its lattice point count.

\blem
\lb{DkdkLem}
For some $C>0$ independent of $k$,
$0\le D_k-d_k\le Ck^{n-1}$.
\elem

\bpf
    By Lemma \ref{lattice upper bound},
$
        D_k\leq\left|\Delta\right|k^n+C_1k^{n-1},
$
    while 
    $
        d_k=\left|\Delta\right|k^n+O\left(k^{n-1}\right)
    $
    \cite[Theorem A]{LM09}.
The result follows.
\epf

\begin{proof}[Proof of Theorem \ref{SdoubleasymCor} when $\tau>0$]
Follow the notation in the proof of Lemma \ref{sliceDivOkBodyLem}
for the Okounkov body associated to the flag \eqref{flagFEq}.

By Definition \ref{upper S}, $i_{k,\ell}\in[0,\cS(v)]$. Thus, by Lemma \ref{weakSkmIneq} (recall
we are working on $\ValXdiv$) together with  \eqref{first est} and Lemma \ref{DkdkLem}
(and using $\tau>0$, i.e., $m_k=O(\tau d_k)=O(k^n)$),
    \begin{align}
        S_{k,m_k}\left(v\right)&\geq\frac{1}{km_k}\left(\sum_{\ell=1}^{m_k+D_k-d_k}i_{k,\ell}\left(v\right)-\left(D_k-d_k\right)k\cS_0\left(v\right)\right)\nonumber\\
        &\geq\frac{1}{km_k}\sum_{\ell=1}^{m_k}i_{k,\ell}\left(v\right)-\frac{D_k-d_k}{m_k}\cS_0\left(v\right)
        \geq\frac{1}{km_k}\sum_{\ell=1}^{m_k}i_{k,\ell}\left(v\right)-\frac{C_3}{k}
        =\overline{S}_{k,m_k}\left(v\right)-\frac{C_3}{k}.\label{1st lower bd}
    \end{align}
    Since $v\in\ValXdiv$,
    Proposition \ref{S_tau def} together with Lemma \ref{sliceDivOkBodyLem} imply that $\cS_\tau(v)$ is the first coordinate of the barycenter of $\Delta_v^{Q_v\left(\tau\right)}$.
    In other words (using
    $\tau>0$ and \eqref{QvtauvolEq}),
\beq
\lb{cCtauvolEq}
\cS_\tau\left(v\right)=\frac{n!}{\tau\VolL}\int_{\Delta_v^{Q_v\left(\tau\right)}}x_1dx
=
\frac1{\big|\Delta_v^{Q_v(\tau)}\big|}\int_{\Delta_v^{Q_v\left(\tau\right)}}x_1dx,
\eeq
where $x_1=p_1(x)$.
    By Theorem \ref{MainEMacThm} applied to $\Delta_v^{Q_v(\tau)}$ (using $\tau>0$),
    $$
        \left(1-\frac{C_4}{k}\right)\tau\frac{\VolL}{n!}k^n\leq M_k\leq\left(1+\frac{C_4}{k}\right)\tau\frac{\VolL}{n!}k^n.
    $$
    Hence,
    \begin{equation}\label{3rd lower bd}
        \left(1-\frac{C_5}{k}\right)\tau d_k\leq M_k\leq\left(1+\frac{C_5}{k}\right)\tau d_k.
    \end{equation}
    Consider the `rooftop' convex body
    $$
        \widehat{K}_\tau:=\left\{\left(x,t\right)\in\RR^n\times\RR\,:\,x\in\Delta_v^{Q_v\left(\tau\right)},~0\leq t\leq x_1\right\}.
    $$
By \eqref{cCtauvolEq},
\beq\lb{rooftopvolEq}
        \left|\widehat{K}_\tau\right|=\int_{\Delta_v^{Q_v\left(\tau\right)}}x_1dx=\tau\frac{\VolL}{n!}\cS_\tau\left(v\right),
    \eeq
    and
    $$
        \widehat{K}_\tau\cap\ZZ^{n+1}/k=\bigcup_{x\in K\cap\ZZ^n/k}\left\{\left(x,0\right),\left(x,\frac{1}{k}\right),\ldots,\left(x,x_1\right)\right\}.
    $$
    Hence (recall \eqref{MkEq})
    $$
        \#\left(\widehat{K}_\tau\cap\ZZ^{n+1}/k\right)=\sum_{x\in K\cap\ZZ^n/k}\left(kx_1+1\right)=\sum_{\ell=1}^{M_k}i_{k,\ell}\left(v\right)+\#\left(K\cap\ZZ^n/k\right)=kM_k\overline{S}_{k,M_k}\left(v\right)+M_k.
    $$
    By Proposition \ref{lattice lower bound} and \eqref{rooftopvolEq},
    $
        kM_k\overline{S}_{k,M_k}\left(v\right)+M_k\geq\left(1-\frac{C_6}{k}\right)\tau\frac{\VolL}{n!}\cS_\tau\left(v\right)k^{n+1},
    $
    i.e.,
    \begin{equation}\label{4th lower bd}
        \overline{S}_{k,M_k}\left(v\right)\geq\left(1-\frac{C_6}{k}\right)\tau\frac{\VolL}{n!}\cS_\tau\left(v\right)\frac{k^n}{M_k}-\frac{1}{k}\geq\left(1-\frac{C_7}{k}\right)\tau\frac{d_k}{M_k}\cS_\tau\left(v\right)-\frac{1}{k}.
    \end{equation}
    \eqref{Case1} follows from \eqref{1st lower bd}, 
    Lemma \ref{mkMkLemma},
    \eqref{3rd lower bd}, and \eqref{4th lower bd}.
\end{proof}

\subsection{Cones and convex hulls in Okounkov bodies}

In performing lattice point counts in Okounkov bodies we will consider
certain canonical two-parameter families of convex sub-bodies of $\Delta$.

\bdefn
Let $v\in\ValXfin$, and let $a,b\in [\sigma(v),\cS_0(v)]$ with $a<b$.
The {\it $(a,b)$-slice-cone~of~$\Delta$}~is 
$$
\cC_v^{a,b}:=\co\left( G_v^{-1}(a)\cup G_v^{-1}(b)\right)\subset \Delta_v^a. 
$$
\edefn
Henceforth we will assume that $v\in\ValXdiv$ and that the flag is chosen as 
in \eqref{flagFEq} so that $G_v=p_1|_\Delta$ (recall \eqref{Deltavt}). 
Inside of the slice-cones of $\Delta$ one can fit
two distinguished families of convex sub-bodies. First,
the one-parameter superlevel sets
$$
\cC_v^{a,b}(t):=\cC_v^{a,b}\cap \{p_1\ge t\}, \q t\in[a,b].
$$
Second, for $V\in p_1^{-1}(b)\cap \Delta$, the cones
$$
\cC_v^{a,b,V}:=\co\left(p_1^{-1}(a)\cap\Delta\cup \{V\}\right)\subset \Delta_v^a,
$$
and their associated one-parameter family of sub-cones (see Figure \ref{Cone})
$$
\cC_v^{a,b,V}(t):=\cC_v^{a,b,V}\cap \{p_1\ge t\}, \q t\in[a,b].
$$

    \begin{figure}[ht]
        \centering
        \begin{tikzpicture}
            \draw[->](-.5,0)--(4.5,0);
            \draw[->](0,-.5)--(0,3);
            \draw(0,0)node[below left]{$0$};
            \draw[domain=30:330,samples=100]plot({2+1.8*cos(\x)},{2.4+sin(\x)-exp(.2*cos(\x))});
            \draw({2+.9*sqrt(3)},{1.9-exp(.1*sqrt(3))})--({2+.9*sqrt(3)},{2.9-exp(.1*sqrt(3))});
            \draw(3.3,2.3)node{$\Delta$};
            \draw(.92,{1.6-exp(-.12)})--(.92,{3.2-exp(-.12)})node[midway,right]{$\cC(a)$}--({2+.9*sqrt(3)},{2.2-exp(-.12)})node[right]{$V$}--cycle;
            \draw[dashed](.92,0)node[below]{$a$}--(.92,{1.6-exp(-.12)});
            \draw[dashed]({1.406+.405*sqrt(3)},{1.87-exp(-.12)})--({1.406+.405*sqrt(3)},{2.75-exp(-.12)})node[midway,right]{$\cC(t)$};
            \draw[dashed]({1.406+.405*sqrt(3)},0)node[below]{$t$};
            \draw[dashed]({2+.9*sqrt(3)},0)node[below]{$\cS_0(v)$}--({2+.9*sqrt(3)},{1.9-exp(.1*sqrt(3))});
        \end{tikzpicture}
        \caption{The cones $\cC(t)=\cC^{a,\cS_0(v),V}_v(t)$ in the Okounkov body $\Delta$.}\label{Cone}
    \end{figure}

    \begin{lemma}\lb{TranslationLemma}
        For $t\in[a,b)$,
        $$
            \frac{b-a}{b-t}\cC_v^{a,b,V}(t)
        $$
        is a translation of $\cC_v^{a,b,V}$.
    \end{lemma}
    \begin{proof}
        The cones $\cC_v^{a,b,V}(t)$ all have the same tip $V$ and by construction
        are all similar. More precisely, 
        $$
            \cC_v^{a,b,V}(t)-V=\frac{b-t}{b-a}\left(\cC_v^{a,b,V}-V\right).
        $$
        I.e.,
        $\disp
            \frac{b-a}{b-t}\cC_v^{a,b,V}(t)=\cC_v^{a,b,V}+\frac{t-a}{b-t}V.
        $
    \end{proof}

    \bclaim\lb{CountingClaim}
For a convex body $K$, $\#(K\cap \ZZ^n/k)=\#(\frac k{\ell}K\cap \ZZ^n/\ell)$.
    \eclaim
\bpf
Observe that
$\#(K\cap \ZZ^n/k)=\#(\frac k{\ell}(K\cap \ZZ^n/k))=\#(\frac k{\ell}K\cap \frac k{\ell}(\ZZ^n/k))$.
\epf

\blem
\lb{cubetranslationlem}
Suppose that $t\in [a, b-\frac{\ell}k(b-a)]$. Then, 
$$
\#\left(\cC_v^{a,b,V}(t)\cap \ZZ^n/k\right)\ge 
\min_{x\in\left[-\frac{1}{2},\frac{1}{2}\right]^n}
\#\left((\cC_v^{a,b,V}+x)\cap \ZZ^n/\ell\right).
$$
\elem

\bpf
By inclusion, it suffices to prove the statement for $t=b-\frac{\ell}k(b-a)$.
By Claim \ref{CountingClaim}, the number of 
$\ZZ^n/k$ lattice
points of $\cC_v^{a,b,V}(t)$ coincides with the number
of $\ZZ^n/\ell$ lattice
points of $\frac k{\ell}\cC_v^{a,b,V}(t)$.
In turn, Lemma \ref{TranslationLemma} identifies the latter cone
with a translation of $\cC_v^{a,b,V}$. Thus,
$$
\#\left(\cC_v^{a,b,V}(t)\cap \ZZ^n/k\right)\ge 
\min_{x\in\RR^n}
\#\left((\cC_v^{a,b,V}+x)\cap \ZZ^n/\ell\right).
$$
The result follows since $\ZZ^n/\ell$ is invariant under translation by any integral vector.
\epf

\bprop\lb{firstconeprop}
Assume
that $\sigma(v)\le a<b\le\cS_0(v)$ and that if $a=\sigma(v)$ then 
$\dim \Delta\cap p_1^{-1}(a)=n-1$. Let $V\in \Delta\cap p_1^{-1}(b)$.
Consider a sequence $\{\ell_k\}$ with $\NN\ni\ell_k\le k$.
Suppose that $t\in [a, b-\frac{\ell_k}k(b-a)]$. Then for some $C_1,C_2>0$
independent of $k$ and $\ell_k$, 
$$
\#\left(\cC_v^{a,b,V}(t)\cap \ZZ^n/k\right)\ge C_2\ell_k^n, \q \hbox{\ for all $k\ge\ell_k>C_1$}.
$$
\eprop

\bpf
The result follows from Lemma \ref{cubetranslationlem} combined
with  Proposition \ref{lattice lower bound} applied to\newline $\cC_v^{a,b,V}+[-1/2,1/2]^n$.
\epf

More generally, an improved estimate for the slice-cones holds depending
on the dimension of the slices.

\bprop\lb{secondconeprop}
Assume
that $\sigma(v)\le a<b\le\cS_0(v)$ and that if $a=\sigma(v)$ then 
$\dim \Delta\cap p_1^{-1}(a)=n-1$. 
Consider a sequence $\{\ell_k\}$ with $\NN\ni\ell_k\le k$.
Suppose that $t\in [a, b-\frac{\ell_k}k(b-a)]$. Then for some $C_1,C_2>0$
independent of $k$ and $\ell_k$, 
$$
\#\left(\cC_v^{a,b}(t)\cap \ZZ^n/k\right)\ge C_2k^{\iota}\ell_k^{n-\iota},
\q \hbox{\ for all $k\ge\ell_k>C_1$},
$$
where
$
\iota:=\dim \Delta\cap p_1^{-1}(b).
$
\eprop
Note that by convexity of $\Delta$, $\iota=n-1$ for $b<\cS_0$, while
for $b=\cS_0$, $\iota$ can be any integer in the range $\{0,\ldots,n-1\}$.
Thus, Proposition \ref{firstconeprop} follows from
Proposition \ref{secondconeprop} by noting that since $\ell_k\le k$,
the estimate with $\iota=0$ always holds (as $\Delta\cap p_1^{-1}(\cS_0(v))\not=\emptyset$
as $\Delta$ is closed). The point of Proposition \ref{secondconeprop} is that the larger
the dimension of the slice above $\cS_0(v)$ the better the estimate. 

\subsection{Proof of Theorem \ref{maxp1Thm}}

\bcor
\lb{maxp1Cor}
Recall \eqref{p1plusbodyEq}.
Then, $
\#\big(\Delta^{\max_{\Delta_k} p_1+}\cap \ZZ^n/k\big)
\ge C (\max_{\Delta} p_1 - \max_{\Delta_k} p_1)^{n-\iota}k^n,
$
where $\iota:=\dim \Delta\cap p_1^{-1}(\max_\Delta p_1)$. 
\ecor

\bpf
The result follows from Proposition \ref{secondconeprop} for 
$a=(\sigma(v)+\cS_0(v))/2$,
$b=\cS_0(v)$, and 
$
\ell_k=\left\lceil 2k\frac
{\max_{\Delta}p_1-\max_{\Delta_k}p_1}
{\max_{\Delta}p_1-\min_{\Delta}p_1}
\right\rceil-1$. Indeed, 
for these choices, there is some $\eps_k>0$ such that
$\cC_v^{a,b}(t)\subset\Delta^{\max_{\Delta_k} p_1+\eps_k}$
for $t=b-\frac{\ell_k}k(b-a)$.
\epf

\bpf[Proof of Theorem \ref{maxp1Thm}]
By \eqref{TrivialLatticeCountEq}
 and Lemma \ref{DkdkLem},
$$
\#\big(\Delta^{\max_{\Delta_k} p_1+}\cap \ZZ^n/k\big)\le 
\#(\Delta\cap \ZZ^n/k)-\#(\Delta_k)=D_k-d_k\le C k^{n-1}.
$$
Theorem \ref{maxp1Thm} now follows from Corollary \ref{maxp1Cor}
by plugging-in the worst scenario $\iota=0$. Equivalently, 
Theorem \ref{maxp1Thm} follows from Proposition \ref{firstconeprop}
that treats only the case $\iota=0$, together with the same
argument as in the proof of Corollary \ref{maxp1Cor}.
\epf

\subsection{Joint asymptotics of \texorpdfstring{$S_{k,m}$: collapsing case}{S{k,mk}}}

In this section we prove Theorem \ref{SdoubleasymCor} in the collapsing case.

\begin{proof}[Proof of Theorem \ref{SdoubleasymCor} when $\tau=0$]
    By \eqref{first est} and the monotonicity 
    of the function $i$ (Definition \ref{upper S}),
    \begin{equation}\label{1st lower bd 0}
        S_{k,m_k}\left(v\right)\geq\frac{1}{km_k}\sum_{\ell=1+D_k-d_k}^{m_k+D_k-d_k}i_{k,\ell}\left(v\right)\geq\frac{1}{k}i_{k,m_k+D_k-d_k}\left(v\right).
    \end{equation}
    Recall Proposition \ref{firstconeprop} for definition of $C_1$ and $C_2$. 
    Since $\tau=0$ and using Lemma \ref{DkdkLem}, $m_k+D_k-d_k=o(k^n)$. Thus, there exists $C_3>0$ such that for $k\geq C_3$,
    $
        m_k+D_k-d_k\leq C_2k^n.
    $
    Consider an integer-valued function of $k$,
    \begin{equation}\label{k' def}
        \ell_k:=\max\left\{\left\lfloor C_1\right\rfloor+1,\left\lceil\left(\frac{m_k+D_k-d_k}{C_2}\right)^\frac{1}{n}\right\rceil\right\}.
    \end{equation}
    Note $\ell_k=o(k)$.
    When $k\geq\max\{\lfloor C_1\rfloor+1,C_3\}$, $k\geq \ell_k>C_1$. Let
    \begin{equation}\label{def t}
        {t}_{k,\ell_k}:=\cS_0\left(v\right)-\frac{\ell_k}{k}\left(\cS_0\left(v\right)-a\right),
    \end{equation}
    with $a$ and $V$ as in Proposition \ref{firstconeprop}, so that
    $$
        \#\left(\cC_{v}^{a,\cS_0(v),V}({t}_{k,\ell_k})\cap\ZZ^n/k\right)\ge
        C_2\ell_k^n\geq m_k+D_k-d_k.
    $$
Hence, by \eqref{bk def},
    \begin{align}
        \frac{1}{k}i_{k,m_k+D_k-d_k}\left(v\right)&=\max\left\{t\geq0\,:\,\#\left(\Delta\cap p_1^{-1}\left(\left[t,\cS_0\left(v\right)\right]\right)\cap\ZZ^n/k\right)\geq m_k+D_k-d_k\right\}\nonumber\\
        &\geq\max\left\{t\geq0\,:\,\#\left(\cC_{v}^{a,\cS_0(v),V}({t})\cap\ZZ^n/k\right)\geq m_k+D_k-d_k\right\}
        \geq{t}_{k,\ell_k}.\label{2nd lower bd 0}
    \end{align}
    Since by \eqref{k' def},
    \beq\lb{replaceellandnEq}
        \ell_k\leq C_4\left(m_k+D_k-d_k\right)^\frac{1}{n}\leq C_5\left(\max\left\{m_k,k^{n-1}\right\}\right)^\frac{1}{n},
    \eeq
    by \eqref{def t},
    \begin{equation}\label{3rd lower bd 0}
        {t}_{k,\ell_k}\geq\cS_0\left(v\right)\left(1-C_6\left(\max\left\{\frac{m_k}{k^n},\frac{1}{k}\right\}\right)^\frac{1}{n}\right)\geq\cS_0\left(v\right)\left(1-C_7\left(\max\left\{\frac{m_k}{d_k},\frac{1}{k}\right\}\right)^\frac{1}{n}\right).
    \end{equation}
    Equation \eqref{Case2} now follows from \eqref{1st lower bd 0}, \eqref{2nd lower bd 0}, and \eqref{3rd lower bd 0}.
\end{proof}
Using
Proposition \ref{secondconeprop}
instead of Proposition \ref{firstconeprop} in the proof above yields the improved estimate
$\ell_k\leq C_4\left(m_k+D_k-d_k\right)^\frac{1}{n-\iota}\leq C_5\left(\max\left\{m_k,k^{n-1}\right\}\right)^\frac{1}{n-\iota},
$
instead of 
\eqref{replaceellandnEq}. Consequently, 
Theorem \ref{SdoubleasymCor} in the case $\tau=0$ can be improved to
    $$
    S_{k,m_k}\left(v\right)\geq
        \left(1-C\left(\max\left\{\frac{m_k}{d_k},\frac{1}{k}\right\}\right)^\frac{1}{n-\iota}\right)\cS_\tau\left(v\right).
        $$

\subsection{Proof of Theorems \ref{MainAlphaThm}
and \ref{DeltatauThm}}
\lb{ProofofTwoThmSubSec}

We prove Theorem \ref{delta main} which (together with
Corollary \ref{BirkarCor}) implies Theorems \ref{MainAlphaThm} and \ref{DeltatauThm}.

\begin{definition}
\lb{computeddivisorialDef}
    Recall Corollary \ref{deltadoubleasymCor}. For a fixed $\tau\in[0,1]$, say $\bm{\delta}_\tau$ {\it is computed by a divisorial valuation} if the infimum in \eqref{delta_tau alg} is attained in $\ValXdiv$. In particular, $\alpha$ ($\delta$, respectively) is computed by a divisorial valuation if 
    for $\tau=0$ ($\tau=1$, respectively) the infimum in \eqref{delta_tau alg} is attained in $\ValXdiv$.
\end{definition}

Theorems \ref{MainAlphaThm} and \ref{DeltatauThm} are a consequence of
the following result.

\begin{theorem}\label{delta main}
    Suppose $m_k\in\{1,\ldots,d_k\}$ for $k\in\NN(L)$, and
    $
        \lim_{k\to\infty}\frac{m_k}{d_k}=\tau\in\left[0,1\right].
    $
    If $\bm{\delta}_\tau$ is computed by a divisorial valuation, then there exists $C>0$ depending on $L$, $\tau$, and $\{m_k\}$ (independent of $k$) such that
    $$
        \left(1-\frac{C}{k}\right)\min\left\{\frac{m_k}{\tau d_k},1\right\}\bm{\delta}_\tau\leq\delta_{k,m_k}\leq\begin{cases}
            \left(1+C\left(\max\left\{\frac{m_k}{d_k},\frac{1}{k}\right\}\right)^\frac{1}{n}\right)\bm{\delta}_\tau,&\tau=0;\\
            \left(1+\frac{C}{k}\right)\max\left\{\frac{m_k}{\tau d_k},1\right\}\bm{\delta}_\tau,&\tau\in\left(0,1\right].
        \end{cases}
    $$
    If $\alpha$ is computed by a divisorial valuation, 
    $
        \alpha_k=\alpha+O\left(k^{-\frac{1}{n}}\right).
    $
    If $\delta$ is computed by a divisorial valuation, 
    $
        \delta_k=\delta+O\left(k^{-1}\right).
    $
\end{theorem}
\begin{proof}
Suppose $v\in\ValXdiv$ computes $\bm{\delta}_\tau$. 
By Lemma \ref{deltakminfProp},
    $
        \delta_{k,m_k}\leq\frac{A\left(v\right)}{S_{k,m_k}\left(v\right)}=\frac{\cS_\tau\left(v\right)}{S_{k,m_k}\left(v\right)}\bm{\delta}_\tau.
    $
The theorem now follows from 
Corollary \ref{deltadoubleasymCor} and Theorem \ref{SdoubleasymCor}.
\end{proof}

\bigskip
\textsc{University of Maryland}\ \ \ 
{\tt cjin123@terpmail.umd.edu, yanir@alum.mit.edu}

\smallskip
\textsc{Peking University}\ \ \ 
{\tt gtian@math.pku.edu.cn}

\end{document}